\documentclass[12pt]{amsart}
\usepackage[mathcal]{eucal}
\usepackage{amsmath, amssymb, graphicx, hyperref, labelfig
}
\usepackage[all]{xy}

\newtheorem{thm}{Theorem}
\newtheorem{prop}[thm]{Proposition}
\newtheorem{lem}[thm]{Lemma}

\newtheorem{con}[thm]{Conjecture}
\newtheorem{claim}[thm]{Claim}

\theoremstyle{remark}
\newtheorem{rem}[thm]{Remark}

\theoremstyle{definition}

\newcommand{\C}{\mathbb C}
\newcommand{\R}{\mathbb R}
\newcommand{\Z}{\mathbb Z}

\newcommand{\HH}{\mathbb H}

\newcommand{\Id}{\mathrm{Id}}
\newcommand{\End}{\mathrm{End}}

\newcommand{\Diff}{\mathrm{Diff}}

\newcommand{\T}{\mathcal T}

\newcommand{\E}{\mathrm{e}}
\newcommand{\I}{\mathrm{i}}

\newcommand{\QDL}{{\mathrm{QDL}}}

\newcommand{\Tr}{\operatorname{\mathrm{Trace}}}
\newcommand{\vol}{\operatorname{\mathrm{vol}}}
\newcommand{\SL}{\mathrm{SL}_2(\C)}
\newcommand{\PSL}{\mathrm{PSL}_2(\C)}
\newcommand{\SSS}{\mathcal{K}^q}
\newcommand{\XX}{\mathcal X_{\SL}}
\newcommand{\XP}{\mathcal X_{\PSL}}
\newcommand{\ZZ}{\mathcal Z^{q^{\frac14}}}
\newcommand{\CP}{\mathbb{CP}}

\newcommand{\dbar}{/\negthinspace\negthinspace/}
\newcommand{\hyp}{\mathrm{hyp}}
\newcommand{\vbar}{\,|\,}

\renewcommand{\leq}{\leqslant}
\renewcommand{\geq}{\geqslant}
\renewcommand{\phi}{\varphi}
\renewcommand{\epsilon}{\varepsilon}

\renewcommand{\Im}{\operatorname{\mathfrak{Im}}}

\title
[Asymptotics of quantum invariants of surface diffeomorphisms I]
{Asymptotics of quantum invariants\\ of surface diffeomorphisms I:
\\
Conjecture and algebraic computations}

\author{Francis Bonahon}
\address {Department
of Mathematics,  University of
Southern California, Los Angeles
CA~90089-2532, U.S.A.}
\email{fbonahon@usc.edu}
\address {Department of Mathematics,  Michigan State University, 
East Lansing MI 48824, U.S.A.}
\email{bonahonf@msu.edu}
\urladdr{https://dornsife.usc.edu/francis-bonahon/}

\author{Helen Wong}
\address {Department of Mathematics,  Claremont McKenna College, 
Claremont CA 91711, U.S.A.}
\email{hwong@cmc.edu}
\urladdr{https://sites.google.com/view/helenwong/}

\author{Tian Yang}
\address {Department of Mathematics,  Texas A\&M University,
College Station TX 77843, U.S.A.}
\email{tianyang@math.tamu.edu}
\urladdr{https://www.math.tamu.edu/~tianyang/}

\date{\today}

\thanks{This work was partially supported by the grants DMS-1711297, DMS-2005656 (PI: Francis Bonahon), DMS-1841221, DMS-1906323 (PI: Helen Wong) and DMS-1812008 (PI: Tian Yang) from the US National Science Foundation.}

\begin{document}
\begin{abstract}
The Kashaev-Murakami-Murakami Volume Conjecture connects the hyperbolic volume of a knot complement to the asymptotics of certain evaluations of the colored Jones polynomials of the knot. We introduce a closely related conjecture for diffeomorphisms of surfaces, backed up by numerical evidence. The conjecture involves isomorphisms between certain representations of the Kauffman bracket skein algebra of the surface, and the bulk of the article is devoted to the development of explicit methods to compute these isomorphisms. These combinatorial and algebraic techniques are exploited in two subsequent articles, which prove the conjecture for a large family of diffeomorphisms of the one-puncture torus and are much more analytic. 
\end{abstract}

\maketitle

\tableofcontents

\section*{Introduction}

This work is motivated by the Kashaev Volume Conjecture \cite{Kash3}, as rephrased by Murakami-Murakami \cite{MurMur}, which connects the asymptotics of the $n$--th colored Jones polynomial $J_K^{(n)}(q) \in \Z[q^{\pm1}]$ of a knot $K\subset S^3$ to the volume $\vol_\hyp(S^3-K)$ of the complete hyperbolic metric of its complement (if it exists). More precisely, the conjecture is that
$$
\lim_{n\to \infty} \frac 1n \log \left\lvert J_K^{(n)}\big( \E^{\frac{2\pi\I}n} \big) \right\rvert = \frac1{2\pi} \vol_\hyp(S^3-K). 
$$
This attractive conjecture, combining two very different areas of low-dimensional topology and geometry, has generated much work in the past twenty years, much of it on the combinatorial side. 

The current series of three articles, consisting of this one and of its companions \cite{BWY2, BWY3}, aims at developing analytic and geometric tools to attack this Kashaev-Murakami-Murakami Conjecture. With this goal in mind, it introduces another conjecture, which has the deceptive appearance of being only 2--dimensional, but similarly relates the asymptotics of purely combinatorial invariants to 3--dimensional hyperbolic volumes. The reader familiar with Kashaev's original approach \cite{Kash1, Kash2}, as rigorously justified by Baseilhac-Benedetti \cite{BasBen1, BasBen2, BasBen3}, may  recognize the filiation of our conjecture. The conjecture is proved for a large family of diffeomorphisms of the one-puncture torus in \cite{BWY3}

This new conjecture involves relatively recent work \cite{BonWon3, FroKanLe1, GanJorSaf} on the representation theory of the Kauffman bracket skein algebra $\SSS(S)$ of an oriented surface $S$. This algebra $\SSS(S)$ was introduced \cite{Tur2, PrzS, BFK2} as a quantization of the character variety $\XX(S)$ formed by the characters of group homomorphisms $\pi_1(S) \to \SL$ or, equivalently, consisting of flat $\SL$--bundles over $S$. From a physical point of view, a quantization of $\XX(S)$ is actually a  \emph{representation} of $\SSS(S)$, usually over a Hilbert space. When the quantum parameter $q = \E^{2\pi \hbar}$ is a root of unity, $\SSS(S)$ turns out to have a rich finite-dimensional representation theory. In particular, the results of \cite{BonWon3, FroKanLe1, GanJorSaf} essentially establish a one-to-one correspondence between irreducible finite-dimensional representations of $\SSS(S)$ on the one hand, and on the other hand classical points in the character variety $\XX(S)$ endowed with the data, at each puncture $v$ of the surface, of a scalar weight $p_v\in\C$ which can only take finitely many values; see \S \ref{subsect:KauffmanReps} for precise (and more accurate) statements. 

If we are given an orientation-preserving diffeomorphism $\phi \colon S \to S$, it acts on the character variety $\XX(S)$ and on the skein algebra $\SSS(S)$.
This action usually has many fixed points, corresponding to points of the character variety $\XX(M_{\phi, r})$ of the mapping torus $M_{\phi, r}$. Indeed, recall that the \emph{mapping torus} $M_{\phi, r}$ of $\phi \colon S \to S$ is obtained from $S\times [0,1]$ by gluing each point $(x,1)$ to $(\phi(x),0)$. An easy property (see \S \ref{subsubsect:InvariantCharacters}) is that an irreducible character $[r]\in \XX(S)$ is $\phi$--invariant if and only if it extends to a character $[\hat r] \in \XX(M_{\phi, r})$. In particular, when $\phi$ is pseudo-Anosov, the monodromy $\pi_1(M_{\phi, r}) \to \PSL$ of the complete hyperbolic metric of $M_{\phi, r}$ \cite{ThuBAMS, Otal, Otal2} provides several $\phi$--invariant characters $[r_\hyp] \in \XX(S)$ (differing by elements of $H^1(M_{\phi, r}; \Z/2)$). The points of  $\XX(M_{\phi, r})$ that are near such a hyperbolic character $[r_\hyp]$ restrict to a complex $c$--dimensional family of other $\phi$--invariant characters $[r] \in \XX(S)$, where $c$ is the number of orbits of the action of $\phi$ on the set of punctures of $S$. See the discussion in \S \ref{subsubsect:InvariantCharacters}.

If we are given a $\phi$--invariant character $[r] \in \XX(S)$ and  $\phi$--invariant punctured weights $p_v$ that are compatible with $[r]$, the classification mentioned above provides an irreducible representation $\rho \colon \SSS(S) \to \End(V)$ that is invariant under the action of $\phi$ \emph{up to isomorphism}. We focus attention on the corresponding isomorphism $\Lambda^q_{\phi, r} \colon V\to V$, normalized so that $\left| \det \Lambda^q_{\phi, r} \right| =1$. Then, the modulus $\left| \Tr \Lambda^q_{\phi, r}  \right| \in \R$ depends only on the diffeomorphism $\phi$, the $\phi$--invariant character $[r] \in \XX(S)$, the root of unity $q$ and the $\phi$--invariant puncture weights $p_v$ (see Proposition~\ref{prop:KauffmanIntertwinerDefined}).

\begin{con}
\label{con:VolConfIntro}
For a pseudo-Anosov diffeomorphism $\phi \colon S \to S$, choose a $\phi$--invariant character $[r] \in \XX(S)$ that, in the fixed point set of the action of $\phi$ on $\XX(S)$,  is in the same component as a hyperbolic character $[r_\hyp]$. For every odd integer $n$, let the quantum parameter be $q=\E^{\frac{2\pi \I}n}$ and let the $\phi$--invariant puncture weights $p_v$ be consistently chosen in terms of $n$, in a sense precisely defined in \S {\upshape\ref{subsect:VolConjSurfaceDiffeos}}. Then, for the above isomorphism  $\Lambda^q_{\phi, r} \colon V\to V$,
$$
\lim_{n\to \infty} \frac1n \log \left| \Tr \Lambda^q_{\phi, r}  \right| = \frac1{4\pi} \vol_\hyp (M_{\phi, r})
$$
where $\vol_\hyp (M_{\phi, r})$ is the volume of the complete hyperbolic metric of the mapping torus $M_{\phi, r}$. 
\end{con}

In particular, the predicted limit is independent of the $\phi$--invariant character $[r]\in \XX(S)$ and puncture weights $p_v \in \C$. However, the numerical evidence of \S \ref{sect:Experimental} (and results of \cite{BWY2, BWY3}) shows that these impact the mode of convergence. 

The hypothesis that $[r]$ is in the same component of the fixed point set of $\phi$ as a hyperbolic character $[r_\hyp]$ is often unnecessary. However, it is required by surprising combinatorial cancellations, discovered in \cite{BWY3}, that can occur  for very specific diffeomorphisms $\phi$. 

The current article is devoted to the algebraic and geometric properties underlying this conjecture, while the subsequent papers \cite{BWY2, BWY3} are much more analytic. 

First, we carefully set up the conjecture in \S \ref{sect:VolConf}.

We then offer some quick numerical evidence for this conjecture in \S \ref{sect:Experimental}, at least when $S$ is the one-puncture torus. 

The bulk of the article,  in \S \ref{sect:ComputeIntertwinerWithCheFock}, is motivated by the fact that the results of  \cite{BonWon3, FroKanLe1, GanJorSaf} are rather abstract, and do not lend themselves well to explicit computations. This can be traced back to the fact that, although the quantization of the character variety $\XX(S)$ by the Kauffman bracket skein algebra $\SSS(S)$ is very intrinsic, it is often hard to work with in practice. When the surface $S$ has at least one puncture, there is a related quantization of $\XX(S)$ provided by the quantum Teichm\"uller space of Chekhov-Fock \cite{CheFoc1, CheFoc2, Liu, BonLiu}. It is much less intrinsic, but it has the great advantages that it is relatively explicit, that it  therefore lends itself better to computations, and that it is also closely related to 3--dimensional hyperbolic geometry. Theorem~\ref{thm:CheFockAndKauffmanIntertwiners} connects the intertwining isomorphism $\Lambda_{\phi, r}^q$ to a similar intertwiner introduced in \cite{BonLiu} in the representation theory of the quantum Teichm\"uller space, which can explicitly be computed, at least in theory.

In \S \ref{sect:IntertwinersPunctTorus}, we fully implement these computations for the one-punctured torus. The corresponding results, and their connection to the geometry of the mapping torus $M_{\phi, r}$, end up playing a critical role in  \cite{BWY2, BWY3}. These computations will enable us to prove Conjecture~\ref{con:VolConfIntro} for one very  specific example in \cite{BWY2}, and for many more diffeomorphisms of the one-puncture torus in \cite{BWY3}. 

In the last section \S \ref{sect:ComputeIntertwinerGeneralSurf}, we briefly show how to carry out a similar program for all surfaces (at the expense of an increased computational complexity).

\section{The Volume Conjecture for surface diffeomorphisms}
\label{sect:VolConf}

\subsection{The $\SL$--character variety and the Kauffman bracket skein algebra of a surface}
\label{subsect:KauffmanReps}
Many notions in quantum algebra and quantum topology are noncommutative deformations of a (commutative) algebra of functions over a geometric object, and depend on a parameter $q = \E^{2\pi \I \hbar}$. For instance, the \emph{Kauffman bracket skein algebra} $\SSS(S)$ of an oriented surface $S$ is a deformation of the algebra of regular functions over the $\SL$--character variety
$$
\XX(S) = \{ r \colon \pi_1(S) \to \SL \} \dbar \SL
$$
consisting of group homomorphisms from the fundamental group $ \pi_1(S) $ to the algebraic group $ \SL$ considered (in the sense of geometric invariant theory) up to conjugation by elements of $\SL$ \cite{Tur2, PrzS, BFK2}.

A general phenomenon is that, when the quantum parameter $q$ is a root of unity, a point in the geometric object usually determines an irreducible finite-dimensional representation of the associated quantum object, up to finitely many well-understood choices. As a consequence, an ``interesting'' geometric situation determines a finite but high-dimensional representation of an algebraic object, which carries a lot of information and from which invariants can be extracted. This principle was explicitly stated for the quantum Teichm\"uller space of a surface in \cite{BonLiu}, but it also occurs in many other contexts such as quantum cluster algebras \cite{BerZel, FocGon1}, quantum cluster ensembles \cite{FocGon2} or quantum character varieties \cite{GanJorSaf}. 

We will here restrict attention to the case of the skein algebra $\SSS(S)$ of an oriented surface $S$ of finite topological type, and rely on the results of \cite{BonWon2, FroKanLe1, GanJorSaf}. The precise definition of the Kauffman bracket skein algebra $\SSS(S)$ will not be important for our purposes. We will just say that it involves the consideration of framed links in the 3--dimensional thickening $S \times [0,1]$ of the surface $S$, considered modulo certain relations, the most important of which is the \emph{Kauffman bracket skein relation} that
$$
K_1 = q^{\frac12} K_0 + q^{-\frac12} K_\infty
$$
whenever the three links $K_1$, $K_0$ and $K_\infty\subset S\times [0,1]$ differ only in a little ball where they are as represented on Figure~\ref{fig:SkeinRelation}. 
In particular, although this is not reflected in the notation, the skein algebra $\SSS(S)$ depends on the choice of a square root $q^{\frac12}$ for the quantum parameter $q$. 

\begin{figure}[htbp]
\SetLabels
( .5 * -.4 ) $K_0$ \\
( .1 * -.4 )  $K_1$\\
(  .9*  -.4) $K_\infty$ \\
\endSetLabels
\centerline{\AffixLabels{\includegraphics{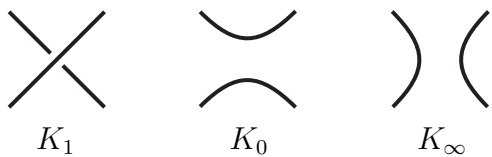}}}
\vskip 15pt

\caption{The Kauffman bracket skein relation.}
\label{fig:SkeinRelation}
\end{figure}

We consider representations of the skein algebra $\SSS(S)$, namely algebra homomorphisms $\rho \colon \SSS(S) \to \End(V)$ from $\SSS(S)$ to the algebra of linear endomorphisms of a finite dimensional vector space $V$ over $\C$.

\begin{thm}
 [\cite{BonWon3, BonWon4, BonWon5}]
 \label{thm:ClassicalShadow}
 Let the quantum parameter $q$ be a primitive $n$--root of unity with $n$ odd, and let the square root $q^{\frac12}$ occurring in the definition of $\SSS(S)$ be chosen so that $\big( q^{\frac12} \big)^n = -1$.  Then, an irreducible representation $\rho \colon \SSS(S) \to \End(V)$ uniquely determines
\begin{enumerate}
 \item a  character $[r_\rho] \in \XX(S)$, represented by a group homomorphism $r_\rho \colon \pi_1(S) \to \SL$; 
 \item a weight $p_v \in \C$ associated to each puncture $v$ of $S$ such that, if  $T_n(X)\in \Z[X]$ denotes the $n$--th Chebyshev polynomial of the first type, $T_n(p_v) = - \Tr r_\rho (\alpha_v)$ when $\alpha_v\in \pi_1(S)$ is represented by a small loop going once around $v$. 
\end{enumerate}

Conversely, every data of a character $[r] \in \XX(S)$ and of puncture weights  $p_v \in \C$  satisfying the above condition is realized by an irreducible representation $\rho \colon \SSS(S) \to \End(V)$.
\end{thm}

See \cite{Le} for versions of Theorem~\ref{thm:ClassicalShadow} where $n$ is allowed to be even. 

\begin{thm}
[\cite{FroKanLe1, GanJorSaf, Fro}]
\label{thm:SkeinRepUnique}
Suppose that $[r] $ is in the smooth part of $ \XX(S)$ or, equivalently, that it is realized by an irreducible homomorphism $r \colon \pi_1(S) \to \SL$. Then the irreducible  representation  $\rho \colon \SSS(S) \to \End(V)$ whose existence is asserted by the second part of Theorem~{\upshape\ref{thm:ClassicalShadow}} is unique up to isomorphism of representations.  This representation has dimension $\dim V=n^{3g+p-3}$ if $S$ has genus $g$ and $p$ punctures. 
\end{thm}

The article \cite{FroKanLe1}  proves uniqueness for generic $[r] \in \XX(S)$, and this property is improved in \cite{GanJorSaf, Fro} to include all irreducible  characters. The weaker uniqueness property of \cite{BonWon5}, which associates a unique representation of $\SSS(S)$ to each irreducible $[r] \in \XX(S)$ and compatible puncture weights (without proving irreducibility, or excluding the existence of other representations), would also be sufficient for our purposes. 

\subsection{Kauffman bracket intertwiners as invariants of surface diffeomorphisms}
\label{subsect:KauffmanIntertwinerInvariantDiffeomorphism}

The combination of Theorems~\ref{thm:ClassicalShadow} and \ref{thm:SkeinRepUnique} shows that a point $[r] $ of the character variety $ \XX(S)$ determines, up to $n^p$ possible choices of puncture weights,  a representation $\rho \colon \SSS(S) \to \End(V)$ of dimension $n^{3g+p-3}$, at least as long as $[r]$ does not belong to the ``bad'' algebraic subset consisting of reducible characters. We will apply this setup to very specific characters $[r]\in \XX(S)$. 

\subsubsection{Characters that are invariant under the action of a diffeomorphism}
\label{subsubsect:InvariantCharacters}

Consider an orientation-preserving  diffeomorphism $\phi \colon S \to S$. It induces a homomorphism $\phi_* \colon \pi_1(S) \to \pi_1(S)$, well-defined up to conjugation by an element of $\pi_1(S)$. It therefore acts on the character variety by $\phi^* \colon \XX(S) \to \XX(S)$, defined by the property that $\phi^* \big( [r] \big) =[r\circ \phi_*]$.

 We are interested in the fixed points of the action of $\phi$ on the smooth part of the character variety $\XX(S)$. These are the characters represented by irreducible homomorphisms $r\colon \pi_1(S) \to \SL$ such that $r \circ \phi_*$ is conjugate to $r$ by an element of $\SL$. We can therefore express this in terms of   the \emph{mapping torus} $M_{\phi, r}$, obtained from $S \times [0,1]$ by gluing $S\times\{1\}$ to $S\times \{0\}$ through $\phi$, and whose fundamental group admits the presentation
$$\pi_1(M_{\phi, r}) = \langle \pi_1(S), t; \,t \alpha t^{-1} = \phi_*(\alpha), \forall \alpha \in \pi_1(S)  \rangle.
$$
Then the irreducible character $[r] \in \XX(S)$ is fixed by the action of $\phi$ precisely when $r$ extends to a homomorphism $r \colon \pi_1(M_{\phi, r}) \to \SL$. 

When $\phi$ is a pseudo-Anosov diffeomorphism, there is a preferred finite family of such $\phi$--invariant characters. Indeed, the mapping torus $M_{\phi, r}$ then admits a unique complete hyperbolic metric \cite{ThuBAMS, Otal, Otal2}. Identifying the group of orientation-preserving isometries of the hyperbolic space $\HH^3$ to $\PSL$, the monodromy of this hyperbolic metric provides a homomorphism $\bar r_\hyp \colon \pi_1(M_{\phi, r}) \to \PSL$, well-defined up to conjugation by an element of $\PSL$, such that $M_{\phi, r}$ is isometric to $\HH^3/ r_\hyp\big(\pi_1(M_{\phi, r})\big)$ by an orientation-preserving isometry.

As an orientable 3--dimensional manifold, $M_{\phi, r}$ is parallelizable and we can use this property to lift $\bar r_\hyp \colon \pi_1(M_{\phi, r}) \to \PSL$ to a homomorphism $ r_\hyp \colon \pi_1(M_{\phi, r}) \to \SL$; the number of such lifts is equal to the cardinal of $H^1(M_{\phi, r}; \Z/2)$. Restricting these lifts $ r_\hyp$ to $\pi_1(S) \subset \pi_1(M_{\phi, r})$, the corresponding characters $[ r_\hyp] \in \XX(S)$ are by construction fixed by the action of $\phi$. The number of the $\phi$--invariant hyperbolic characters $[ r_\hyp] \in \XX(S)$ thus associated to the monodromy $\bar r_\hyp$ is equal to the cardinal of the kernel of the subtraction $H^1(\phi; \Z/2) - \Id_{ H^1(S; \Z/2)}$, where $H^1(\phi;\Z/2) \colon H^1(S; \Z/2) \to  H^1(S; \Z/2)$ denotes the homomorphism induced by $\phi$. 

However, there are many more fixed points, in particular when $S$ has at least one puncture. For instance, near the hyperbolic character $[\bar r_\hyp]$,  the $\PSL$--character variety $\XP(S)$ is smooth with complex dimension  equal to the number $c$ of cusps of $M_{\phi, r}$, by Weil rigidity \cite{Weil1, Weil2, GarRag} applied to $M_{\phi, r}$; see also \cite[\S 5]{ThuNotes}, \cite[\S E.6]{BenPet} or the discussion in \S \ref{subsect:SweepInvariantChar}. Note that the number $c$ of topological ends of $M_{\phi, r}$ is also the number of orbits of the action of $\phi$ on the set of punctures of the surface $S$. Lifting these characters to $\SL$ then provides a complex $c$--dimensional submanifold of $\phi$--invariant characters $[ r] \in \XX(S)$ near each $\phi$--invariant lift  $[ r_\hyp] \in \XX(S)$ of the hyperbolic character  $[ \bar r_\hyp] \in \XP(S)$. 

\subsubsection{Invariants of surface diffeomorphisms}

The diffeomorphism $\phi$ also acts on the skein algebra by $\phi_* \colon \SSS(S) \to \SSS(S)$ in such a way that $\phi_* \big( [K] \big) =\big[(\phi\times\Id_{[0,1]})(K)\big]$ for every element $[K]\in \SSS(S)$ represented by a framed link $K\subset S \times [0,1]$.

If we are given an irreducible $\phi$--invariant character $[r]\in \XX(S)$, and if we choose puncture weights $p_v \in \C$ that are \emph{$\phi$--invariant}  in the sense that $p_{\phi(v)}=p_v$ for every puncture $v$, the uniqueness property of Theorem~\ref{thm:SkeinRepUnique} shows that the  representation $\rho \colon \SSS(S) \to \End(V)$ associated to this data by Theorem~\ref{thm:ClassicalShadow} is isomorphic to the representation $\rho \circ \phi_* \colon \SSS(S) \to \End(V)$.  This means that there exists a linear isomorphism $\Lambda_{\phi, r}^q  \colon V \to V$ such that
$$
(\rho \circ \phi_*)(X) = \Lambda_{\phi, r}^q \circ \rho(X)\circ (\Lambda_{\phi, r}^q)^{-1} \in \End(V)
$$
for every $X \in \SSS(S)$. 

Note that this property is unchanged if we replace  the intertwiner $\Lambda_{\phi, r}^q$ by a scalar multiple. Lacking a better idea, we normalize it so that $ \left| \det \Lambda_{\phi, r}^q \right|=1$. 

\begin{prop}
 \label{prop:KauffmanIntertwinerDefined}
 Let  $\Lambda_{\phi, r}^q  \colon V \to V$ be the above intertwiner, normalized so that $ \left| \det \Lambda_{\phi, r}^q \right|=1$. Then, up to conjugation and  multiplication by a scalar with modulus $1$,   $\Lambda_{\phi, r}^q $ depends only on the diffeomorphism $\phi\colon S \to S $, the $\phi$--invariant  character $[r] \in \XX(S)$, the primitive $n$--root of unity $q$ and the $\phi$--invariant puncture weights $p_v$.
 
 In particular,  the modulus $| \Tr \Lambda_{\phi, r}^q |$  of its trace is uniquely determined by the above data. 
\end{prop}

To be specific the uniqueness statement means that, if different intermediate choices lead to another intertwiner $\Lambda_{\phi, r}^{q\, \prime}  \colon V' \to V'$, there exists a linear isomorphism $\Lambda \colon V \to V'$ such that $\Lambda_{\phi, r}^{q\, \prime} = s\, \Lambda \circ \Lambda_{\phi, r}^q \circ \Lambda^{-1}$ for some scalar $s\in \C$ with $|s|=1$. 

\begin{proof} We made two implicit choices: a representative $\rho\colon \SSS(S) \to \End(V)$ in an isomorphism class of representations; and the intertwiner $\Lambda_{\phi, r}^q$. 

If $\rho$ is given, the irreducibility property of this representation implies that the  intertwiner $\Lambda_{\phi, r}^q$ is unique up to multiplication by a nonzero scalar $s \in \C^*$, by Schur's lemma. Our hypothesis that $ \left| \det \Lambda_{\phi, r}^q \right|=1$  then constrains this scalar $s$ to have modulus $1$. 

Replacing $\rho$ by an isomorphic representation $\rho'\colon \SSS(S) \to \End(V')$ will only replace $\Lambda_{\phi, r}^q$ by its conjugate $\Lambda \circ \Lambda_{\phi, r}^q \circ \Lambda^{-1}$, where $\Lambda\colon V\to V'$ is the isomorphism between $\rho$ and $\rho'$. 

This proves the property of uniqueness up to conjugation and  multiplication by scalar with modulus 1. The uniqueness property for $| \Tr \Lambda_{\phi, r}^q |$ immediately follows. 
\end{proof}

\begin{rem}
 We could have required that $ \det \Lambda_{\phi, r}^q $ is exactly equal to 1, in which case the above argument shows that $\Lambda_{\phi, r}^q $ is unique up to conjugation and multiplication by a $d$--root of unity, where $d=n^{3g+p-3}$ is the dimension of $V$. However, we do not know of any good use for this more precise statement at this point. 
\end{rem}

\subsection{The Volume Conjecture for surface diffeomorphisms}
\label{subsect:VolConjSurfaceDiffeos}

The modulus $\left| \Tr \Lambda_{\phi, r}^q \right|$ of Proposition~\ref{prop:KauffmanIntertwinerDefined} depends on the odd integer $n$, the primitive $n$--root of unity $q$, the $\phi$--invariant character $[r] \in \XX(S)$, and the $\phi$--invariant puncture weights $p_v \in \C^*$, constrained for every puncture $v$ by the Chebyshev condition that $T_n(p_v) = - \Tr r(\alpha_v)$ as in Theorem~\ref{thm:ClassicalShadow}. After fixing the diffeomorphism $\phi \colon S \to S$ and the $\phi$--invariant character $[r] \in \XX(S)$, we want to consider the asymptotic behavior of this quantity as $q$ tends to 1, namely as $n$ tends to $\infty$. However, we need to choose the other quantities in a consistent way as functions of $n$. 

For the quantum parameter $q$, we take it to be equal to $q_n = \E^{\frac{2\pi \I}n}$. 

For the puncture weights $p_v$, we need solutions of the equation $T_n(x)=  - \Tr r(\alpha_v)$. An elementary property of the Chebyshev polynomial $T_n$ (see for instance \cite[Lem. 17]{BonWon4}) is that, if we write $\Tr r(\alpha_v) =  -\E^{\theta_v} - \E^{-\theta_v}$ for some $\theta_v \in \C$, the solutions of the equation $T_n(x)=  - \Tr r(\alpha_v)$ are all numbers of the form $x = y +y^{-1}$ where $y$ is an $n$--root of $ \E^{\theta_v} $. In particular, once the numbers $\theta_v$ are chosen, we can take $p_v =  \E^{\frac1n\theta_v} + \E^{-\frac1n \theta_v}$. Since $[r]$ is $\phi$--invariant, we can choose the $\theta_v$ to be $\phi$--invariant in the sense that $\theta_{\phi(v)}=\theta_v$ for every puncture $v$; then the puncture weights $p_v$ will be $\phi$--invariant as well. 

\begin{con} [The Volume Conjecture for surface diffeomorphisms]
\label{con:MainConjecture}
Let $\phi \colon S \to S$ be a pseudo-Anosov diffeomorphism of the surface $S$, and let $[r]\in \XX(S)$ be a $\phi$--invariant character which, in the fixed point set of the action of $\phi$ on $\XX(S)$,  is in the same component as a hyperbolic character $[r_\hyp]$. Choose $\phi$--invariant puncture weights $\theta_v\in \C$ such that $\Tr r(\alpha_v)= -\E^{\theta_v} - \E^{-\theta_v}$ for every puncture $v$, where $\alpha(v) \in \pi_1(S)$ is represented by a loop going once around $v$. With this data, for every odd integer $n$, let $\Lambda_{\phi, r}^{q_n}$ be the intertwiner associated by Proposition~{\upshape\ref{prop:KauffmanIntertwinerDefined}} to the quantum parameter  $q = \E^{\frac{2\pi \I}n}$, the $\phi$--invariant character $[r]\in \XX(S)$ and the $\phi$--invariant puncture weights $p_v =  \E^{\frac1n\theta_v} + \E^{-\frac1n \theta_v}$. Then,
$$
\lim_{n\to \infty}  \frac1n \log \left | \Tr \Lambda_{\phi, r}^{q_n} \right| =\frac1{4\pi}  \vol_\hyp(M_\varphi)
$$
where $ \vol_\hyp(M_\varphi)$ is the volume of the complete hyperbolic metric of the mapping torus $M_{\phi, r}$ of $\phi$. 
\end{con}

Conjecture~\ref{con:MainConjecture} admits a straightforward generalization to the root of unity $q=\E^{\frac{2k\pi\I}n}$ for some fixed integer $k\geq 1$, in which case heuristic and experimental evidence suggest that the limit should be $\frac1{4k\pi}  \vol_\hyp(M_\varphi)$. The case $k=1$ is sufficiently complicated that we will restrict attention to this framework. 

Once the character $[r]\in \XX(S)$ is fixed, the puncture weights $\theta_v$ are determined only modulo $2\pi \I$ (and also up to a change of sign that does not impact the puncture weights $p_v$). Different choices yield different intertwiners $\Lambda_{\phi, r}^q$. The conjecture predicts the same limit for all choices but, as we will see in  \S \ref{sect:Experimental} and \cite{BWY2, BWY3}, these choices do impact the mode of convergence.

\section{Experimental evidence}
\label{sect:Experimental}

 In preliminary work to test Conjecture~\ref{con:MainConjecture}, we developed computer code (running on \emph{Mathematica}\texttrademark) that computes $ \Tr \Lambda_{\phi, r}^q$ for diffeomorphisms $\phi$ of the simplest surface where it applies, the one-puncture torus $S_{1,1}$. 

For the one-puncture torus $S_{1,1}$, the  mapping class group $\pi_0 \,\mathrm{Diff}^+(S_{1,1})$ is isomorphic to $\mathrm{SL}_2(\mathbb Z)$, and it is well-known that every orientation preserving element is conjugate to a composition 
$$
\phi = \pm \phi_1 \circ \phi_2 \circ \dots \circ \phi_{k_0}
$$
where each diffeomorphism $\phi_k$ corresponds to one of the matrices $L=
\left(\begin{smallmatrix}
 1&1\\0&1
\end{smallmatrix}\right)
$ and $R=
\left(\begin{smallmatrix}
 1&0\\1&1
\end{smallmatrix}\right)
$. The $\pm$ sign turns out to be irrelevant for our purposes, and the diffeomorphism $\phi$ is pseudo-Anosov precisely when both $L$ and $R$ occur in the list of the elementary diffeomorphisms~$\phi_k$. 

The simplest pseudo-Anosov diffeomorphism corresponds to $\phi = LR= \left(\begin{smallmatrix}
 2&1\\1&1
\end{smallmatrix}\right)$, in which case the mapping torus $M_{\phi, r}$ is diffeomorphic to the complement of the figure-eight knot. In this example, it turns out that the algebraic expression of the trace $\Tr \Lambda_{\phi, r}^q$ is relatively simple, as is the geometry of the hyperbolic metric of the mapping torus $M_{\phi, r}$. This enables us, in \cite{BWY2}, to prove Conjecture~\ref{con:MainConjecture} for this case ``by hand''.  

The next example by order of complexity is $\phi = LLR= \left(\begin{smallmatrix} 3&2\\1&1\end{smallmatrix}\right)$, for which there does not seem to exist any elementary proof. The two diagrams of Figure~\ref{fig:LLR} plot, for this diffeomorphism $\phi = LLR$,  the quantity  $\frac1n \log | \Tr \Lambda_{\phi, r}^{q_n} | $ for all odd $n\leq 301$. They are both associated to the same (arbitrary) $\phi$--invariant character $[r]\in \XX(S_{1,1})$,  which is \emph{not} a hyperbolic character $[r_\hyp]$ coming from the complete hyperbolic metric of the mapping torus $M_{\phi, r}$. However, they correspond to different choices of the puncture weight $\theta_v\in \C$ such that $\Tr r(\alpha_v)= \E^{\theta_v} + \E^{-\theta_v}$ when $\alpha(v) \in \pi_1(S)$ is represented by a loop going once around the puncture.

\begin{figure}[htbp]

\includegraphics[width=3in]{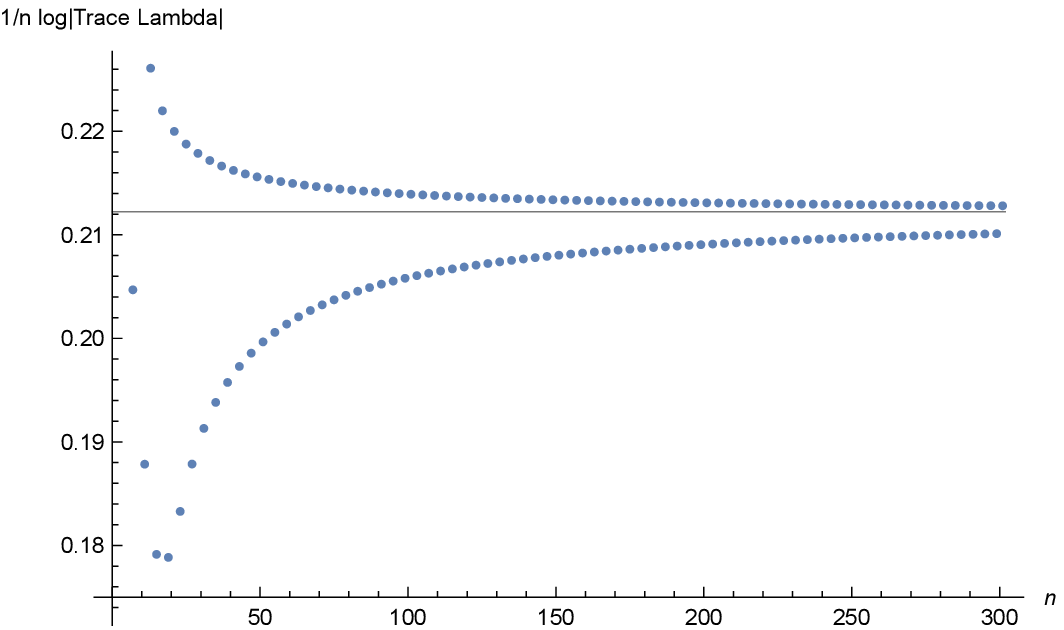} \hskip 20pt
\includegraphics[width=3in]{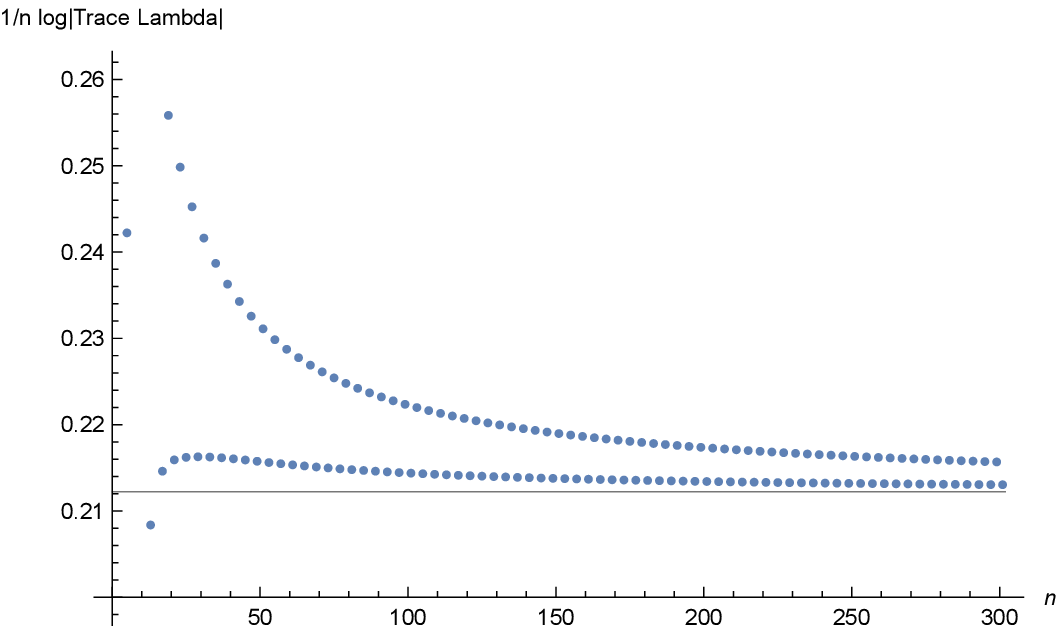}

\caption{Two plots of $\frac1n \log \left | \Tr \Lambda_{\phi, r}^{q_n} \right | $ for $\phi=LLR$, $q_n = \E^{\frac{2\pi\I}n}$, the same  $\phi$--invariant character $[r]\in \XX(S)$, but different puncture weights $\theta_v$.}
\label{fig:LLR}
\end{figure}

 In both cases, a clear bimodal pattern emerges, depending on the congruence of $n $ modulo $ 4$, although the type of the two modes is different in the two cases illustrated. If we use a curve fitting algorithm to approximate each mode with a curve of the form $a + \frac bn + \frac c{n^2} + \frac d{n^3} + \frac e{n^4}$, the output predicts a limit (= the asymptotic value $a$) approximately equal to 0.212213 for each mode, and for each of the two examples illustrated. It turns out that $0.212213$ is also the 6-digit approximation of  $\frac1{4\pi} \vol(M_{\phi, r}) $.
 
 Figure~\ref{fig:LLRother} displays a strikingly different behavior, in a slightly modified context. In the previous example, the data was actually computed by using another intertwiner $\bar \Lambda_{\phi, \bar r}^q$, associated to the projection $[\bar r] \in \XP(S_{1,1})$ and to a suitable puncture invariant, which turns out to be equal to $ \Lambda_{\phi,  r}^q$ (see \S \ref{sect:ComputeIntertwinerWithCheFock}). This intertwiner $\bar \Lambda_{\phi, \bar r}^q$ is defined for every $\phi$--invariant character $[\bar r] \in \XP(S_{1,1})$.  Figure~\ref{fig:LLRother} plots $\frac1n \log \left | \bar\Tr \Lambda_{ \phi, \bar r}^{q_n} \right | $ for $\phi=LLR$ and for a $\phi$--invariant character $[\bar r] \in \XP(S_{1,1})$ that does not lift to a $\phi$--invariant character $[r]\in \XX(S)$. Very clearly, the sequence $\frac1n \log \left | \Tr \bar \Lambda_{\phi, \bar r}^{q_n} \right | $ in this example does not exhibit the same smooth bimodal convergence as in the two cases of Figure~\ref{fig:LLR}.

 \begin{figure}[htbp]

\includegraphics[width=3in]{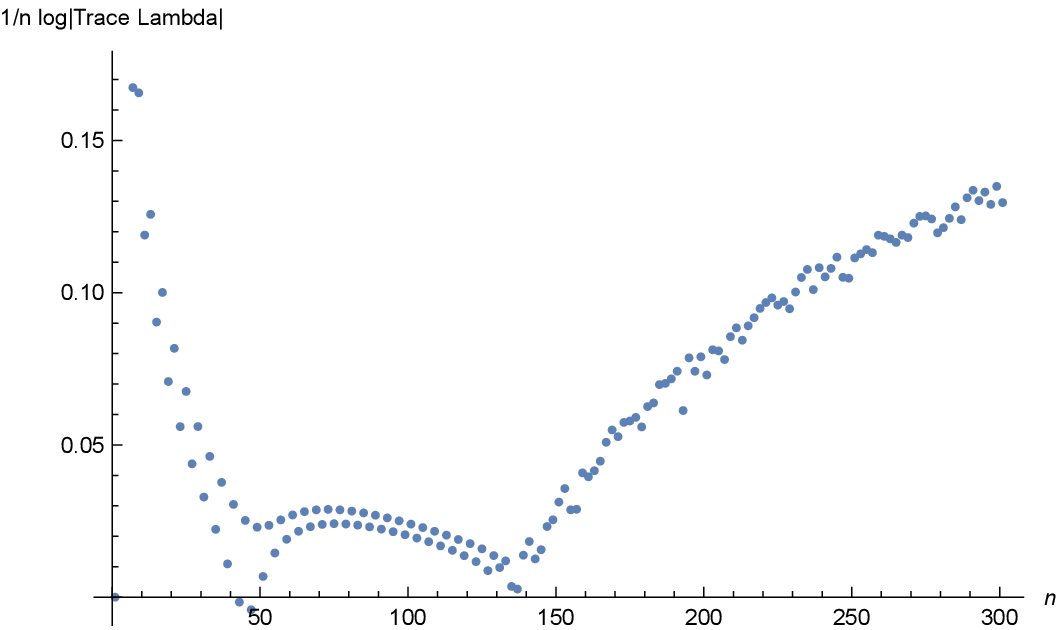} 

\caption{The plot of $\frac1n \log \left | \Tr \bar \Lambda_{ \phi, \bar r}^{q_n} \right | $ for $\phi=LLR$, $q_n = \E^{\frac{2\pi\I}n}$, and for another  $\phi$--invariant character $[\bar r]\in \XP(S)$.}
\label{fig:LLRother}
\end{figure}

 These examples are fully analyzed in \cite{BWY3}. In particular, the discrepancy between the two types of behavior is explained by different combinatorial topology properties of the corresponding invariant characters $[r]\in \XX(S)$, related to the correction factors and logarithm choices that we will encounter in \S \ref{subsect:CorrectionFactors}. In all cases, the asymptotic behavior of the sequence $\frac1n \log \left | \Tr \Lambda_{\phi, r}^{q_n} \right | $ is controlled by a finite sum of leading terms, all equal up to sign. What happens in the case of Figure~\ref{fig:LLRother} is that these leading terms all cancel out, whereas they add up to a nontrivial contribution in the cases of Figure~\ref{fig:LLR}. In the statement of Conjecture~\ref{con:MainConjecture}, the hypothesis that the $\phi$--invariant character $r\in \XX(S)$ is in the same component as a hyperbolic character $r_\hyp \in \XX(S)$ in the fixed point set of the action of $\phi$ is a simple way of excluding such full cancellations; in particular, it is not necessary for the conclusion of the conjecture to hold. 
 
 In the presence of full cancellations, as in the example of Figure~\ref{fig:LLRother}, we are not ready to make any prediction on whether there is still a limit for the sequence $\frac1n \log \left | \Tr \Lambda_{\phi, r}^{q_n} \right | $, or  what this limit might be if it does exist. 
 
 \section{Computing the intertwiner $\Lambda_{\phi, r}^q$ using ideal triangulations}
 \label{sect:ComputeIntertwinerWithCheFock}
 
 The proof of the uniqueness property of Theorem~\ref{thm:SkeinRepUnique} is rather abstract, and consequently so is the construction of the intertwiner $\Lambda_{\phi, r}^q$ of Proposition~\ref{prop:KauffmanIntertwinerDefined}. When the surface $S$ has at least one puncture, the explicit constructions of \cite{BonWon4} provide a more practical determination of $\Lambda_{\phi, r}^q$, by connecting it to the intertwiners of \cite[\S 9]{BonLiu}. This section describes this approach, which we will use in our explicit computations. The first two subsections \S \ref{subsect:EdgeWeightsGiveCharacters} and \S \ref{subsect:SweepInvariantChar} cover classical geometric material, but are needed to carefully set up the correspondence between the two points of view. They will also play an important role in the  geometric part of the arguments of \cite{BWY3}. 
 
 In the rest of the article, the surface $S$ will always be assumed to admit an ideal triangulation as defined below. This is equivalent to requiring that $S$ has at least one puncture, and even at least 3 punctures if its genus is 0. 
 
 \subsection{Ideal triangulations, complex edge weights  and $\PSL$--characters} 
 \label{subsect:EdgeWeightsGiveCharacters}
 
 If we represent the punctured surface $S = \widehat S - \{v_1, v_2, \dots, v_p \}$ as the complement of $p$ points in a compact surface $\widehat S$, an \emph{ideal triangulation} of $S$ is a triangulation of $\widehat S$ whose vertex set is equal to the set $\{v_1, v_2, \dots, v_p \}$ of the punctures. If $S$ has genus $g$ and $p$ punctures, such a triangulation necessarily has $6g+3p-6$ edges and $4g+2p-4$ triangular faces. In particular, the existence of an ideal triangulation requires, and is in fact equivalent to, the condition that $p\geq 3$ if $g=0$ (and $p>0$ in all cases). 
 
 We will insist that the data of an ideal triangulation $\tau$ includes, in addition to the triangulation itself, an indexing of its edges as $\gamma_1$, $\gamma_2$, \dots, $\gamma_e$ for $e=6g+3p-6$ (but no indexing of its triangle faces). We consider such ideal triangulations up to isotopy of $S$. 
  
 Given an ideal triangulation $\tau$, assigning a nonzero complex weight $a_i \in \C^*$ to each edge $\gamma_i$ of $\tau$ determines a character $[\bar r] \in \XP(S)$. We  give a quick outline of the construction, and refer to, for instance,  \cite[\S 8]{BonLiu} for details. 
 
 More precisely, the edge weights $a_i$ enable us to construct a pleated surface $\widetilde f \colon \widetilde S \to \HH^3$ from the universal cover $\widetilde S$ of $S$ to the hyperbolic space $\HH^3$, pleated along the triangulation $\widetilde \tau$ of $\widetilde S$ induced by $\tau$ and $\bar r$-equivariant for a unique group homomorphism $\bar r \colon \pi_1(M) \to \PSL$. By definition, this map is totally geodesic on the faces of $\widetilde \tau$, and its local geometry along an edge $\widetilde \gamma_i$  of $\widetilde \tau$ lifting an edge $\gamma_i$ of $\tau$ is determined by the following property: If we consider the four points of the boundary at infinity $\partial_\infty  \HH^3 = \CP^1$ that are the images under $\widetilde f$ of the vertices of the square formed by the two faces of $\widetilde \tau$ adjacent to $\widetilde \gamma_i$, then the crossratio of these four points is equal to $-a_i$. (Since the crossratio of four points depends on an ordering of these points, we refer to \cite[\S 8]{BonLiu} for details.) The $\bar r$--equivariance property means that, for the standard isometric action of $\PSL$ on the hyperbolic space $\HH^3$,  $\widetilde f(\alpha \widetilde x) = \bar r(\alpha) \widetilde f(\widetilde x)$ for every $\alpha \in \pi_1(S)$ and $\widetilde x \in \widetilde S$. In this construction, the edge weights $a_i \in \C^*$ are the \emph{shear-bend parameters} of the pleated surface $\widetilde f$. 
 
 This pleated surface $\widetilde f \colon \widetilde S \to \HH^3$ is uniquely determined by the ideal triangulation $\tau$ and the edge weights $a_i$, up to post-composition by an orientation-preserving isometry of $\HH^3$ and pre-composition by the lift to $\widetilde S$ of an isotopy of $S$. In particular, the homomorphism $\bar r \colon \pi_1(M) \to \PSL$ is uniquely determined up to conjugation by an element of $\PSL$. As a consequence, the corresponding character $[\bar r] \in \XP(S)$ is uniquely determined by the edge weights $a_i$. 
 
 We combine the weights $a_i \in \C^*$ into an \emph{edge weight system} $a = (a_1, a_2, \dots, a_e) \in \left( \C^* \right)^e$. 
 
 The following statement shows that ``most'' characters $[\bar r ] \in \XP(S)$ are associated in this way to an edge weight system $a \in \left( \C^* \right)^e$.
 
 \begin{lem}
 \label{lem:IdealTriangWeightsAndCharacters}
Given an ideal triangulation $\tau$, there exists a  Zariski-open dense subset $U \subset \XP(S)$ such that every character $[\bar r] \in U$ is associated as above  to an edge weight system $a\in\left( \C^* \right)^e$  for $\tau$.
\end{lem}

This is a well-known property, but we will need to refer to the details of its proof for the less familiar Lemma~\ref{lem:PeriodicEdgeWeightSystemsInvariantChar} below.

\begin{proof} The key idea is that, for a group homomorphism $\bar r\colon \pi_1(S) \to \PSL$, an $\bar r$--equivariant pleated surface $\widetilde f \colon \widetilde S \to \HH^3$ pleated along $\tau$ can be reconstructed from fixed point data of the action of $\bar r\big(\pi_1(S) \big) \subset \PSL$ on the complex projective line $\CP^1$. 

A puncture of the surface $S$ determines infinitely many \emph{peripheral subgroups} of $\pi_1(S)$, all conjugate to each other. These are each generated by an element of $\pi_1(S)$ represented by a small loop going once around the puncture. 

In particular, for every edge $\widetilde \gamma$ of the ideal triangulation $\widetilde \tau$ of $\widetilde S$ induced by $\tau$,  each end of $\widetilde\gamma$ determines a peripheral subgroup $\pi \subset \pi_1(S)$. We can then consider the image  $\widetilde f \left( \widetilde \gamma \right)$ under the pleated surface  $\widetilde f \colon \widetilde S \to \HH^3$. The endpoints of this geodesic $\partial_\infty \HH^3 = \CP^1$ are each fixed by  the image $\bar r(\pi) \subset \PSL$  of the peripheral subgroup $\pi \subset \pi_1(S)$ associated to the corresponding end of $\widetilde \gamma$. Because $\widetilde f$ sends every face of $\widetilde\tau$ to an ideal triangle in $\HH^3$ (such that any two sides of the triangle share an endpoint), this endpoint $\xi(\pi) \in \CP^1$ depends only on the peripheral subgroup $\pi$, not on the particular edge $\widetilde\gamma$ leading to it. 

The fundamental group $\pi_1(S)$ acts by conjugation on the set $\Pi$ of its peripheral subgroups. A \emph{$\tau$--enhancement} for a homomorphism $\bar r \colon \pi_1(S) \to \PSL$ is a map $\xi \colon \Pi \to \CP^1$ such that:
\begin{enumerate}
\item
 The map $\xi$ is $\bar r$--equivariant, in the sense that
 $$
 \xi\left(\alpha \pi \alpha^{-1} \right) = \bar r(\alpha) \xi(\pi)
 $$
 for every peripheral subgroup $\pi \in \Pi$ and for every $\alpha \in \pi_1(S)$. 
 \item  If $\pi$, $\pi'\in \Pi$ are the peripheral subgroups respectively associated to the two ends of an edge $\widetilde \gamma$ of the ideal triangulation $\widetilde\tau$ of $\widetilde S$ induced by $\tau$, then $\xi(\pi) \neq \xi(\pi')$ in $\CP^1$. 
\end{enumerate}

We just saw that an $\bar r$--equivariant pleated surface $\widetilde f \colon \widetilde S \to \HH^3$ pleated along $\tau$ uniquely determines a $\tau$--enhancement. Conversely, given a $\tau$--enhancement $\xi$, one can use the second condition in the definition of enhancements to reconstruct an $\bar r$--equivariant pleated surface $\widetilde f$ realizing it, by considering the geodesics and ideal triangles of $\HH^3$ that will form the images under $\widetilde f$ of the edges and faces of the ideal triangulation $\widetilde \tau$. The shear-bend parameters of this pleated surface then form an edge weight system $a \in \left( \C^* \right)^e$, associated to the character $[\bar r]\in \XP(S)$ by construction. 

We consequently have converted the problem of finding an edge weight system $a \in \left( \C^* \right)^e$ associated to $[\bar r]$ to finding a $\tau$--enhancement for $\bar r$. 

With this in mind, we consider a homomorphism $\bar r\colon \pi_1(S) \to \PSL$ and look for a $\tau$--enhancement for $\bar r$. This will require us to impose some restrictions on the character $[\bar r] \in \XP(S)$. 

Since every element of $\PSL$ has at least one fixed point in $\CP^1$, the image $\bar r(\pi) \in \PSL$ of each peripheral subgroup $\pi \in \Pi$ admits a fixed point $\xi(\pi) \in \CP^1$, and we can easily turn this to an $\bar r$--equivariant map $\xi \colon \Pi \to \CP^1$ by considering the finitely many orbits of the action of $\pi_1(S)$ on $\Pi$. 

If we want $\xi$ to be a $\tau$--enhancement for $\bar r$, we need to make sure that $\xi(\pi) \neq \xi(\pi')$ whenever $\pi$ and $\pi' \in \Pi$ are associated to the endpoints of an edge of the ideal triangulation $\widetilde \tau$. One easy way to guarantee this is to arrange that the cyclic subgroups $\bar r(\pi)$ and $\bar r(\pi')\subset \PSL$ have no common fixed points, and to take advantage of   the following elementary property: Two elements $A$, $B\in \PSL$ have a common fixed point in $\CP^1$ if and only if $\Tr \left(ABA^{-1}B^{-1} \right) = +2$. (Note that, although the trace of an element of $\PSL$ is only defined up to sign, the trace of a commutator is uniquely determined.)

For this purpose, pick a generator $\alpha_\pi \in \pi $ for each peripheral subgroup $\pi \in \Pi$, and do this in a $\pi_1(S)$--equivariant way in the sense that $\alpha_{\beta \pi \beta^{-1}} = \beta \alpha_{\pi} \beta^{-1}$ for every $\beta \in \pi_1(S)$ and $\pi \in \Pi$. Then define $U \subset \XP(S)$ as the set of characters $[\bar r]$ such that $\Tr \bar r\left( \alpha_\pi \alpha_{\pi'} \alpha_\pi^{-1} \alpha_{\pi'}^{-1} \right ) \neq +2$ whenever $\pi$ and $\pi' \in \Pi$ are associated to the endpoints of an edge of the ideal triangulation $\widetilde \tau$.

By $\bar r$--equivariance and because $\tau$ only has finitely many edges, there are only finitely many such conditions $\Tr \bar r\left( \alpha_\pi \alpha_{\pi'} \alpha_\pi^{-1} \alpha_{\pi'}^{-1} \right ) \neq +2$ that are relevant. It follows that $U$ is Zariski-open. 

For any two distinct peripheral subgroups $\pi$, $\pi' \in \Pi$, the set of group homomorphisms $\bar r\colon \pi_1(S) \to \PSL$ such that $\Tr \bar r\left( \alpha_\pi \alpha_{\pi'} \alpha_\pi^{-1} \alpha_{\pi'}^{-1} \right ) \neq +2$  is easily seen to be dense in the representation variety $\mathcal R_{\PSL}(S)=\PSL^{2g+p-1}$ consisting of all homomorphisms $\bar r\colon \pi_1(S) \to \PSL$, by consideration of a suitable presentation for the free group $\pi_1(S)$. It follows that its image in $\XP(S)$ is dense. Since $U$ is the intersection of finitely many such images, it is dense in $\XP(S)$. 

We consequently found a dense Zariski-open subset $U\subset \XP(S)$ such that every group homomorphism $\bar r\colon \pi_1(S) \to \PSL$ with $[\bar r] \in U$ admits a $\tau$--enhancement, and is therefore associated to an edge weight system for the ideal triangulation $\tau$. 
\end{proof}

We will need to extract some geometric information from this construction. In the above proof, for every peripheral subgroup $\pi$ of $\pi_1(S)$, we had considered a generator $\alpha_\pi$ of $\pi$  represented by a loop going once around the puncture $v$ of $S$ corresponding to $\pi$. We now impose the additional condition that this loop goes counterclockwise around $v$. Then, if we are given a $\tau$--enhancement $\xi \colon \Pi \to \CP^1$ for the homomorphism $\bar r\colon \pi_1(S) \to \PSL$, the point $\xi(\pi)\in \CP^1$ is a line in $\C^2$ of eigenvectors of $\bar r(\alpha_\pi)\in \PSL$ corresponding to an eigenvalue $\lambda_v\in \C^*$. This eigenvalue $\lambda_v$ is determined up to sign since $\bar r(\alpha_\pi)\in \PSL$ is only projectively defined, and depends only on the puncture $v$ by $\bar r$--equivariance of the construction, whence the notation. 

The following is a simple consequence of the precise definition of the shear-bend parameters $a_i\in \C^*$. 

\begin{lem}
 \label{lem:PunctureEigenvalueShearbendParam}
 Let the character $[\bar r] \in \XP(S)$ be associated to the ideal triangulation $\tau$ and to the edge weight system $a\in \left( \C^* \right)^e$ and, for a puncture $v$ of $S$, let the eigenvalue $\lambda_v$ be defined as above. Then,
 $$
 \lambda_v^2 = a_{i_1} a_{i_2} \dots a_{i_k}
 $$
 where the $a_{i_j}$ are the weights of the edges $\gamma_{i_1}$, $\gamma_{i_2}$, \dots, $\gamma_{i_k}$ of $\tau$ that are adjacent to $v$ (counting an edge twice when both of its ends lead to $v$).  
 \qed
\end{lem}

\begin{rem} 
\label{rem:PunctureEigenvalueShearbendParam}
Although the eigenvalue $\lambda_v$ is only defined up to sign, it will be uniquely determined if we are given an $\SL$--character $[r] \in \XX(S)$ lifting $[\bar r] \in \XP(S)$. This will occur in \S \ref{subsect:KauffmanCheFockIntertwiner}. 
\end{rem}

 \subsection{Ideal triangulation sweeps and $\phi$--invariant characters}
 \label{subsect:SweepInvariantChar}

 There are two important elementary operations on ideal triangulations. The first one is a simple \emph{edge reindexing}, where an ideal triangulation $\tau$ with edges $\gamma_i$ is replaced with the ideal triangulation $\tau'$ with edges $\gamma_i' = \gamma_{\sigma(i)}$ for some permutation $\sigma$ of the index set $\{1, 2, \dots,e\}$. 
 
 The second elementary operation is the \emph{diagonal exchange} at the $i_0$--th edge, where the ideal triangulation $\tau$ with edges $\gamma_i$ is replaced with the ideal triangulation $\tau'$ with edges $\gamma_i'$ such that
\begin{itemize}
 \item $\gamma_i'=\gamma_i$ for every $i \neq i_0$;
 \item $\gamma_{i_0}'$ is the other diagonal of the square formed by the two faces of $\tau$ that are adjacent to $\gamma_{i_0}$. 
\end{itemize}
See Figure~\ref{fig:DiagonalExchange} for a pictorial description. Beware that the edges $\gamma_{i_1}=\gamma_{i_1}'$, $\gamma_{i_2}=\gamma_{i_2}'$, $\gamma_{i_3}=\gamma_{i_3}'$, $\gamma_{i_4}=\gamma_{i_4}'$ that form the boundary of the square may not necessarily be distinct. For instance, the identifications ${i_1}={i_3}$ and ${i_2}={i_4}$ will hold in the case of the one-puncture torus that we will consider in \S \ref{sect:IntertwinersPunctTorus}. 

\begin{figure}[htbp]
\vskip 10pt
\SetLabels
(.15 * .55) $\gamma_{i_0}$\\
( .18 * 1.05 ) $\gamma_{i_1}$  \\
( .18 * -.12 ) $\gamma_{i_3}$  \\
( -.03 * .5 ) $\gamma_{i_4}$  \\
( .39 * .5 ) $\gamma_{i_2}$  \\
(.84 * .55) $\gamma_{i_0}'$\\
( .82 * 1.05 ) $\gamma_{i_1}'$  \\
( .82 * -.12 ) $\gamma_{i_3}'$  \\
( .61 * .5 ) $\gamma_{i_4}'$  \\
( 1.04 * .5 ) $\gamma_{i_2}'$  \\
\endSetLabels
\centerline{\AffixLabels{\includegraphics[width=.5\textwidth]{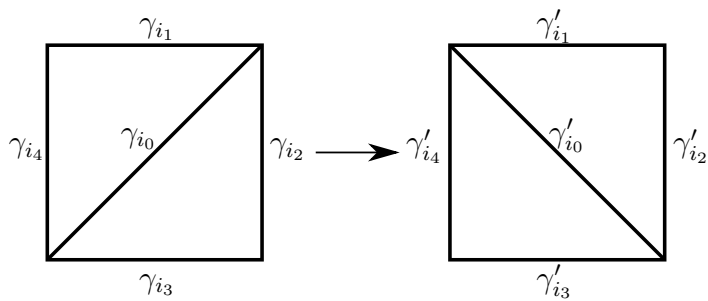}}}
\vskip 10pt

\caption{A diagonal exchange}
\label{fig:DiagonalExchange}
\end{figure}

 If $\tau'$ is obtained by reindexing the edges of $\tau$ by a permutation $\sigma$ of $\{1, 2, \dots, e\}$, the edge weights $a_i'= a_{\sigma(i)}$ for $\tau'$ clearly define the same character $[\bar r] \in \XP(S)$ as the edge weights $a_i$ for $\tau$, as the corresponding pleated surface is unchanged. 
 
 If $\tau'$ is obtained from $\tau$ by a diagonal exchange at its $i_0$--th edge, there similarly exists edge weights $a_i'$ for $\tau'$ that define the same character $[\bar r] \in \XP(S)$ as the edge weights $a_i$ for $\tau$, as long as $a_{i_0} \neq -1$. The precise formula expressing the $a_i'$ in terms of the $a_j$ depends on the possible identifications between the sides of the square where the diagonal exchange takes place and can, for instance, be found in \cite[\S 2]{Liu} or \cite[\S 8]{BonLiu}. To give the flavor, the expression when there are no such side identifications, and with the indexing of Figure~\ref{fig:DiagonalExchange}, is that
\begin{equation}
 \label{eqn:ShearbendEmbeddedDiagEx}
 a_i' =
\begin{cases}
a_i &\text{if } i \neq i_0, i_1, i_2, i_3, i_4\\
a_{i_0}^{-1} &\text{if } i=i_0\\
a_i\big(1+a_{i_0} \big)  &\text{if } i=i_1 \text{ or } i_3\\
a_i\big(1+ a_{i_0}^{-1} \big)^{-1} &\text{if }  i=i_2 \text{ or } i_4
\end{cases}
\end{equation}
We will also encounter in \S \ref{subsect:LeftRightIntertwiners} the appropriate formulas for the case of the one-puncture torus, where $i_1=i_3$ and $i_2=i_4$. 

Now, suppose that we can connect two ideal triangulation $\tau$ and $\tau'$ by a finite sequence of ideal triangulations $\tau = \tau^{(0)}$, $ \tau^{(1)}$, \dots, $ \tau^{(k_0-1)}$, $ \tau^{(k_0)}= \tau'$ where each $ \tau^{(k)}$ is obtained from $ \tau^{(k-1)}$ by an edge reindexing or by a diagonal exchange (such a sequence always exist). We  call this an \emph{ideal triangulation sweep} from $\tau$ to $\tau'$.

We can then start with an edge weight system $a^{(0)} = \big( a_1^{(0)}, a_2^{(0)}, \dots, a_{e}^{(0)} \big) \in \left( \C^* \right)^{e}$ for $\tau^{(0)}= \tau$, and then apply the above formulas to obtain edge weight systems $a^{(k)} \in \left( \C^* \right)^{e}$ for each $\tau^{(k)}$ that all define the same character $[\bar r] \in \XP(S)$. This requires that $a_i^{(k-1)} \neq -1$ whenever $\tau^{(k)}$ is obtained from $\tau^{(k-1)}$ by a diagonal exchange at its $i$--th edge. We will then refer to the sequence $a^{(0)}$, $a^{(1)}$, \dots,  $a^{(k_0)} \in \left( \C^* \right)^{e}$ as an \emph{edge weight system for the ideal triangulation sweep} $\tau = \tau^{(0)}$, $ \tau^{(1)}$, \dots, $ \tau^{(k_0-1)}$, $ \tau^{(k_0)}=\tau'$. 

By construction, an edge weight system for an ideal triangulation sweep uniquely determines a character $[\bar r]\in \XP(S)$. 

We are particularly interested in the case where the ideal triangulation $\tau'$ is the image $\phi(\tau)$ of $\tau$ 
under the orientation-preserving diffeomorphism $\phi \colon S \to S$.  Namely, $\tau'=\phi(\tau)$ is the ideal triangulation whose $i$--th edge is the image under $\phi$ of the $i$--th edge of $\tau$. This example actually provides the motivation for the terminology of ``sweeps'', as such an ideal triangulation sweep provides an ideal triangulation of the mapping torus $M_{\phi, r}$,  whose 2--skeleton is a union of surfaces sweeping around $M_{\phi, r}$; see for instance \cite{Agol}.

\begin{lem}
\label{lem:PeriodicEdgeWeightSystemsInvariantChar}
Let  $\tau = \tau^{(0)}$, $ \tau^{(1)}$, \dots, $ \tau^{(k_0-1)}$, $ \tau^{(k_0)}=\phi(\tau)$ be an ideal triangulation sweep from $\tau$ to $\phi(\tau)$. If the weight system  $a^{(0)}$, $a^{(1)}$, \dots,  $a^{(k_0)} \in \left( \C^* \right)^{e}$ for this sweep is such that $a^{(k_0)} = a^{(0)}$, then its associated character $[\bar r]\in \XP(S)$ is fixed by  the action of $\phi$ on $\XP(S)$. 

Conversely, there is a Zariski-open dense subset $U \subset \XP(S)$ such that every $\phi$--invariant character $[\bar r] \in U$ is associated in this way to an edge weight system $a^{(0)}$, $a^{(1)}$, \dots,  $a^{(k_0)}$ with $a^{(k_0)} = a^{(0)}$. 
\end{lem}
\begin{proof}
Lift $\phi$ to a diffeomorphism $\widetilde \phi \colon \widetilde S \to \widetilde S$. It  is $\phi_*$--equivariant for a suitable homomorphism $\phi_* \colon \pi_1(S) \to \pi_1(S)$, in the sense that $\widetilde \phi \left( \alpha \widetilde x \right) = \phi_*(\alpha) \widetilde\phi \left( \widetilde x \right)$ for every $\widetilde x \in \widetilde S$ and $\alpha \in \pi_1(S)$. 

 If $\widetilde f^{(0)}$, $\widetilde f^{(k_0)} \colon \widetilde S \to \HH^3$ are $\bar r$--equivariant pleated surfaces respectively associated to the ideal triangulations $\tau^{(0)}$, $\tau^{(k_0)}$ and to the edge weight systems $a^{(0)}$, $a^{(k_0)}$, then $\widetilde f^{(k_0)} \circ \widetilde \phi$ is pleated along $\tau^{(0)}$ according to the edge weight system   $a^{(k_0)}$. If  $a^{(0)} = a^{(k_0)}$, it follows that $\widetilde f^{(k_0) }\circ \widetilde \phi$ coincides with $\widetilde f^{(0)}$ up to post-composition with an isometry $A \in \PSL$ of $\HH^3$ and pre-composition with the lift to $\widetilde S$ of an isotopy of $S$. Note that $\widetilde f^{(k_0) }\circ \widetilde \phi$ is $\bar r\circ \phi_*$--equivariant. Since $\bar r$ is the only homomorphism $\pi_1(S) \to \PSL$ for which $\widetilde f^{(0)}$ is $\bar r$--equivariant, it follows that $\bar r(\alpha ) = A\left( \bar r\circ\phi_*(\alpha) \right) A^{-1}$ for every $\alpha \in \pi_1(S)$. 
 
 In particular, $[\bar r] = [\bar r \circ \phi_*] $ in $ \XP(S)$. In other words, the character $[\bar r]\in \XP(S)$ is fixed by the action of $\phi$ on $\XP(S)$.

For the second statement, we will use the set-up of Lemma~\ref{lem:IdealTriangWeightsAndCharacters} and its proof. For each $k=0$, $1$, \dots, $k_0$, that proof provides  a dense Zariski-open subset $U_k \subset \XP(S)$ such, for every character $[\bar r]\in U_k$, every $\bar r$--equivariant  map $\xi \colon \Pi \to \CP^1$ is a $\tau^{(k)}$--enhancement of $\bar r$. (Note that the $\bar r$--equivariance implies that, for every peripheral subgroup $\pi \in \Pi$, the point $\xi(\pi)$ is fixed by $\bar r(\pi)$.) Also, let $V$ be the set of $[\bar r]\in \XP(S)$ such that $\Tr \bar r(\alpha_\pi)\neq 0$ for every generator $\alpha_\pi$ of a peripheral subgroup $\pi \in \Pi$; this is equivalent to the property that no $\bar r(\pi)$ is cyclic of order 2, generated by a rotation of $180^\circ$. This subset $V$ is dense and Zariski-open in $\XP(S)$. Set $U$ to be the intersection $V \cap U_0 \cap U_1 \cap \dots \cap U_{k_0}$. 

Suppose that the character $[\bar r] \in U$ is $\phi$--invariant.  This means that $[\bar r]=[\bar r\circ \phi_*]$ in $\XP(S)$. Note that $\bar r$ is necessarily irreducible by definition of the $U_k$ in the proof of Lemma~\ref{lem:IdealTriangWeightsAndCharacters}. Therefore, there exists $A \in \PSL$ such that $\bar r\circ \phi_*(\alpha) = A\, \bar r(\alpha) A^{-1}$ for every $\alpha \in \pi_1(S)$. 

The group homomorphism $\phi_* \colon \pi_1(S) \to \pi_1(S)$ permutes the peripheral subgroups of $\pi_1(S)$, since it comes from a diffeomorphism of $S$. For such a peripheral subgroup $\pi \in \Pi$, the map $A \colon \CP^1 \to \CP^1$ sends a fixed point of $\bar r(\pi)$ to a fixed point of $\bar r( \phi_*(\pi))$ since $\bar r\circ \phi_*(\pi) = A\,\bar r(\pi) A^{-1}$. 

In view of the above observation, we would like to find an $\bar r$--equivariant map $\xi \colon \Pi \to \CP^1$ such that, in addition, $\xi(\phi_*(\pi) ) = A(\xi(\pi))$ for every $\pi \in \Pi$. As usual, we will proceed by considering the finitely many orbits of the group generated by the actions of $\pi_1(S)$ and $\phi_*$ on $\Pi$. However, there is a compatibility condition to check when there exists a nonzero power $j\in \Z$ and an element $\alpha \in \pi_1(S)$ such that $\phi^j_*(\pi) = \alpha \pi \alpha^{-1}$ for some peripheral subgroup $\pi$ (this happens precisely when $\pi$ corresponds to a puncture that is preserved by $\phi^j$); we then need that $A^j \xi(\pi) =\bar r(\alpha) \xi(\pi)$. 

To verify this compatibility condition, we use the fundamental property that  $\phi_*$ comes from an \emph{orientation-preserving} diffeomorphism of $S$. This implies that, if $\phi^j_*(\pi) = \alpha \pi \alpha^{-1}$ for some $\alpha \in \pi_1(S)$, then $\phi^j_*(\beta) = \alpha \beta \alpha^{-1}$ (as opposed to  $\alpha \beta^{-1} \alpha^{-1}$) for every $\beta \in \pi$. As a consequence, $A^j \bar r(\beta) A^{-j} =\bar  r\circ \phi_*^j (\beta ) = \bar r(\alpha) \bar r(\beta) \bar r(\alpha)^{-1}$ for every $\beta \in \pi$, and $\bar r(\alpha)^{-1}A^j$ commutes with every element of $\bar r(\pi)$. By definition of the subset $U \subset \XP(S)$, the cyclic subgroup $\bar r(\pi)\subset \PSL$ is nontrivial, and therefore fixes exactly 1 or 2 points of $\CP^1$. Since $\bar r(\alpha)^{-1}A^j$ commutes with every element of $\bar r(\pi)$ and since we excluded the case where $\bar r(\pi)$ has only two elements by our definition of $U$, an elementary argument in $\PSL$ shows that $\bar r(\alpha)^{-1}A^j$  fixes each of the fixed points of $\bar r(\pi)$. In particular, $A^j \xi(\pi) =\bar r(\alpha) \xi(\pi)$ for any choice of fixed point $\xi(\pi)$. 

Because of this compatibility property, we can proceed orbit by orbit and define an $\bar r$--equivariant map $\xi \colon \Pi \to \CP^1$ such that $\xi(\phi_*(\pi) ) = A(\xi(\pi))$ for every $\pi \in \Pi$. This map is a $\tau^{(k)}$--enhancement of $r$ for every $k=0$, $1$, \dots, $k_0$ by definition of $U$. As a consequence, it provides for every $k$ an  $\bar r$--equivariant pleated surface $\widetilde f^{(k)} \colon \widetilde S \to \HH^3$ pleated along $\tau^{(k)}$. Let $a^{(k)} \in \left( \C^* \right)^e$ be the edge weight system for $\tau^{(k)}$ defined by the pleated surface $\widetilde f^{(k)}$.

By definition, the $i$--th coordinate $a_i^{(0)} \in \C^*$ of $a^{(0)} \in \left( \C^* \right)^e$ is defined as follows.
Let $\gamma_i$ be the $i$--th edge of the ideal triangulation $\tau^{(0)}$, and lift it to and edge $\widetilde \gamma_i $  of the induced ideal triangulation $\widetilde \gamma_i$ of $\widetilde S$.  Let $\pi$, $\pi'$, $\pi''$ and $\pi''' \in \Pi$ be the peripheral subgroups associated to the four vertices of the square formed by the two faces of $\widetilde \tau^{(0)}$ that are adjacent to $\widetilde\gamma_i$. Then the edge weight $a_i^{(0)} \in \C^*$ is the crossratio of the four points $\xi(\pi)$, $\xi(\pi')$, $\xi(\pi'')$, $\xi(\pi''') \in \CP^1$. 

Similarly, to compute  the $i$--th coordinate $a_i^{(k_0)}$ of $a^{(k_0)} \in \left( \C^* \right)^e$, we consider the edge $\phi(\gamma_i)$ of $\tau^{(k_0)} = \phi\left( \tau^{(0)} \right)$,  its lift $\widetilde \phi(\widetilde\gamma_i) \subset \widetilde S$. The peripheral subgroups associated to the vertices of the square formed by the two faces of $\widetilde \tau^{(k_0)}$ that are adjacent to $\widetilde \phi(\widetilde\gamma_i) $ are then $\phi_*(\pi)$, $\phi_*(\pi')$, $\phi_*(\pi'')$ and $\phi_*(\pi'') \in \Pi$. By definition, $a_i^{(k_0)} \in \C^*$ is the crossratio of $\xi(\phi_*(\pi))$, $\xi(\phi_*(\pi'))$, $\xi(\phi_*(\pi''))$, $\xi(\phi_*(\pi''')) \in \CP^1$. There four points are also $A\,\xi(\pi)$, $A\,\xi(\pi')$, $A\,\xi(\pi'')$, $A\,\xi(\pi''') $ by construction of $\xi$, and their crossratio is therefore the same as the crossratio of  $\xi(\pi)$, $\xi(\pi')$, $\xi(\pi'')$, $\xi(\pi''')$ by invariance of the crossratio under $\PSL$. 

This proves that $a_i^{(k_0)} = a_i^{(0)}$ for every $i$, and therefore that the edge weight systems $a^{(k_0)} $ and $a^{(0)} \in \left( \C^* \right)^e$ coincide. Since our $\phi$--invariant $[\bar r]\in U$ was associated to these edge weights, this concludes the proof of the second statement of Lemma~\ref{lem:PeriodicEdgeWeightSystemsInvariantChar}. 
 \end{proof}

We will say that an edge weight system $a^{(0)}$, $a^{(1)}$, \dots,  $a^{(k_0)} \in \left( \C^* \right)^{e}$  for the ideal triangulation sweep $\tau = \tau^{(0)}$, $ \tau^{(1)}$, \dots, $ \tau^{(k_0-1)}$, $ \tau^{(k_0)}=\phi(\tau)$ is \emph{periodic} if $a^{(k_0)} = a^{(0)}$.

Given an ideal triangulation sweep $\tau = \tau^{(0)}$, $ \tau^{(1)}$, \dots, $ \tau^{(k_0-1)}$, $ \tau^{(k_0)}=\phi(\tau)$, Lemma~\ref{lem:PeriodicEdgeWeightSystemsInvariantChar} asserts that, provided we restrict attention to a Zariski-open dense subset of $\XP(S)$, finding all $\phi$--characters is equivalent to finding all periodic edge weight systems for this ideal triangulation sweep. Just beware that a $\phi$--invariant character $[\bar r] \in U$ can be associated to several periodic edge weight systems. 

An edge weight system $a^{(0)}$, $a^{(1)}$, \dots,  $a^{(k_0)} \in \left( \C^* \right)^{e}$  for an ideal triangulation sweep is completely determined by its first element $a^{(0)} \in \left( \C^* \right)^{e}$, which canonically embeds this set of edge weight systems in $\left( \C^* \right)^{e}$. In this space, the condition  that  $a^{(k_0)} = a^{(0)}$ is given by  $e$ rational equations in $e$ complex unknowns. One would consequently expect the solution space to be $0$--dimensional. However, these equations are not independent, and the dimension of the space of periodic edge weight system is actually higher. 

\begin{lem}
\label{lem:HypEdgeWeightSystemForSweep}
 For an  ideal triangulation sweep $\tau = \tau^{(0)}$, $ \tau^{(1)}$, \dots, $ \tau^{(k_0-1)}$, $ \tau^{(k_0)}=\phi(\tau)$, there is a unique periodic edge weight system $a_\hyp^{(0)}$, $a_\hyp^{(1)}$, \dots,  $a_\hyp^{(k_0-1)}$, $a_\hyp^{(k_0)}=a_\hyp^{(0)}$ whose associated character $[\bar r_\hyp] \in \XP(S)$ is represented by the restriction $\bar r_{\mathrm{hyp}} \colon \pi_1(S) \to \PSL$ of the monodromy of the complete hyperbolic metric of the mapping torus $M_{\phi, r}$. 
 
 In the space $ \left( \C^* \right)^{e}$ of edge weight systems for the  ideal triangulation $\tau^{(0)}$, those corresponding to periodic edge weight systems for the ideal triangulation sweep form a subspace of complex dimension $c$ near this ``hyperbolic'' edge weight system $a_\hyp^{(0)}$, where $c\geq 1$ is the number of orbits of the action of $\phi$ on the set of punctures of $S$. 
\end{lem}
\begin{proof}
We rely on the following two properties of the homomorphism $\bar r_\hyp \colon \pi_1(S) \to \PSL$. The first one is that, for every peripheral subgroup $\pi \in \Pi$, the image $\bar r_\hyp(\pi)$ is a parabolic subgroup of $\PSL$, and in particular fixes a unique point of $\CP^1$. The second property is that $\bar r_\hyp$ is injective. 

Because of the first property, there is a unique $\bar r_\hyp$--equivariant map $\xi \colon \Pi \to \CP^1$. Two parabolic subgroups of $\PSL$ that fix the same point of $\CP^1$ necessarily commute with each other and, any two distinct parabolic subgroups $\pi$, $\pi' \in \Pi$ generate a free group on two generators in $\pi_1(S)$. Since $\bar r_\hyp$ is injective it follows that, for distinct $\pi$, $\pi' \in \Pi$, the parabolic subgroups $\bar r_\hyp(\pi)$ and $\bar r_\hyp(\pi')$ fix different points of $\CP^1$. As a consequence, for any ideal triangulation $\tau$, the map $\xi \colon \Pi \to \CP^1$ is a $\tau$--enhancement of $\bar r_\hyp$ and this enhancement is unique. 

We can then use this enhancement $\bar \xi$ to construct pleated surfaces $\widetilde f_\hyp^{(k)} \colon \widetilde S \to \HH^3$, respectively pleated along the ideal triangulations $\tau^{(k)}$. The shear-bend parameters of these pleated surfaces provide an edge weight system $a_\hyp^{(0)}$, $a_\hyp^{(1)}$, \dots,  $a_\hyp^{(k_0)}$ for the ideal triangulation sweep. Since $[\bar r_\hyp \circ \pi_*]=[\bar r_\hyp]$, it follows from the uniqueness of the $\tau^{(0}$--enhancement $\xi$ that the pleated surfaces $\widetilde f_\hyp^{(k)} $ and $\widetilde f_\hyp^{(k)}  \circ \widetilde \phi$ has the same shear-bend parameters (see  the end of the proof or Lemma~\ref{lem:PeriodicEdgeWeightSystemsInvariantChar}), which proves that $a_\hyp^{(k_0)}=a_\hyp^{(0)}$ and completes the proof of the first statement. 

For the second statement, we use the well-known property that the character variety $\XP(M_{\phi, r})$ is smooth near its hyperbolic character, and that its complex dimension there is equal to the number $c$ of ends of the mapping torus $M_{\phi, r}$ (see for instance \cite[\S 5]{ThuNotes} or \cite[\S E.6]{BenPet}). It follows that the set of $\phi$--invariant characters in $\XP(S)$ has dimension $c$ near the restriction $[\bar r_\hyp] \in \XP(S)$. Note that $c$ is also the number of orbits of the action of $\phi$ on the set of punctures of the surface $S$. Our proof of Lemma~\ref{lem:PeriodicEdgeWeightSystemsInvariantChar} shows that every $\phi$--invariant  character $[\bar r] \in \XP(S)$ near  $[\bar r_\hyp]$ admits at most $2^c$ $\tau^{(0)}$--enhancements, since $r(\alpha_\pi)$ can have at most two fixed points for every parabolic subgroup $\pi$. In other words, the map which associates a $\phi$--invariant character to each periodic edge weight system for the ideal triangulation sweep is at most $2^c$ to $1$ near $a_\hyp$. This proves that this space of periodic edge weight systems has complex dimension $c$ near $a_\hyp$. (See comment in Remark~\ref{rem:HypEdgeWeightSystemForSweep} below.)
\end{proof}

\begin{rem}
\label{rem:HypEdgeWeightSystemForSweep}
 In the second part of Lemma~\ref{lem:HypEdgeWeightSystemForSweep}, we have been deliberately vague about the meaning of ``a subspace of complex dimension $c\,$''. With a little more work, it can be shown that this subspace is actually a $c$--dimensional complex submanifold of $\big( \C^* \big)^e$ near $a_\hyp^{(0)}$. See \cite[\S 15.2.7]{Mart} for a closely related property. However, we will not need this fact. 
\end{rem}

 \subsection{The Chekhov-Fock algebra of an ideal triangulation, and its representations}
 \label{subsect:CheFock}
  
 The \emph{Chekhov-Fock algebra} $\T_\tau^q(S) = \C[X_1^{\pm1}, X_2^{\pm1}, \dots, X_e^{\pm1}]_\tau^q$ of the ideal triangulation $\tau$ is the algebra of Laurent polynomials in variables $X_1$, $X_2$, \dots, $X_e$ associated to the edges of $\tau$, which do not commute but instead are subject to the  skew-commutativity relations
 $$
 X_i X_j = q^{2\sigma_{ij}} X_jX_i,
 $$
 where each $\sigma_{ij} \in \Z$ is an integer determined by adjacency properties between the $i$--th and $j$--th edges in the triangulation $\tau$. See \cite{CheFoc1,Liu, BonLiu} for precise definitions of these coefficients $\sigma_{ij}$. 
 
Algebraically, $\T_\tau^q(S)$ is what is known as a quantum torus. In particular, its representation theory is fairly simple, at least once we know how to diagonalize over the integers the antisymmetric matrix whose entries are the $\sigma_{ij}$. This is done in \cite{BonLiu}. 

The classification of irreducible representations of $\T_\tau^q(S)$ involves special elements $H_v \in \T_\tau^q(S)$ associated to each puncture $v$ of $S$. If $\gamma_{i_1}$, $\gamma_{i_2}$, \dots, $\gamma_{i_j}$ are the edges  of $\tau$ that are adjacent to $v$ (with an edge occurring twice in this list if both of its ends lead to $v$,), then
$$
H_v = q^{-\sum_{1 \leq k < k' \leq j}  \sigma_{kk'} }  X_{i_1} X_{i_2} \dots X_{i_j}
$$
where the $\sigma_{kk'}$ are the coefficients occurring in the skew-commutativity relations defining $\T_\tau^q(S) $. The power of $q$ is specially designed so that this element $H_v \in \T_\tau^q(S)$ does not depend on the ordering of the edges $\gamma_{i_1}$, $\gamma_{i_2}$, \dots, $\gamma_{i_j}$. 

\begin{prop} [{\cite[Theorem 21]{BonLiu}}]
\label{prop:RepsCheFock}
If $q$ is a primitive $n$--root of unity with $n$ odd, and if $\rho \colon \T_\tau^q(S) \to \End(V)$ is an irreducible representation of the Chekhov-Fock algebra $\T_\tau^q(S)$, then $\dim V = n^{3g+p-3}$ if the surface $S$ has genus $g$ and $p$ punctures, and there exist numbers $a_1$, $a_2$, \dots, $a_e$, $h_{v_1}$, $h_{v_2}$, \dots, $h_{v_p}\in \C^*$ such that 
\begin{align*}
 \rho(X_i^n) &= a_i \, \Id_V &
 \rho(H_{v_j}) &= h_{v_j} \,\Id_V 
\end{align*}
for every edge $\gamma_i$ and puncture $v_j$ of the ideal triangulation $\tau$. 

In addition, $\rho$ is determined up to isomorphism by this set of invariants $a_1$, $a_2$, \dots, $a_e$, $h_{v_1}$, $h_{v_2}$, \dots, $h_{v_p}\in \C^*$, and such a collection of numbers is realized by an irreducible representation if and only if it satisfies the  condition that
$$h_{v}^n = a_{i_1} a_{i_2} \dots a_{i_j}$$
for every puncture $v$ that is  adjacent to the edges $\gamma_{i_1}$, $\gamma_{i_2}$, \dots, $\gamma_{i_j}$ of $\tau$ (where an edge occurs twice in this list when both of its ends lead to $v$). 
\qed
\end{prop}

In particular, an irreducible representation $\rho \colon \T_\tau^q(S) \to \End(V)$ is determined by an edge weight system $a=(a_1, a_2, \dots, a_e) \in \left( \C^* \right)^e$, up to a finite number $n^p$ of choices of roots for data determined by $a$. There is a similar classification when $n$ is even, but it is somewhat more cumbersome. This general case involves a global invariant $h$ such that $h^n=a_1a_2 \dots a_e$ and $h^2 = h_{v_1} h_{v_2} \dots h_{v_p}$, but this $h$ is uniquely determined by the $h_{v_j}$ when $n$ is odd. 

The key idea underlying \cite{BonLiu}  is that this classification is well-behaved with respect to the ideal triangulation sweeps of \S \ref{subsect:SweepInvariantChar}. The connection is provided by the \emph{Chekhov-Fock coordinate change isomorphisms} $\Phi_{\tau \tau'}^q \colon \widehat \T_{\tau'}^q(S) \to \widehat \T_\tau^q(S)$ which, for any two ideal triangulations $\tau$ and $\tau'$, identify the fraction algebras $ \widehat \T_\tau^q(S)$ and $ \widehat \T_{\tau'}^q(S) $ of the two Chekhov-Fock algebras $  \T_\tau^q(S)$ and $  \T_{\tau'}^q(S) $. These isomorphisms satisfy the fundamental relation that
 $$
 \Phi_{\tau \tau''}^q = \Phi_{\tau \tau'}^q \circ \Phi_{\tau' \tau''}^q
 $$
 for any three ideal triangulations $\tau$, $\tau'$, $\tau''$, and  are given by explicit formulas when we are given an ideal triangulation sweep connecting the two ideal triangulations. These formulas can be found in  \cite{CheFoc1, CheFoc2, Liu, BonLiu}, but   we can illustrate their flavor by giving them in two simple cases.
 
 If $\tau'$ is obtained from $\tau$ by an edge reindexing, and more precisely if the $i$--th edge $\gamma_i'$ of $\tau'$ is the $\sigma(i)$--th edge $\gamma_{\sigma(i)}$ of $\tau$ for some permutation $\sigma$ of $\{ 1,2, \dots, e\}$, then $\Phi_{\tau \tau'}^q \colon \widehat \T_{\tau'}^q(S) \to \widehat \T_\tau^q(S)$ is the unique algebra homomorphism such that $\Phi_{\tau \tau'}^q (X_i') = X_{\sigma(i)}$. 
 
If $\tau'$ is obtained from $\tau$ by a diagonal exchange as in Figure~\ref{fig:DiagonalExchange}, with this edge indexing, and if the edges $\gamma_{i_1}$, $\gamma_{i_2}$, $\gamma_{i_3}$, $\gamma_{i_4}$ are all distinct, then $\Phi_{\tau \tau'}^q \colon \widehat \T_{\tau'}^q(S) \to \widehat \T_\tau^q(S)$ is the unique algebra homomorphism such that
\begin{equation}
 \label{eqn:CheFockIsomEmbeddedDiagEx}
 \Phi_{\tau \tau'}^q (X_i') =
\begin{cases}
X_i &\text{if } i \neq i_0, i_1, i_2, i_3, i_4\\
X_{i_0}^{-1} &\text{if } i=i_0\\
X_i\big(1+ q^{-1} X_{i_0} \big)  &\text{if } i=i_1 \text{ or } i_3\\
X_i\big(1+q^{-1}  X_{i_0}^{-1} \big)^{-1} &\text{if }  i=i_2 \text{ or } i_4
\end{cases}
\end{equation}
Note the analogy with (\ref{eqn:ShearbendEmbeddedDiagEx}). 

The formulas for a diagonal exchange taking place in a square where some sides are identified are given in \cite{Liu, BonLiu}. We will encounter the extreme case of the one-puncture torus in \S  \ref{subsect:ChekhovFockPunctTorus}.

Given two ideal triangulations $\tau$ and $\tau'$, and a representation  $\rho \colon \T_\tau^q(S) \to \End(V)$, it would be natural to define a representation  $\rho' \colon \T_{\tau'}^q(S) \to \End(V)$ as $\rho' = \rho \circ \Phi_{\tau\tau'}^q$. However, we need to be careful to make sense of this statement, as $\Phi_{\tau\tau'}^q$ is valued in the fraction algebra of $\T_\tau^q(S)$, and we have to worry about denominators.  

We will say that the representation $ \rho \circ \Phi_{\tau\tau'}^q \colon \T_{\tau'}^q(S) \to \End(V)$  \emph{makes sense} if, for every $W' \in  \T_{\tau'}^q(S) $, its image $ \Phi_{\tau \tau'}^q(W') \in  \widehat  \T_\tau^q(S) $ can be written as a right and a left fraction
$$
 \Phi_{\tau \tau'}^q(W')  = P_1 Q_1^{-1} = Q_2^{-1} P_2,
$$
with $P_1$, $Q_1$, $P_2$, $Q_2 \in    \T_\tau^q(S)$,  such that $\rho(Q_1)$ and $\rho(Q_2)$ are invertible in $\End(V)$. We then define
$$
 \rho \circ \Phi_{\tau\tau'}^q  (W) = \rho(P_1) \rho(Q_1)^{-1} = \rho(Q_2)^{-1} \rho(P_2). 
$$
The condition on left and right fractions  guarantees that this uniquely determines an  algebra homomorphism $ \rho \circ \Phi_{\tau\tau'}^q \colon \T_{\tau'}^q(S) \to \End(V)$. 

\begin{prop} [{\cite[Lemma 27]{BonLiu}}]
\label{prop:CheFockRepsCompatibleIdealTriangChange}
 Let $\tau = \tau^{(0)}$, $ \tau^{(1)}$, \dots, $ \tau^{(k_0-1)}$, $ \tau^{(k_0)}=\tau'$ be an ideal triangulation sweep connecting the two ideal triangulations $\tau$ and $\tau'$, and let $a=a^{(0)}$, $a^{(1)}$, \dots,  $a^{(k_0)} =a' \in \left( \C^* \right)^{e}$ be an edge weight system for this sweep. If $a=(a_1, a_2, \dots, a_e)$, let $\rho \colon \T_\tau^q(S) \to \End(V)$ be the irreducible representation associated by Proposition~{\upshape\ref{prop:RepsCheFock}} to these  edge weights $a_i$ and to puncture invariants $h_v$. Then the representation $\rho' = \rho \circ \Phi_{\tau\tau'}^q \colon \T_{\tau'}^q(S) \to \End(V)$ makes sense, and is classified by the edge weight system $a'=(a_1', a_2', \dots, a_e')$ for $\tau'$ and by the same puncture  invariants $h_v$  as $\rho$. \qed
\end{prop}
 
\subsection{Chekhov-Fock intertwiners for surface diffeomorphisms} 
\label{subsect:CheFockIntertwinerInvariantDiffeomorphism}

We can now use a version of the arguments of \S \ref{subsect:KauffmanIntertwinerInvariantDiffeomorphism} to construct invariants for orientation-preserving diffeomorphisms $\phi \colon S \to S$ endowed with $\phi$--invariant characters $[\bar r] \in \XP(S)$. 

Let $\tau = \tau^{(0)}$, $ \tau^{(1)}$, \dots, $ \tau^{(k_0-1)}$, $ \tau^{(k_0)}=\phi(\tau)$ be an ideal triangulation sweep connecting the two ideal triangulations $\tau$ to its image $\phi(\tau)$, and let $a= a^{(0)}$, $ a^{(1)}$, \dots, $ a^{(k_0-1)}$, $ a^{(k_0)}= a$ be a periodic edge weight system for this ideal triangulation sweep, defining a $\phi$--invariant character $[\bar  r] \in \XP(S)$. 

The classification of irreducible representations of $\T_\tau^q(S)$ in  Proposition~\ref{prop:RepsCheFock} involves, in addition to an edge weight system  $ a\in \left(\C^* \right)^e$ for $\tau$, puncture weights $h_v$ constrained by the property that
$$
h_v^n = a_{i_1} a_{i_2} \dots a_{i_j}
$$
if the puncture $v$ is adjacent to the edges $\gamma_{i_1}$, $\gamma_{i_2}$, \dots, $\gamma_{i_j}$ of $\tau$.

If $v' = \phi(v)$ is adjacent to the edges $\gamma_{i_1'}$, $\gamma_{i_2'}$, \dots, $\gamma_{i_{j'}'}$ of $\tau= \tau^{(0)}$, it is also adjacent to the edges $\phi(\gamma_{i_1})$, $\phi(\gamma_{i_2})$, \dots, $\phi(\gamma_{i_j})$ of $\phi(\tau) = \tau^{(k_0)}$. By  inspection of the compatibility equations  for the edge weight systems $a^{(k)}$ considered in \S \ref{subsect:SweepInvariantChar} (see for instance \cite[Prop.~14]{Liu}),
$$
a_{i_1'} a_{i_2'} \dots a_{i_{j'}'}  = a_{i_1'}^{(0)} a_{i_2'}^{(0)} \dots a_{i_{j'}'}^{(0)} = a_{i_1}^{(k_0)} a_{i_2}^{(k_0)} \dots a_{i_j}^{(k_0)}  = a_{i_1} a_{i_2}\dots a_{i_j}.
$$
We can therefore choose the puncture weights $h_v$ so that $h_{\phi(v)}=h_v$ for every puncture $v$ of $S$. We will say that  such a choice of puncture weights is \emph{$\phi$--invariant}. 
 
Let $\bar \rho \colon \T_\tau^q(S) \to \End(\bar V)$ be an irreducible representation of the Chekhov-Fock algebra of $\tau$ that, as in Proposition~\ref{prop:RepsCheFock}, is classified by the edge weight system $a \in \left( \C^* \right)^e$ and by $\phi$--invariant puncture weights $h_v$. 
By Proposition~\ref{prop:CheFockRepsCompatibleIdealTriangChange}, the representation $\bar \rho \circ \Phi_{\tau\phi(\tau)}^q \colon \T_{\phi(\tau)}^q \to \End(\bar V)$ makes sense, and is classified by the same edge weight system $a=a^{(k_0)} \in \left( \C^* \right)^e$, considered as an edge weight system for $\phi(\tau) = \tau^{(k_0)}$, and by the same puncture weights $h_v$ as $\bar \rho$. 

The diffeomorphism $\phi$ itself induces another natural isomorphism $\Psi_{\phi(\tau) \tau}^q \colon \T_\tau^q(S) \to \T_{\phi(\tau)}^q(S)$, sending the generator of $\T_\tau^q(S)$ corresponding to the $i$--th edge $\gamma_i$ of $\tau$ to the generator of $\T_{\tau'}^q(S)$ corresponding to the $i$--th edge $\phi(\gamma_i)$ of $\phi(\tau)$. Indeed, the existence of $\phi$ and the precise definition of the coefficients $\sigma_{ij}$ guarantees that these generators satisfy the same $q$--commutativity relations $X_iX_j = q^{2\sigma_{ij}} X_jX_i$ in both $\T_\tau^q(S)$ and $\T_{\phi(\tau)}^q(S)$. 

We can then consider the representation $\bar \rho \circ \Phi_{\tau\phi(\tau)}^q \circ\Psi_{\phi(\tau) \tau}^q  \colon \T_{\tau}^q \to \End(\bar V)$. Its invariants are the same edge weight system $a \in \left( \C^* \right)^e$ as $\bar \rho \circ \Phi_{\tau\phi(\tau)}^q$, but this time considered as an edge weight system for $\tau$, and its puncture invariant $h_v'$ at the puncture $v$ is $h_v'= h_{\phi^{-1}(v)}$. Since we arranged that $h_v'= h_{\phi^{-1}(v)}$ for every puncture $v$, it follows from Proposition~\ref{prop:RepsCheFock} that the representations $\bar \rho$ and $\bar \rho \circ \Phi_{\tau\phi(\tau)}^q \circ\Psi_{\phi(\tau) \tau}^q  \colon \T_{\tau}^q \to \End(\bar V)$ are isomorphic, by an isomorphism  $\bar \Lambda_{\phi, \bar r}^q \colon \bar V \to \bar V$ such that 
$$
\left( \bar \rho \circ \Phi_{\tau\phi(\tau)}^q \circ \Psi_{\phi(\tau) \tau}^q \right) (X) = \bar \Lambda_{\phi, \bar r}^q  \circ  \rho  (X)  \circ \bar \Lambda_{\phi, \bar r}^q{}^{-1} \in \End( \bar V)
$$
for every $X \in \T_{\tau}^q(S)$. 

As before, we normalize $ \bar \Lambda_{\phi, \bar r}^q $ so that its determinant $\det  \bar \Lambda_{\phi, \bar r}^q $ has modulus $1$. 

As in Proposition~\ref{prop:KauffmanIntertwinerDefined}, the following is an immediate consequence of the fact that the representation $\bar\rho$ is irreducible and unique up to isomorphism. 

\begin{prop}
 \label{prop:QuantumTeichmullerIntertwinerDefined}
 Let  $\bar \Lambda_{\phi, \bar r}^q  \colon\bar  V \to \bar V$ be the above intertwiner, normalized so that $\left| \det \bar \Lambda_{\phi, \bar r}^q \right|=1$. Then, up to conjugation and  multiplication by a scalar of modulus $1$,  $\bar \Lambda_{\phi, \bar r}^q $ depends only on the diffeomorphism  $\phi$, the ideal triangulation sweep from $\tau$ to $\phi(\tau)$, the periodic edge weight system for this sweep, and the $\phi$--invariant puncture weights $h_v \in \C^*$.
 
 In particular,  the modulus $\big\lvert  \Tr \bar \Lambda_{\phi, \bar r}^q \,\big\rvert$  of its trace is uniquely determined by this data. \qed
\end{prop}

The combination of  Proposition~\ref{prop:KauffmanIntertwinerDefined} with the following statement  provides a  stronger uniqueness statement, involving much less data. 

For this statement, we need to eliminate some very special characters $[r]\in \XX(S)$, those which admit a non-trivial \emph{sign-reversal symmetry} as defined below. The cohomology group $H^1(S; \Z/2)$  acts on the character variety $\XX(S)$ by the property that, if  $[r] \in \XX(S)$ is represented by a homomorphism $r \colon \pi_1(S) \to \SL$, its image $\epsilon[r] \in \XX(S)$ under the action of $\epsilon \in H^1(S; \Z/2)$ is represented by the homomorphism $\epsilon r \colon \pi_1(S) \to \SL$ such that
 $$
 \epsilon r(\alpha)  = (-1)^{\epsilon([\alpha])} r(\alpha) \in \SL
 $$
 for every $\alpha \in \pi_1(S)$, where $[\alpha]$ denote the class of $\alpha$ in the homology group $H_1(S;\Z/2)$. A character  $[r] \in \XX(S)$ has a \emph{nontrivial sign-reversal symmetry} if it is fixed under the action of some nontrivial $\epsilon \in H^1(S; \Z/2)$. This is equivalent to the property that $\Tr r(\alpha)=0$ for every $\alpha \in \pi_1(S)$ with $\epsilon([\alpha]) \neq 0$. These symmetries are rare, and the characters $[r] \in \XX(S)$ admitting a non-trivial sign-reversal symmetry form a Zariski closed subset of high codimension in $\XX(S)$; see the end of \cite[\S 5.1]{BonWon4}. 

\begin{thm}
 \label{thm:CheFockAndKauffmanIntertwiners}
 
 Let $q$ be a primitive $n$--root of unity with $n$ odd, and choose a square root $q^{\frac12}$ such that $\big( q^{\frac12} \big)^n = -1$. Let $\Lambda_{\phi, r}^q \colon V \to V$ be the intertwiner associated  by Proposition~{\upshape\ref{prop:KauffmanIntertwinerDefined}} to a $\phi$--invariant irreducible character $[r] \in \XX(S)$ and to $\phi$--invariant puncture weights $p_v \in \C^*$. 
 
 Similarly, let $\bar \Lambda_{\phi, \bar r}^q \colon \bar V \to \bar V$ be the  intertwiner associated by Proposition~{\upshape\ref{prop:QuantumTeichmullerIntertwinerDefined}} to a periodic edge weight system $a= a^{(0)}$, $ a^{(1)}$, \dots, $ a^{(k_0-1)}$, $ a^{(k_0)}= a\in \left(\C^* \right)^e$ for an ideal triangulation sweep $\tau = \tau^{(0)}$, $ \tau^{(1)}$, \dots, $ \tau^{(k_0-1)}$, $ \tau^{(k_0)}=\phi(\tau)$ and to $\phi$--invariant puncture weights $h_v$. 
 
 Suppose that the character $[r]$ admits no nontrivial sign-reversal symmetry, and that the two data sets are connected by the following properties:
\begin{enumerate}
 \item   the $\phi$--invariant character $[\bar r] \in \XP(S)$ associated to the periodic edge weight system is equal to the image of $[r] \in \XX(S)$ under the canonical projection $\XX(S) \to \XP(S)$;
 \item  $p_v^2 = h_v + h_v^{-1} +2$ for every puncture $v$.
\end{enumerate}
 Then, up to scalar multiplication by an $n$--root of unity, the intertwiners $\Lambda_{\phi, r}^q$ and $\bar \Lambda_{\phi, \bar r}^q$ are conjugate to each other by an isomorphism $V \to \bar V$. 
\end{thm}

The next subsection is devoted to the proof of Theorem~\ref{thm:CheFockAndKauffmanIntertwiners}.

\subsection{Comparing the Kauffman and Chekhov-Fock intertwiners}
\label{subsect:KauffmanCheFockIntertwiner}
The proof will use a third algebra, the balanced Chekhov-Fock (square root) algebra $\ZZ_\tau(S)$ of an ideal triangulation, which contains both  the Chekhov-Fock algebra $\T_\tau^q(S)$ and the Kauffman bracket skein algebra $\SSS(S)$. 

The introduction of this balanced Chekhov-Fock algebra in \cite{Hiatt} is grounded in the following two geometric facts. The first one is that, if the  character $[\bar r] \in \XP(S)$  is associated to an edge weight system $a\in \left(\C^*\right)^e$ for  $\tau$ as in \S \ref{subsect:EdgeWeightsGiveCharacters}, the trace $\Tr \bar r(\alpha)$ of an element $\alpha \in \pi_1(S)$ can be expressed as an explicit Laurent polynomial in the \emph{square roots} of the coordinates $a_i$ of $a$. The second fact is that, algebraically, these square roots are not well-behaved under changes of the ideal triangulation $\tau$; for instance,  in the case considered in Equation (\ref{eqn:ShearbendEmbeddedDiagEx}), solving for $\sqrt{a_i'}$ as a function of the $\sqrt{a_j}$ yields an expression that may involve terms $\sqrt{1+ \sqrt{a_{i_0}}^2}$ and consequently may not be rational. The first fact leads us to consider formal square roots $Z_i$ for the generators $X_i$ of the Chekhov-Fock algebra $\T_\tau^q(S)$. The technical difficulties arising from the second fact are addressed by introducing parity constraints. 
 
 More precisely, we first choose a $4$--root $q^{\frac14}$ of $q$ and, for every ideal triangulation $\tau$ of $S$, we consider the Chekhov-Fock algebra $\T^{q^{\frac14}}_\tau(S) = \C \left[ Z_1^{\pm1}, Z_2^{\pm1}, \dots, Z_e^{\pm1} \right] ^{q^{\frac14}}_\tau$ defined by generators $Z_i$ associated to the edges of $\tau$ and by the skew-commutativity relations
$$
Z_i Z_j = \big( q^{\frac14} \big) ^{2\sigma_{ij}}  Z_j Z_i 
$$
where  $\sigma_{ij} \in \Z$ is the same integer coefficients as in the definition of $\T^q_\tau(S)$, and is determined by the combinatorics of the ideal triangulation $\tau$. 

In particular, there is a natural embedding $\T^q_\tau(S) \to \T^{q^{\frac14}}_\tau(S) $  sending  each generator $X_i \in \T^q_\tau(S)$ to the square $Z_i^2$ of the generator $Z_i \in  \T^{q^{\frac14}}_\tau(S)$ associated to the same edge of $\tau$. 

The \emph{balanced Chekhov-Fock algebra} of the ideal triangulation $\tau$ is the subalgebra $\ZZ_\tau(S)$ of  $\T^{q^{\frac14}}_\tau(S) $ that, as a vector space, is spanned by the monomials $Z_1^{n_1} Z_2^{n_2} \dots Z_e^{n_e}$ where the exponents $n_1$, $n_2$, \dots, $n_e\in \Z$ satisfy the following condition: Whenever a triangle of $\tau$ is bounded by the $i$--th, $j$--th and $k$--th edges of $\tau$ (with $i$, $j$, $k$ not necessarily distinct), then the sum $n_i+n_j+n_k$ is even. In particular, this balanced Chekhov-Fock algebra $\ZZ_\tau(S)$ contains the original Chekhov-Fock algebra $\T^q_\tau(S)$, considered as a subset of $\T^{q^{\frac14}}_\tau(S)$. 

The balanced Chekhov-Fock algebra $\ZZ_\tau(S)$ also contains the Kauffman bracket skein algebra $\SSS(S)$. Indeed, the quantum trace homomorphism of \cite{BonWon1} provides an injective homomorphism $\mathrm{Tr}_\tau^q \colon \SSS(S) \to \ZZ_\tau(S)$.

Hiatt \cite{Hiatt} (see also comments in \cite[\S 7]{BonWon1} and \cite[\S 5.2]{BonWon5}) proved that the Chekhov-Fock coordinate change isomorphisms $\Phi_{\tau\tau'}^q(S) \colon \widehat\T^q_{\tau'}(S) \to \widehat\T^q_\tau(S)$ admit extensions $\Phi_{\tau\tau'}^q(S) \colon \widehat{\mathcal Z}^{q^{\frac14}}_{\tau'}(S) \to \widehat{\mathcal Z}^{q^{\frac14}}_\tau(S)$, such that 
$$
 \Phi_{\tau \tau''}^q = \Phi_{\tau \tau'}^q \circ \Phi_{\tau' \tau''}^q
 $$
 for every three ideal triangulations $\tau$, $\tau'$, $\tau''$. (In spite of what the notation might suggest, this extension of $\Phi_{\tau \tau'}^q $ to $ \widehat{\mathcal Z}^{q^{\frac14}}_{\tau'}(S)$ depends on the $4$--root $q^{\frac14}$, not just on $q$.)

 The  balanced Chekhov-Fock algebra $\ZZ_\tau(S)$ is not very different from the Chekhov-Fock algebra  $\T^{q^{\frac14}}_\tau(S)$, and its irreducible representations are classified by a statement analogous to Proposition~\ref{prop:RepsCheFock}. However, there is a price to pay  if, in order to prove Theorem~\ref{thm:CheFockAndKauffmanIntertwiners}, we choose the 4--root $q^{\frac14}$ so that $\big(q^{\frac14}\big)^{2n} = \big(q^{\frac12}\big)^n =-1$. The complex edge weights of Proposition~\ref{prop:RepsCheFock} need to be replaced by  more complicated data, the twisted cocycles twisted by the Thurston intersection form that occur in \cite[Props. 14--15]{BonWon4}. Instead of stating this general classification result for irreducible representations of  $\ZZ_\tau(S)$, we will give a technical result that is adapted to our goals.

 As in Proposition~\ref{prop:RepsCheFock}, we again have for each puncture $v$ of $S$ a preferred element 
$$
K_v = q^{-\frac14\sum_{1 \leq k < k' \leq j}  \sigma_{kk'} }  Z_{i_1} Z_{i_2} \dots Z_{i_j} \in \ZZ_\tau(S)
$$
where  $\gamma_{i_1}$, $\gamma_{i_2}$, \dots, $\gamma_{i_j}$ are the edges  of $\tau$ that are adjacent to $v$. This puncture element is related to that of  Proposition~\ref{prop:RepsCheFock} by the property that 
$$
K_v^2 = H_v \in \T_\tau^q(S) \subset \ZZ_\tau(S). 
$$
 
 The following statement follows from the combination of Propositions~22 and 23 in \cite{BonWon4}. 
 
\begin{prop} 
\label{prop:RepsBalancedCheFock}
Let $q$ be a primitive $n$--root of unity with $n$ odd, and choose a $4$--root $q^{\frac14}$ such that $\big( q^{\frac14} \big)^{2n}=-1$. Suppose that we are given:
\begin{itemize}
 \item a character $[r] \in \XX(S)$ whose projection $[\bar r] \in \XP(S)$ is associated to an edge weight system $a\in \left( \C^* \right)^e$ for the ideal triangulation $\tau$;
 \item for each puncture $v$, an $n$--root $k_v$ of the preferred eigenvalue $\lambda_v$ of $r(\alpha_v) \in \SL$ described in Lemma~{\upshape\ref{lem:PunctureEigenvalueShearbendParam}} and Remark~{\upshape\ref{rem:PunctureEigenvalueShearbendParam}}.
\end{itemize}
Suppose in addition that the character $[r]$ admits no nontrivial sign-reversal symmetry. Then, up to isomorphism, there is a unique irreducible representation $\breve \rho \colon \ZZ_\tau(S) \to \End(\breve V)$ such that:
\begin{enumerate}
\item if $Z_i$ is the generator of $\T_\tau^{q^{\frac14}}(S)$  associated to the $i$--th edge of $\tau$ and if $a_i \in \C^*$ is the $i$--th coordinate of the edge weight system $a\in \left( \C^* \right)^e$, then $\breve \rho(Z_i^2) = a_i \, \Id_{\breve V}$ for every $i=1$, $2$, \dots, $e$;
\item for every puncture $v$ of $S$ and for the  element $K_v \in \ZZ_\tau(S)$ associated to $v$ as above, $\breve\rho (K_v) = k_v \, \Id_{\breve V}$;

\item for the quantum trace homomorphism $\mathrm{Tr}_\tau^q \colon \SSS(S) \to \ZZ_\tau(S)$, the representation $\breve \rho \circ \mathrm{Tr}_\tau^q \colon \SSS(S) \to \End(\breve V)$ has classical shadow equal to the given character $[r] \in \XX(S)$, in the sense of Theorem~{\upshape\ref{thm:ClassicalShadow}}.  
 
\end{enumerate}

In addition, the representation has dimension $\dim \breve V = n^{3g+p-3}$ if the surface $S$ has genus $g$ and $p$ punctures.   \qed

\end{prop}

We now have all the ingredients needed to prove  Theorem~\ref{thm:CheFockAndKauffmanIntertwiners}. 

\begin{proof}
 [Proof of Theorem~\ref{thm:CheFockAndKauffmanIntertwiners}]
 
 Let us first summarize the data of Theorem~\ref{thm:CheFockAndKauffmanIntertwiners}. We have a $\phi$--invariant irreducible character $[r] \in \XX(S)$ that  admits no nontrivial sign-reversal symmetry, and whose projection $[\bar r] \in \XP(S)$ is associated to an edge weight system for an ideal triangulation sweep from $\tau$ to $\phi(\tau)$. 
 The data of $[ r] \in \XX(S)$ and of $\phi$--invariant compatible puncture weights $p_v \in \C$ define an irreducible representation $\rho \colon \SSS(S) \to \End(V)$ as in Theorem~\ref{thm:ClassicalShadow}, and Proposition~\ref{prop:KauffmanIntertwinerDefined} then provides an isomorphism $\Lambda_{\phi, r}^q \colon V \to V$ between the representations $\rho \circ \phi_*$ and $\rho$.  The edge weight system and a data of $\phi$--invariant compatible puncture weights $h_v \in \C^*$ define an irreducible representation $\bar \rho \colon \T_\tau^q(S) \to \End(\bar V)$, and Proposition~\ref{prop:QuantumTeichmullerIntertwinerDefined} now provides an isomorphism $\bar \Lambda_{\phi, \bar r}^q \colon \bar V \to \bar V$ between the representations $\bar \rho \circ \Phi_{\tau \phi(\tau)}^q  \circ \Psi_{\phi(\tau) \tau}^q$ and $\bar \rho$.  
 
We want to show that, under the hypothesis that $p_v^2 = h_v + h_v^{-1} + 2$ for every puncture $v$, the two intertwiners $ \Lambda_{\phi,  r}^q $ and $ \bar \Lambda_{\phi, \bar r}^q $ are conjugate, up to scalar multiplication by a root of unity.

 It is time to reveal what underlies this condition that $p_v^2 = h_v + h_v^{-1} + 2$. In Theorem~\ref{thm:ClassicalShadow}, the puncture weight $p_v$ is constrained by the property that $T_n(p_n) = - \Tr r(\alpha_v)$ for the Chebyshev polynomial $T_n$ and  for an element $\alpha_v \in \pi_1(S)$ represented by a small loop going once around the puncture.  In Proposition~\ref{prop:RepsCheFock}, the constraint on $h_v$ is that $h_v^n$ must be equal to a quantity which, by Lemma~\ref{lem:PunctureEigenvalueShearbendParam} and Remark~\ref{rem:PunctureEigenvalueShearbendParam},  turns out to be equal to $\lambda_v^2$ for an eigenvalue $\lambda_v$ of $r(\alpha_v) \in \SL$ singled out by the presentation of $[\bar r]$ in terms of an edge weight system for the ideal triangulation $\tau$. Therefore, for the eigenvalue $\lambda_v$ determined by the geometric setup, $p_v$ and $h_v$ are constrained by the equations $T_n(p_v) = - \lambda_v - \lambda_v^{-1}$ and $h^n = \lambda_v^2$. By an elementary property of the Chebyshev polynomial already mentioned in \S \ref{subsect:VolConjSurfaceDiffeos} (see for instance \cite[Lem. 17]{BonWon4}), the solutions of the first equation are the numbers of the form $p_v = -k_v -k_v^{-1}$ as $k_v$ ranges over all $n$--roots of $\lambda_v$. A simple algebraic manipulation (slightly more elaborate when $\lambda_v=\pm1$ or $\pm\I$) then shows that the equation $p_v^2 = h_v + h_v^{-1} + 2$ is equivalent to the existence of an $n$--root $k_v$ of $\lambda_v$ such that $p_v = -k_v -k_v^{-1}$ and $h_v = k_v^2$.

Therefore, for every puncture $v$, we have an $n$--root $k_v$ of the eigenvalue $\lambda_v$, such that $p_v = -k_v -k_v^{-1}$ and $h_v = k_v^2$. We also have a character $[r]\in\XX(S)$ whose projection $[\bar r] \in \XX(S)$ is associated to an edge weight system $a\in \left( \C^* \right)^e$ for the ideal triangulation $\tau$. Proposition~\ref{prop:RepsBalancedCheFock} associates to this data a representation $\breve \rho \colon \ZZ_\tau(S) \to \End(\breve V)$. 

We first consider the restrictions of $\breve \rho$ to the Kauffman bracket skeing algebra $\SSS(S)$ and to the Chekhov-Fock algebra $\T_\tau^q(S)$. 

\begin{claim}
\label{claim:CheFockAndKauffmanIntertwiners1}
 The restriction $\breve \rho_{| \T_\tau^q(S)} \colon \T_\tau^q(S) \to \End(\breve V)$ is isomorphic to the representation $\bar \rho  \colon \T_\tau^q(S) \to \End(\bar V)$. 
\end{claim}

\begin{proof}
The representation  $\breve \rho_{| \T_\tau^q(S)} $ is irreducible, by comparison of the dimensions in Propositions~\ref{prop:RepsCheFock} and \ref{prop:RepsBalancedCheFock}. 

Because the generator $X_i \in \T_\tau^q(S)$ corresponds to $Z_i^2 \in \Z_\tau(S)$, $\breve \rho$ sends $X_i$ to $a_i \, \Id_{\breve V}$, where $a_i$ if the $i$--coordinate of $a\in \left( \C^* \right)^e$. 

Also, the puncture element $H_v \in \T_\tau^q(S)$ is equal to $K_v^2$ for the puncture element $K_v \in \ZZ_\tau(S)$, and is therefore sent to $k_v^2 \,\Id_{\breve V} = h_v \, \Id_{\breve V}$ by $\breve\rho$. It follows that the irreducible representation of $\T_\tau^q(S)$ provided by the restriction of $\rho$ is classified by the same invariants as $\bar \rho$. By Proposition~\ref{prop:RepsCheFock}, these two representations are therefore isomorphic. 
\end{proof}

We can also consider, for the embedding $\mathrm{Tr}_\tau^q \colon \SSS(S) \to \ZZ_\tau(S)$, the restriction  $\breve \rho \circ \mathrm{Tr}_\tau^q \colon \SSS(S) \to \End(\breve V)$ of $\breve\rho$ to $\SSS(S)$. 

\begin{claim}
\label{claim:CheFockAndKauffmanIntertwiners2}
 The restriction $\breve \rho \circ \mathrm{Tr}_\tau^q \colon \SSS(S) \to \End(\breve V)$  is isomorphic to the representation $ \rho  \colon\SSS(S) \to \End( V)$. 
\end{claim}

\begin{proof}
The classical shadow of $\breve \rho \circ \mathrm{Tr}_\tau^q$ is equal to $[r]\in\XX(S)$, by the third conclusion of Proposition~\ref{prop:RepsBalancedCheFock}. 

For a puncture $v$, its invariant is defined by considering an element $[P_v] \in \SSS(S)$ represented by a simple closed loop going around the puncture. Lemma~18 of \cite{BonWon4} shows that  $\mathrm{Tr}_\tau^q \big( [P_v] \big) = K_v + K_v^{-1} $ in $ \ZZ_\tau(S)$. Therefore,
$$
\breve \rho \circ \mathrm{Tr}_\tau^q \big( [P_v]  = \breve \rho (K_v + K_v^{-1} )
=- (k_v + k_v^{-1}) \Id_{\breve V}=p_v  \Id_{\breve V}.
$$

As a consequence, the representations $\breve \rho \circ \mathrm{Tr}_\tau^q$ and $\rho$ of $\SSS$ have the same dimension $n^{3g+p-3}$, the same classical shadow $[r] \in \XX(S)$ and the same puncture invariants $p_v\in \C$. Theorem~\ref{thm:SkeinRepUnique} then shows that they are isomorphic. 
\end{proof}

We have not yet used the diffeomorphism $\phi \colon S \to S$, the  ideal triangulation sweep $\tau = \tau^{(0)}$, $ \tau^{(1)}$, \dots, $ \tau^{(k_0-1)}$, $ \tau^{(k_0)}=\phi(\tau)$, or the periodic edge weight system $a= a^{(0)}$, $ a^{(1)}$, \dots, $ a^{(k_0-1)}$, $ a^{(k_0)}= a\in \left(\C^* \right)^e$. 

Let $\Phi_{\tau\phi(\tau)}^q \colon \widehat{\mathcal{Z}}^{q^{\frac14}}_{\phi(\tau)}(S) \to \widehat{\mathcal{Z}}^{q^{\frac14}}_\tau(S)$ be the Chekhov-Fock-Hiatt coordinate change isomorphism. 
By Lemma~28 of \cite{BonWon5} applied to each diagonal exchange, the existence of the edge weight system guarantees that the representation $\breve\rho \circ \Phi_{\tau\phi(\tau)}^q  \colon \ZZ_{\phi(\tau)}(S) \to \End(\breve V)$ makes sense. As in the definition of the intertwiner $\bar \Lambda_{\phi, \bar r}^q$, the diffeomorphism $\phi$ also induces an isomorphism $\Psi_{\phi(\tau) \tau}^{q^{\frac14}} \circ  \T_\tau^{q^{\frac14}} (S)\to  \T_{\phi(\tau)}^{q^{\frac14}} (S)$ which restricts to an isomorphism $\Psi_{\phi(\tau) \tau}^{q^{\frac14}} \colon \ZZ_\tau (S)\to  \ZZ_{\phi(\tau)} (S)$ between the corresponding balanced Chekhov-Fock algebras. We can then consider the representation $\breve\rho \circ \Phi_{\tau\phi(\tau)}^q \circ \Psi_{\phi(\tau) \tau}^{q^{\frac14}} \colon \ZZ_{\tau}(S) \to \End(\breve V)$. 

\begin{claim}
\label{claim:CheFockAndKauffmanIntertwiners3}
 The representations $\breve\rho \circ \Phi_{\tau\phi(\tau)}^q \circ \Psi_{\phi(\tau) \tau}^{q^{\frac14}}$ and $\breve\rho  \colon \ZZ_{\tau}(S) \to \End(\breve V)$ are isomorphic. 
\end{claim}
\begin{proof} 
Because the square $Z_i^2 \in  \ZZ_\tau(S)$ is also the $i$--th generator of the Chekhov-Fock algebra $\T_\tau^q(S)$, the argument that we used in \S \ref{subsect:CheFockIntertwinerInvariantDiffeomorphism} for $\bar \rho$, using Proposition~\ref{prop:CheFockRepsCompatibleIdealTriangChange} and the periodicity of the edge weight system, show that 
$$\breve\rho \circ \Phi_{\tau\phi(\tau)}^q \circ \Psi_{\phi(\tau) \tau}^{q^{\frac14}}(Z_i^2) = a_i \, \Id_{\breve V}.$$

Also, for every puncture $v$, the homomorphism $\Psi_{\phi(\tau) \tau}^{q^{\frac14}} \circ  \T_\tau^{q^{\frac14}} (S)\to  \T_{\phi(\tau)}^{q^{\frac14}} (S)$ induced by $\phi$ clearly sends the puncture element $K_v \in  \ZZ_\tau(S)$ to $K_{\phi(v)}  \ZZ_{\phi(\tau)}(S) $, while Lemma~29 of  \cite{BonWon5} shows that the Chekhov-Fock-Hiatt coordinate change isomorphism $\Phi_{\tau\phi(\tau)}^q \colon \widehat{\mathcal{Z}}^{q^{\frac14}}_{\phi(\tau)}(S) \to \widehat{\mathcal{Z}}^{q^{\frac14}}_\tau(S)$ sends $K_{v'} \in {\mathcal{Z}}^{q^{\frac14}}_{\phi(\tau)}(S) $ to $K_{v'} \in {\mathcal{Z}}^{q^{\frac14}}_\tau(S)$. It follows that
$$
\breve\rho \circ \Phi_{\tau\phi(\tau)}^q \circ \Psi_{\phi(\tau) \tau}^{q^{\frac14}}(K_v)
=\breve\rho \circ \Phi_{\tau\phi(\tau)}^q (K_{\phi(v)})
=\breve\rho (K_{\phi(v)}) = k_{\phi(v)} \, \Id_{\breve V} = k_v \, \Id_{\breve V} ,
$$
using the $\phi$--invariance of the puncture weights $k_v$. 

Finally, a fundamental property of the quantum trace homomorphism $\mathrm{Tr}_\tau^q \colon \SSS(S) \to \ZZ_\tau(S)$ is that it is well behaved with the  Chekhov-Fock-Hiatt coordinate change isomorphisms, in the sense that $ \mathrm{Tr}_{\tau}^q  =  \Phi_{\tau\tau'}^q \circ \mathrm{Tr}_{\tau'}^q  $ for any two ideal triangulations $\tau$, $\tau'$; see Theorem~28 in \cite{BonWon1}. A more immediate property is that it also behaves well with respect to the homomorphism $\Psi_{\phi(\tau) \tau}^{q^{\frac14}} \circ  \T_\tau^{q^{\frac14}} (S)\to  \T_{\phi(\tau)}^{q^{\frac14}} (S)$ induced by $\phi$. Indeed, it immediately follows from the construction of $\mathrm{Tr}_{\tau}^q $ in \cite{BonWon1} that $ \Psi_{\phi(\tau) \tau}^{q^{\frac14}} \circ \mathrm{Tr}_{\tau}^q = \mathrm{Tr}_{\phi(\tau)}^q  \circ \phi_*$ where $\phi_* \colon \SSS(S) \to \SSS(S)$ is the homomorphism induced by $\phi$. Then,
$$
\breve\rho \circ \Phi_{\tau\phi(\tau)}^q \circ \Psi_{\phi(\tau) \tau}^{q^{\frac14}} \circ \mathrm{Tr}_{\tau}^q
= \breve\rho \circ \Phi_{\tau\phi(\tau)}^q \circ \mathrm{Tr}_{\phi(\tau)}^q \circ \phi_*
= \breve\rho \circ \mathrm{Tr}_{\tau}^q \circ \phi_*
$$
Looking at the definition of the classical shadow in \cite{BonWon3}, it immediately follows that, since $\breve\rho \circ \mathrm{Tr}_{\tau}^q$ has classical shadow $[r] \in \XX(S)$, then $\breve\rho \circ \Phi_{\tau\phi(\tau)}^q \circ \Psi_{\phi(\tau) \tau}^{q^{\frac14}} \circ \mathrm{Tr}_{\tau}^q$ has classical shadow $[r \circ \phi_*] \in \XX(S)$ where, this time, $\phi_* \colon \pi_1(S) \to \pi_1(S)$ denotes an arbitrary  isomorphism of $\pi_1(S)$ induced by $\phi$. Since $[r]$ is $\phi$--invariant by hypothesis, $[r \circ \phi_*] = [r]$ and we conclude that $\breve\rho \circ \Phi_{\tau\phi(\tau)}^q \circ \Psi_{\phi(\tau) \tau}^{q^{\frac14}} \circ \mathrm{Tr}_{\tau}^q$ and $ \breve\rho \circ \mathrm{Tr}_{\tau}^q $.

Therefore, the representations $\breve\rho \circ \Phi_{\tau\phi(\tau)}^q \circ \Psi_{\phi(\tau) \tau}^{q^{\frac14}} $ and $ \breve\rho  $ have the same invariants in the sense of Proposition~\ref{prop:RepsBalancedCheFock}. By our assumption that $[r]$ has no nontrivial sign-reversal symmetry, we can  apply  that statement  and conclude that the two representations  are isomorphic. 
\end{proof}

We are now ready to conclude the proof of Theorem~\ref{thm:CheFockAndKauffmanIntertwiners}. 

By Claim~\ref{claim:CheFockAndKauffmanIntertwiners1},  the restriction $\breve \rho_{| \T_\tau^q(S)} \colon \T_\tau^q(S) \to \End(\breve V)$ is isomorphic to the representation $\bar \rho  \colon \T_\tau^q(S) \to \End(\bar V)$. 
Modifying $\bar \rho$ by the corresponding isomorphisms (which only changes the intertwiner $\bar\Lambda_{\phi, \bar r}^q$ by a conjugation), we can therefore assume that $\bar V = \breve V$ and that $\bar \rho$ is the restriction of $\breve \rho$ to the Chekhov-Fock algebra $\T_\tau^q(S) \subset \ZZ_\tau(S)$.

Similarly, we can use  Claim~\ref{claim:CheFockAndKauffmanIntertwiners2} to arrange that the original representation $ \rho  \colon\SSS(S) \to \End( V)$ of the skein algebra is equal to 
$\breve \rho \circ \mathrm{Tr}_\tau^q \colon \SSS(S) \to \End(\breve V)$. 

By Claim~\ref{claim:CheFockAndKauffmanIntertwiners3}, there exists an intertwiner $\breve \Lambda_{\phi, r}^q \colon \breve V \to \breve V $  such that 
$$
\breve\rho \circ \Phi_{\tau\phi(\tau)}^q \circ \Psi_{\phi(\tau) \tau}^{q^{\frac14}} (Z)
=
\breve \Lambda_{\phi, r}^q \cdot  \breve \rho(\epsilon Z)     \cdot \breve \Lambda_{\phi, r}^q 
$$
for every $Z \in \ZZ_\tau(S) $. 

Since  we arranged that the representations $\bar\rho$ and $\breve \rho$ coincide on $\T_\tau^q(S)$, 
$$
\bar \rho \circ \Phi_{\tau\phi(\tau)}^q \circ \Psi_{\phi(\tau) \tau}^q (X)
=
\breve\rho \circ \Phi_{\tau\phi(\tau)}^q \circ \Psi_{\phi(\tau) \tau}^{q^{\frac14}} (X)
= \breve \Lambda_{\phi, r}^q \cdot  \breve \rho(\epsilon X)     \cdot \breve \Lambda_{\phi, r}^q 
= \breve \Lambda_{\phi, r}^q \cdot  \bar \rho( X)     \cdot \breve \Lambda_{\phi, r}^q  
$$
for every $X\in \T_\tau^q(S)$. This is the same intertwining property satisfied by $\bar  \Lambda_{\phi, \bar r}^q $ and, since $\bar \rho$ is irreducible, it follows that $\breve  \Lambda_{\phi, r}^q $ is a scalar multiple of $\bar  \Lambda_{\phi, \bar r}^q $. Since the determinants of both isomorphisms have modulus 1, this scalar must be also have modulus 1. 

The same argument applied to the other restriction $\rho= \breve \rho \circ \mathrm{Tr}_\tau^q$  shows that $\breve  \Lambda_{\phi, r}^q $ is obtained by multiplying $\bar  \Lambda_{\phi, \bar r}^q $ by a scalar with modulus 1. 

This shows that the intertwiners $  \Lambda_{\phi,  r}^q $  and $\bar  \Lambda_{\phi, \bar r}^q $ coincide up to  multiplication by a scalar with modulus 1, which concludes the proof of Theorem~\ref{thm:CheFockAndKauffmanIntertwiners}. 
\end{proof}

 \section{The case of the one-puncture torus}
 \label{sect:IntertwinersPunctTorus}
 
  We now carry out the program of \S \ref{sect:ComputeIntertwinerWithCheFock} in the special case where the surface $S$ is the one-puncture torus $S_{1,1}$. We want to explicitly compute $\left| \Tr \Lambda_{\phi, r}^q \right|$ for the interwiner  $\Lambda_{\phi, r}^q$ associated by Proposition~\ref{prop:KauffmanIntertwinerDefined} to an orientation preserving diffeomorphism $\phi \colon S_{1,1} \to S_{1,1}$ equipped with additional data. Because of Theorem~\ref{thm:CheFockAndKauffmanIntertwiners}, we will actually compute instead the trace of the intertwiner $\bar \Lambda_{\phi, \bar r}^q$ of Proposition~\ref{prop:QuantumTeichmullerIntertwinerDefined}, using the technology of Chekhov-Fock algebras. 
  
  Similar computations already appeared in \cite{Liu2}, but the conventions we use here are more symmetric and more explicit. They are also better suited for our purposes in \cite{BWY2, BWY3}.

  \subsection{Chekhov-Fock algebras for the one-puncture torus}
  \label{subsect:ChekhovFockPunctTorus}

An ideal triangulation $\tau$ of the one-puncture torus $S_{1,1}$ necessarily consists of three edges  and two triangles. Changing the notation from earlier sections, we label these edges as $e$, $f$ and $g$ so that they occur \emph{clockwise} in this order around both triangles. 

As a consequence of its definition, the Chekhov-Fock algebra $\T^q_\tau(S_{1,1})$ is then isomorphic to the algebra $\C\left[X^{\pm1}, Y^{\pm1}, Z^{\pm1}\right]^q$ of Laurent polynomials in  variables $X$, $Y$ and $Z$,  respectively associated to the edges $e$, $f$, $g$, satisfying the $q$--commutativity relations 
\begin{align*}
 XY&=q^4 YX &
 YZ&=q^4 ZY &
 ZX&=q^4 ZX.
\end{align*}
The clockwise order of $e$, $f$, $g$ around the faces of $\tau$ is here critical. 

This will enable us to express our computations in terms of  the single abstract algebra $\T^q=\C\left[X^{\pm1}, Y^{\pm1}, Z^{\pm1}\right]^q$ defined by the above relations. In particular, we will now insist that an ideal triangulation $\tau$ of $S_{1,1}$ comes with a labelling of its edges as $e$, $f$, $g$, appearing clockwise in this order around the faces of $\tau$. This  specifies a canonical isomorphism 
$$\theta_\tau \colon \T^q \to \T_\tau^q(S_{1,1})$$ 
from the abstract algebra $\T^q =\C\left[X^{\pm1}, Y^{\pm1}, Z^{\pm1}\right]^q$ to the Chekhov-Fock algebra $ \T_\tau^q(S_{1,1})$, sending the generators $X$, $Y$ and $Z$ to the generators respectively associated to the edges $e$, $f$, $g$.

\subsection{Standard representations of the algebra $\T^q$}
\label{subsect:StandardReps}

 We now restrict attention to the case where $q$ is a primitive $n$--root of unity with $n$ odd. 

The \emph{standard representation}  of the algebra $\T^q=\C\left[X^{\pm1}, Y^{\pm1}, Z^{\pm1}\right]^q$ associated to the nonzero numbers $x$, $y$, $z\in \C^*$  is the representation $\rho_{xyz} \colon \T^q \to \End(\C^n)$ where, if $\{ w_1, w_2, \dots, w_n \}$ is the standard basis for $\C^n$, 
\begin{align*}
\rho_{xyz} (X) (w_i) &= x   q^{4i} w_i\\
\rho_{xyz} (Y) (w_i) &= y  q^{-2i}w_{i+1}\\
\rho_{xyz} (Z) (w_i) &= z  q^{-2i} w_{i-1}
\end{align*}
for every $i=1$, $2$, \dots, $n$ and counting indices modulo $n$. 

Note that $\rho_{xyz}(X^n) = x^n\, \Id_{\C^n}$, $\rho_{xyz}(Y^n) = y^n \,\Id_{\C^n}$, $\rho_{xyz}(Z^n) = z^n \,\Id_{\C^n}$ and $\rho_{xyz} \big( [XYZ] \big)=xyz\, \Id_{\C^n}$ for the Weyl quantum ordering $[XYZ]=q^{-2} XYZ$ of the monomial $XYZ$. 

The following property is easily proved by elementary linear algebra (see also \cite[\S 4]{BonLiu}). Note that it crucially uses the fact that $n$ is odd. 

\begin{prop}
\label{prop:StandardRep}$ $
\begin{enumerate}
\item Every standard representation $\rho_{xyz} $ is irreducible. 
\item Every irreducible representation of $\T^q$ is isomorphic to a standard representation $\rho_{xyz} $. 
\item Two standard representation $\rho_{xyz} $ and $\rho_{x'y'z'} $  are isomorphic if and only if $x^n=x^{\prime n}$, $y^n=y^{\prime n}$, $z^n=z^{\prime n}$ and $xyz=x'y'z' $. \qed
\end{enumerate}
\end{prop}

As a consequence, an irreducible representation $\rho \colon \T^q \to \End(V)$ is classified, up to isomorphism, by the numbers $a$, $b$, $c$, $h\in \C^*$ such that $\rho(X^n) = a\, \Id_V$, $\rho(Y^n) = b \,\Id_V$, $\rho(Z^n) = c \,\Id_V$ and $\rho \big( [XYZ] \big)=h\, \Id_V$. Note that these four invariants $a$, $b$, $c$, $h$ are tied by the relation $h^n = abc$, but that this is the only constraint between them. To compare with Proposition~\ref{prop:RepsCheFock}, note that the puncture invariant $h_v$ arising there is equal to $h^2$, and that $h$ is the unique square root of $h_v$ such that $h^n = abc$. .

 \subsection{The discrete quantum dilogarithm}
 \label{subsect:DiscreteQuantumDilog}
  We introduce a fundamental building block for our computations of intertwiners.

 The \emph{discrete quantum dilogarithm}, or \emph{Fadeev-Kashaev quantum dilogarithm} \cite{FadKash}, is the function of the integer $i$ defined by
$$
\QDL^q(u,v \vbar  i) =  \prod_{k=1}^i \frac{1 + u q^{-2k}}{v} = v^{-i} \prod_{k=1}^i (1+ u q^{-2k}),
$$
where the parameters $u$, $v \in \C$ are such that $v^n = 1 + u^n \neq 0$, and where $q$ is a primitive $n$--root of unity with $n$ odd. 

An elementary computation provides the following periodicity property. 
\begin{lem}
\label{lem:DiscreteQuantumDilogPeriodic}
\begin{equation*}
 \QDL^q(u, v \vbar  i+n) = \QDL^q(u,v \vbar  i). 
\end{equation*}
for every index $i$.  \qed
\end{lem}

In particular, Lemma~\ref{lem:DiscreteQuantumDilogPeriodic} enables us to define  $\QDL^q(u,v \vbar  i)$ for every $i\in \Z$.

 \subsection{The left and right diffeomorphisms,  isomorphisms and intertwiners}
\label{subsect:LeftRightIntertwiners}

It is well known that, up to isotopy, a diffeomorphism of $S_{1,1}$ is completely determined by its action on the homology group $H_1(S_{1,1})$. If we choose an isomorphism $H_1(S_{1,1})\cong \Z^2$, this provides us with an identification between the mapping class group $\pi_0\, \mathrm{Diff}^+(S_{1,1})$ and the group $\mathrm{SL}_2(\Z)$. 
Once such an identification is chosen, the following  elements
\begin{align*}
L&= \left( 
\begin{matrix}
1&1\\0&1
\end{matrix}
\right)
&
R&= \left( 
\begin{matrix}
1&0\\1&1
\end{matrix}
\right)
&
J&= \left( 
\begin{matrix}
-1&0\\ 0&-1
\end{matrix}
\right)
\end{align*}
of  $\pi_0\, \mathrm{Diff}^+(S_{1,1}) \cong \mathrm{SL}_2(\Z)$ will play a fundamental role in our computation. The \emph{left} and \emph{right} elements $L$ and $R$ are well-known generators of $ \mathrm{SL}_2(\Z)$, with the terminology coming from their action on the Farey tessellation.

The following isomorphisms of the fraction algebra $\widehat \T^q$ will play an important role in our computations. We will see in \S \ref{subsect:IntertwinerCompElementaryIntertwiner} that they are closely related to the above generators $L$ and $R$ of $\mathrm{SL}_2(\Z) \cong \pi_0 \, \Diff^+(S_{1,1})$. 

The \emph{left isomorphism} $\mathcal L \colon \widehat \T^q \to \widehat\T^q$ and  \emph{right isomorphism} $\mathcal R  \colon \widehat \T^q \to \widehat\T^q$ are  defined by 
\begin{align*}
\mathcal L(X) &= Y^{-1}
&
\mathcal R(X) &= Z^{-1}
\\
\mathcal L(Y) &= (1+qY)  (1+q^3Y) X
&
\mathcal R(Y) &= (1+qZ)  (1+q^3Z) Y
\\
\mathcal L(Z) &= (1+qY^{-1})^{-1}  (1+q^3Y^{-1})^{-1} Z
&
\mathcal R(Z) &= (1+qZ^{-1})^{-1}  (1+q^3Z^{-1})^{-1} X
\end{align*}

For a standard representation $\rho_{x_1y_1z_1} \colon \T^q \to \End(\C^n)$, we want to determine the representations $\rho_{x_1y_1z_1} \circ \mathcal L$ and $\rho_{x_1y_1z_1} \circ \mathcal R$ of $\T^q$. 

\begin{prop}
\label{prop:LeftRight}
Suppose that $y_1^n \neq -1$. 
Then the  representation  $\rho_{x_1y_1z_1} \circ \mathcal L \colon \mathcal T^q \to \End(\C^n)$ makes sense,  and  is  isomorphic to any standard representation $\rho_{x_2y_2z_2}$ with 
\begin{align*}
x_2^{ n} &=  y_1^{-n}
&
y_2^{ n}&=(1+y_1^n)^{2} x_1^n \\
z_2^{ n} &= (1+y_1^{-n})^{-2} z_1^n
&
x_2y_2z_2  &= x_1 y_1 z_1.
\end{align*}

Similarly, if $z_1^n \neq -1$,   $\rho_{x_1y_1z_1} \circ \mathcal R $ makes sense and is isomorphic to  any $\rho_{x_3y_3z_3}$ with 
\begin{align*}
x_3^{ n} &= z_1^{-n}
&
y_3^{ n}&= (1+z_1^n)^{2} y_1^n \\
z_3^{ n} &=  (1+z_1^{-n})^{-2}  x_1^n
&
x_3y_3z_3  &= x_1y_1z_1 .
\end{align*}
\end{prop} 

\begin{proof}
This computation of the invariants of $\rho_{x_1y_1z_1} \circ \mathcal L$ and $\rho_{x_1y_1z_1} \circ \mathcal R$ can be found in \cite[Lemma~27]{BonLiu}. We already encountered this property in Proposition~\ref{prop:CheFockRepsCompatibleIdealTriangChange}. 
\end{proof}

We now explicitly determine the isomorphisms whose existence is abstractly predicted by Proposition~\ref{prop:LeftRight}. 

We first define, for $u$, $v\in \C$ with $v^n = 1 + u^n \neq 0$,  the quantity
\begin{align*}
D^q(u) &= \prod_{i=1}^n \QDL^q(u, v \vbar 2i) =  \prod_{j=1}^n \QDL^q(u, v \vbar  j)\\ 
&= v^{- \frac12 n(n+1)} \prod_{k=1}^{n} (1+ u q^{-2k})^{n-k+1} = (1+u^n)^{- \frac{n+1}2} \prod_{k=1}^{n} (1+ u q^{-2k})^{n-k+1}
\end{align*}
which depends only on $u$, $q$ and $n$. 

The \emph{left} and \emph{right intertwiners} are the linear isomorphisms $L_{uv}$ and $R_{uv} \colon \C^n \to \C^n$, depending on $u$, $v\in \C$ with $v^n = 1 + u^n \neq 0$, defined by 
$$
L_{uv} (w_j) = \frac {\QDL^q(u,v\vbar 2j) } {\left| D^q(u) \right|^{\frac1n}  \sqrt n }\sum_{i=1}^n { q^{-i^2 +j^2+4ij+i-j} }  w_i
$$
and
$$
R_{uv} (w_j) = \frac {\QDL^q(u,v \vbar 2j) } {\left| D^q(u) \right|^{\frac1n}  \sqrt n }\sum_{i=1}^n { q^{i^2 +3j^2-4ij+i-j} }  w_i
$$
for the standard basis $\{ w_1, w_2, \dots, w_n\}$ of $\C^n$. 

The normalization by the factor $\left| D^q(u) \right|^{\frac1n}  \sqrt n $ was introduced to achieve the following property. 

\begin{lem}
\label{lem:LeftRightIntertwinerDeterminant}
The determinants $\det L_{uv}$ and $\det R_{uv}$ of the above linear isomorphisms have modulus $1$. 
\end{lem}

\begin{proof}
 By definition of $D^q(u)$ and remembering that $\sum_{i=1}^n i^2 = \frac16 n(n+1)(2n+1)$,  
 $$\det R_{uv} = \frac{D^q(u)}{\left| D^q(u) \right|} \,q^{\frac23n(n+1)(2n+1)} \det A$$
 where $A$ is the matrix whose $ij$--entry is $A_{ij}= \frac 1{\sqrt n} q^{-4ij}$. 
 
 The matrix $A$ is a variation of the well-known discrete Fourier transform matrix. In particular, one easily computes that $A^4$ is the identity matrix. It follows that $\det A$ is a 4--root of unity, and in particular has modulus 1.  This proves that $\left| \det R_{uv} \right|=1$. 
 
A similar argument shows that $\left| \det L_{uv} \right|=1$. 
\end{proof}

We now have the following explicit version of Proposition~\ref{prop:LeftRight}. 

\begin{lem}
\label{lem:LeftRightIntertwiners}
Let  $x_1$, $y_1$ and $z_1 \in \C^*$ be given.
\begin{enumerate}
\item If $y_1^n \neq -1$, choose $u_2$,  $v_2\in \C$ such that $u_2=q y_1$ and $v_2^n = 1 + u_2^n$, and set 
\begin{align*}
 x_2 &=y_1^{-1}
&
y_2 &= v_2^2 x_1
&
 z_2 &= v_2^{-2} y_1^2 z_1.
\end{align*}
Then, for every $W\in \T^q$,
$$
( \rho_{x_1 y_1 z_1} \circ \mathcal L) (W)  = L_{u_2v_2} \circ  \rho_{x_2 y_2 z_2} (W)\circ  L_{u_2v_2}^{-1} .
$$

\item If $z_1^n \neq -1$, choose $u_3$,  $v_3\in \C$ such that $u_3= qz_1$ and $v_3^n = 1 + u_3^n$, and set 
\begin{align*}
x_3&=z_1^{-1} 
&
 y_3 &= v_3^2 y_1
 &
 z_3 &= v_3^{-2} z_1^2 x_1.
\end{align*}
Then, for every $W\in \T^q$, 
$$
( \rho_{x_1 y_1 z_1} \circ \mathcal R) (W)  = R_{u_3v_3} \circ  \rho_{x_3 y_3 z_3} (W)\circ  R_{u_3v_3}^{-1} .
$$
\end{enumerate}

\end{lem}

\begin{proof}
 This a tedious, but elementary, computation. In fact,  $L_{uv}$ and $R_{uv}$ were  precisely designed to realize the isomorphisms of Proposition~\ref{prop:LeftRight}. 
\end{proof}

 \subsection{The twist intertwiner}
\label{subsect:TwistIntertwiner}

If $x_1^n=x_2^n$, $y_1^n=y_2^n$, $z_1^n=z_2^n$ and  $x_1y_1z_1=x_2y_2z_2$, Proposition~\ref{prop:StandardRep} asserts that the standard representations $\rho_{x_1y_1z_1}$ and $\rho_{x_2y_2z_2}$ are isomorphic.   We determine the corresponding isomorphism $\C^n \to \C^n$. 

The formulas are a little simpler if we use the $n$--root of unity $q^4$, which is primitive since $n$ is odd. The property that  $x_1^n=x_2^n$, $y_1^n=y_2^n$, $z_1^n=z_2^n$ and  $x_1y_1z_1=x_2y_2z_2$  is then equivalent to the existence of $l_0$, $m_0$, $n_0\in \Z$  such that $x_2 = q^{4l_0} x_1$, $y_2 =q^{4m_0} y_1$, $z_2=q^{4n_0} z_1$ and  $l_0+m_0+n_0=0 \mod n$. 

Define $T_{l_0m_0n_0} \colon \C^n \to \C^n$ by the property that, for the standard basis $\{ w_1, w_2, \dots, w_n\}$ of~$\C^n$,
$$
T_{l_0m_0n_0}(w_j) =  q^{2j(n_0 -m_0)} w_{j+l_0} = \sum_{i=1}^n q^{2j(n_0-m_0)} \delta_{i, j+l_0} w_i 
$$
for the Kronecker symbol $\delta_{i,j} \in \{0,1\}$. 

\begin{lem}
\label{lem:TwistIntertwiner}
If $x_2 = q^{4l_0} x_1$, $y_2=q^{4m_0} y_1$ and $z_2=q^{4n_0} z_1$ for some $l_0$, $m_0$, $n_0\in \Z$ with $l_0 + m_0 + n_0=0$, then
$$
\rho_{x_1y_1z_1}(W) =T_{l_0m_0n_0} \circ  \rho_{x_2y_2z_2} (W) \circ T_{l_0m_0n_0}^{-1}.
$$
for every $W\in \T^q$. In addition, $\det T_{l_0m_0n_0}  =  1$. 
\end{lem}
\begin{proof}
 This is again an elementary computation. 
\end{proof}

\subsection{The intertwiner $\Lambda_{\phi, r}^q$ as a product of elementary intertwiners}
\label{subsect:IntertwinerCompElementaryIntertwiner} 

We  begin our computation of the intertwiner $ \Lambda_{\phi, r}^q \colon V \to V$ of Proposition~\ref{prop:KauffmanIntertwinerDefined}. Because of Theorem~\ref{thm:CheFockAndKauffmanIntertwiners}, we will actually compute the intertwiner $\bar \Lambda_{\phi, \bar r}^q \colon\bar V \to \bar V$ of Proposition~\ref{prop:QuantumTeichmullerIntertwinerDefined}. 

Our data consists of a $\PSL$--character $[\bar r ] \in \XP(S_{1,1})$ that is associated to an edge weight system $a= a^{(0)}$, $ a^{(1)}$, \dots, $ a^{(k_0-1)}$, $ a^{(k_0)}= a\in \left(\C^* \right)^e$ for an ideal triangulation sweep $\tau = \tau^{(0)}$, $ \tau^{(1)}$, \dots, $ \tau^{(k_0-1)}$, $ \tau^{(k_0)}=\phi(\tau)$. In particular, $[\bar r]$ is $\phi$--invariant. We are also given a puncture weight $h_v$, which is an  $n$--root of the square of the product of the three edge weights of $a^{(k)}$ and does not depend on $k$. 

We first take advantage of the small size of the surface $S_{1,1}$ to simplify the data, as well as the notation.

\begin{lem}
\label{lem:OrganizePhi}
 We can choose the isomorphism $H_1(S_{1,1}) \cong \Z^2$ (and the corresponding identification $\pi_0\, \Diff^+(S_{1,1}) \cong \mathrm{SL}_2(\Z)$) in such a way that:
 \begin{enumerate}
 \item The diffeomorphism $\phi$ can be written as a composition
  $$
 \phi = \phi_1\circ \phi_2 \circ \dots \circ  \phi_{k_0-1} \circ \phi_{k_0}\circ  J^\epsilon
 $$ 
 where each $\phi_k$ is one of the diffeomorphisms $L= \left( 
\begin{smallmatrix}
1&1\\0&1
\end{smallmatrix}
\right)$,  $R= \left( 
\begin{smallmatrix}
1&0\\1&1
\end{smallmatrix}
\right)$ of  \S {\upshape\ref{subsect:ChekhovFockPunctTorus}}, where $J= \left( 
\begin{smallmatrix}
-1&0\\0&-1
\end{smallmatrix}
\right)$, and where $\epsilon \in \{0,1\}$. 

\item Let $\tau_0$ be the ideal triangulation whose edges $e_0$, $f_0$, $g_0$ are respectively disjoint from closed curves representing the homology classes $(1,1)$, $(0,1)$, $(1,0 )$ in $H_1(S_{1,1}) \cong \Z^2$, and define an ideal triangulation $
\tau_k = \phi_1\circ \phi_2\circ \dots\circ \phi_k(\tau_0)
$  for each $k=1$, $2$, \dots, $k_0$. Then, the character $[\bar r ] \in \XP(S)$ is associated to an edge weight system $(a_k, b_k, c_k)$ for the edges $e_k$, $f_k$, $g_k$ of $\tau_k$ that satisfies the induction relation that
\begin{align*}
\qquad\qquad a_k &= b_{k-1}^{-1}&
 b_k&=  (1+b_{k-1})^2 a_{k-1}&
c_k &=  (1+b_{k-1})^{-2} b_{k-1}^2 c_{k-1}
\end{align*}
if $\phi_k=L$, and 
\begin{align*}
\qquad\qquad a_k &= c_{k-1}^{-1}&
 b_k&= (1+c_{k-1})^2 b_{k-1}&
c_k &= (1+c_{k-1})^{-2} c_{k-1}^2 a_{k-1}
\end{align*}
if $\phi_k=R$.

\item The first and last edge weight systems $ (a_0, b_0, c_0)$ and $(a_{k_0}, b_{k_0}, c_{k_0}) $ are equal. 
\end{enumerate}
\end{lem}

\begin{proof} We can simplify the ideal triangulation sweep  $\tau = \tau^{(0)}$, $ \tau^{(1)}$, \dots, $ \tau^{(k_0-1)}$, $ \tau^{(k_0)}=\phi(\tau)$ by arranging that, if $\tau^{(k)}$ differs from $\tau^{(l)}$ by an edge relabelling, then necessarily $l=k$ or $k \pm1$. Indeed, the compatibility condition between edge weight systems corresponds to the case $q=1$ of the Chekhov-Fock coordinate change isomorphisms of \S \ref{subsect:CheFock}, and in particular satisfy the relation  $\Phi^1_{\tau'\tau''} \circ \Phi^1_{\tau''\tau'''} = \Phi^1_{\tau'\tau'''}$. It follows that, if $\tau^{(k)}$ differs from $\tau^{(l)}$ by a permutation of its edges, the edge weight system $a^{(k)}$ and $a^{(l)}$ coincide up to the same permutation. If $|l-k|>1$, we can then shorten the ideal triangulation sweep to progressively eliminate this situation. 

We can similarly arrange that no $\tau^{(k)}$ is equal to an edge relabelling of $\phi(\tau^{(l)})$ for some $l>0$. 

We then temporarily omit the edge labellings in this ideal triangulation sweep, and obtain a sequence of unlabelled ideal triangulations $\tau=\tau_0$, $\tau_1$, \dots, $\tau_{k_0-1}$, $\tau_{k_0}=\phi(\tau)$ such that each $\tau_k$ is obtained from $\tau_{k-1}$ by a diagonal exchange. By construction, each $\tau_k$ corresponds to a single triangulation $\tau^{(l)}$, or to two consecutive $\tau^{(l)}$ and $\tau^{(l+1)}$ differing by an edge relabelling, in the sweep. 

For every $k\geq 1$, we now label the edges of $\tau_k$ as $e_k$, $f_k$, $g_k$ in such a way that  $e_k$ is the edge that is not an edge of $\tau_{k-1}$, namely is the new diagonal is the diagonal exchange from $\tau_{k-1}$ to $\tau_k$; the edges $e_k$, $f_k$ and $g_k$ are then uniquely determined by the property that they occur clockwise in this order around the faces to $\tau_k$. For $k=0$, we set $e_0 = \phi^{-1}(e_{k_0})$, $f_0 = \phi^{-1}(f_{k_0})$, $g_0 = \phi^{-1}(g_{k_0})$. 

By construction, each ideal triangulation $\tau_k$ is equal, up to edge relabelling, to a triangulation $\tau^{(k')}$ of the original ideal triangulation sweep. The edge weight system $a^{(k')}$ of this ideal triangulation sweep then gives an edge weight system $(a_k, b_k, c_k)$ for $\tau_k$, where the weights $a_k$, $b_k$ and $c_k \in \C^*$ are respectively associated to the edges $e_k$, $f_k$, $g_k$ of $\tau_k$. Since $\tau_k$ is obtained from $\tau_{k-1}$ by a diagonal exchange followed with an edge relabelling, each $(a_k, b_k, c_k)$ is obtained from $(a_{k-1}, b_{k-1}, c_{k-1})$ by formulas similar to those we encountered in \S \ref{subsect:SweepInvariantChar}, except that we are now in the case of the one-puncture torus. These precise formulas can, for instance, be found in  \cite[\S 2]{Liu} or \cite[\S 8]{BonLiu}; see also below. 

We now choose the identification $H_1(S_{1,1}) \cong \Z^2$ so that the edges $e_0$, $f_0$, $g_0$ are  respectively disjoint from closed curves representing the homology classes $(1,1)$, $(0,1)$, $(1,0 )$. 

\begin{figure}[htbp]
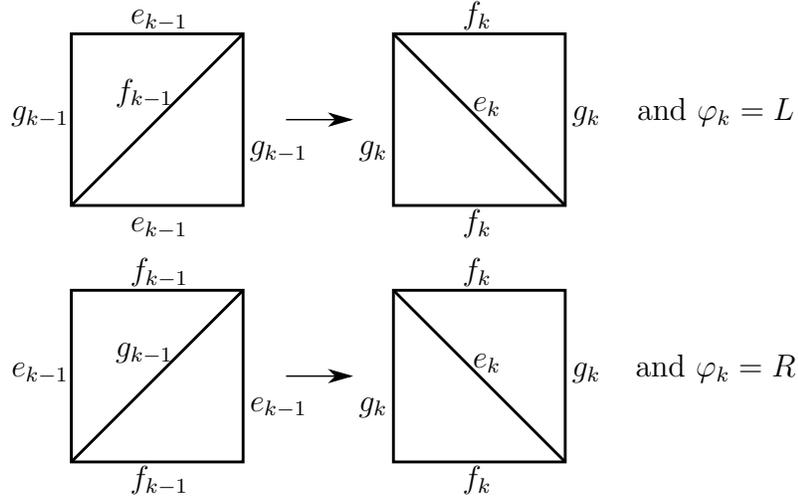

\vskip 10pt
\SetLabels
(.15 * .61) $f_{k-1}$\\
( .18 * 1.05 ) $e_{k-1}$  \\
( .18 * -.14 ) $e_{k-1}$  \\
( -.06 * .5 ) $g_{k-1}$  \\
( .42 * .3 ) $g_{k-1}$  \\
(.84 * .55) $e_k$\\
( .82 * 1.05 ) $f_k$  \\
( .82 * -.14 ) $f_k$  \\
( .61 * .3 ) $g_k$  \\
( 1.04 * .5 ) $g_k$  \\
(1.3*.5) and $\phi_k=L$\\
\endSetLabels
\centerline{\AffixLabels{\includegraphics[width=.4\textwidth]{DiagEx.eps}}\hskip 3cm}
\vskip 30pt
\SetLabels
(.15 * .61) $g_{k-1}$\\
( .18 * 1.05 ) $f_{k-1}$  \\
( .18 * -.14 ) $f_{k-1}$  \\
( -.06 * .5 ) $e_{k-1}$  \\
( .42 * .3 ) $e_{k-1}$  \\
(.84 * .55) $e_k$\\
( .82 * 1.05 ) $f_k$  \\
( .82 * -.14 ) $f_k$  \\
( .61 * .3 ) $g_k$  \\
( 1.04 * .5 ) $g_k$  \\
(1.3*.5) and $\phi_k=R$\\
\endSetLabels
\centerline{\AffixLabels{\includegraphics[width=.4\textwidth]{DiagEx.eps}}\hskip 3cm}
\vskip 10pt

\caption{The two types of diagonal exchanges}
\label{fig:CasesLandR}
\end{figure}

For each $k\geq 1$, the fact that $\tau_k$ is different from $\tau_{k-2}$ (and from $\phi^{-1}(\tau_{k_0-1})$ for $k=1$) shows that $e_k \neq e_{k-1}$. Therefore, there are only two possibilities: 
\begin{enumerate}
 \item  the diagonal exchange occurs along the edge $f_{k-1}$, in which case $f_k = e_{k-1}$ and $g_k = g_{k-1}$ and  the edge weight systems $(a_k, b_k, c_k)$ and $(a_{k-1}, b_{k-1}, c_{k-1})$  are related by the formulas
 \begin{align*}
\qquad\qquad a_k &= b_{k-1}^{-1}&
 b_k&=  (1+b_{k-1})^2 a_{k-1}&
c_k &=  (1+b_{k-1})^{-2} b_{k-1}^2 c_{k-1};
\end{align*}
  in this case we set $\phi_k=L =\left( 
\begin{smallmatrix}
1&1\\0&1
\end{smallmatrix}
\right)$;

 \item or the diagonal exchange occurs along the edge $g_{k-1}$, in which case $f_k = f_{k-1}$, $g_k= e_{k-1}$ and the relations
 \begin{align*}
\qquad\qquad a_k &= c_{k-1}^{-1}&
 b_k&= (1+c_{k-1})^2 b_{k-1}&
c_k &= (1+c_{k-1})^{-2} c_{k-1}^2 a_{k-1}
\end{align*}
hold; in this case we set $\phi_k=R=\left( 
\begin{smallmatrix}
1&0\\1&1
\end{smallmatrix}
\right)$.

\end{enumerate}

By induction, it follows from the construction that $ \tau_k = \phi_1 \circ \phi_2 \circ \dots \circ \phi_k(\tau_0) $ for every $k$. (Note that the order of the $\phi_i$ is not necessarily the one that might have been anticipated.) 

In particular, $\phi(\tau_0) = \tau_{k_0} = \phi_1 \circ \phi_2 \circ \dots \circ \phi_{k_0}(\tau_0)$. Since the stabilizer of $\tau_0$ in the mapping class group $\pi_0 \, \Diff^+(S_{1,1})$ is generated by $J= \left( 
\begin{smallmatrix}
-1&0\\0&-1
\end{smallmatrix}
\right)$, it follows that
$$
\phi =  \phi_1 \circ \phi_2 \circ \dots \circ \phi_{k_0} \circ J^\epsilon
$$
for some $\epsilon\in \{0,1\}$. This concludes the proof. 
\end{proof}

Proposition~\ref{prop:RepsCheFock} associates to the edge weight system $(a_k, b_k, c_k)$ and to a puncture weight $h_v$ with $h_v^n = a_k^2 b_k^2 c_k^2$  an irreducible representation $\rho_k \colon \T_{\tau_k}^q(S) \to \End(V)$. Our computation will require us to be more specific. 

First of all, instead of the puncture invariant $h_v$, it is more convenient to consider its square root $h$ uniquely determined by the properties that
\begin{align*}
 h^2&= h_v && h^n = a_0b_0c_0.
\end{align*}

We begin with an arbitrary triple $(x_0, y_0, z_0)$ with  $x_0^n = a_0$, $y_0^n = b_0$, $z_0^n = c_0$ and $x_0y_0z_0=h$. Then, to  take advantage of Lemma~\ref{lem:LeftRightIntertwiners}, we inductively define $(x_k, y_k, z_k)$ in the  following way.

If $\phi_k=L$, we set $u_k = q y_{k-1}$ and we  pick an arbitrary $v_k \in \C$ such that $v_k^n = 1 + u_k^n $. Then,
\begin{align}
\label{eqn:DefineX_kY_kZ_kLeft}
x_{k} &= y_{k-1}^{-1}
&
y_{k} &=v_k^2  x_{k-1}
&
z_{k} &=v_k^{-2} y_{k-1}^2 z_{k-1}.
\end{align}

If $\phi_k=R$, we set $u_k = q z_{k-1}$ and we again pick an arbitrary $v_k \in \C$ such that $v_k^n = 1 + u_k^n $. Then,
\begin{align}
\label{eqn:DefineX_kY_kZ_kRight}
x_{k} &= z_{k-1}^{-1}
&
y_{k} &=v_k^2 y_{k-1}
&
z_{k} &= v_k^{-2} z_{k-1}^2 x_{k-1}.
\end{align}

Note that $(x_k^n, y_k^n, z_k^n) = (a_k, b_k, c_k)$, and that $ u_k = q x_k^{-1} $ and $v_k^n = 1+a_k^{-1}$ in both cases. Also, the product $x_ky_kz_k$ is independent of $k$, and is consequently equal to the preferred square root $h$ of the puncture weight $h_v$.

Although $(a_{k_0}, b_{k_0}, c_{k_0}) = (a_0, b_0, c_0)$, there is no reason for $(x_{k_0}, y_{k_0}, z_{k_0}) $ to be equal to $ (x_0, y_0, z_0) $. We just know that $(x_{k_0}^n, y_{k_0}^n, z_{k_0}^n) = (x_0^n, y_0^n, z_0^n) $ and $x_{k_0} y_{k_0} z_{k_0} = x_0 y_0 z_0 $. As a consequence, there exists $l_0$, $m_0$, $n_0\in \Z$ with $l_0 + m_0 + n_0 =0$ such that
\begin{align*}
x_{0} &= q^{4l_0} x_{k_0}
&
y_{0} &= q^{4m_0} y_{k_0}
&
z_{0} &= q^{4n_0} z_{k_0}.
\end{align*}

\begin{prop}
\label{prop:IntertwinerCompElementaryIntertwiner}
Under the hypotheses of Theorem~{\upshape\ref{thm:CheFockAndKauffmanIntertwiners}}, the intertwiner $\Lambda_{\phi, r}^q$ of Proposition~{\upshape\ref{prop:KauffmanIntertwinerDefined}} coincides, up to conjugation and rescaling by a factor with modulus $1$, with the composition
$$
\Lambda  =  \Lambda_1\circ    \Lambda_2 \circ\dots   \circ\Lambda_{k_0} \circ T_{l_0m_0n_0}
$$
where
$$
\Lambda_k =
\begin{cases}
L_{u_kv_k} &\text{ if } \phi_k=L\\
R_{u_kv_k} &\text{ if } \phi_k=R
\end{cases}
$$
for the elementary intertwining operators $L_{uv}$, $R_{uv}$ and $T_{l_0m_0n_0} \colon \C^n \to \C^n $ of \S {\upshape\ref{subsect:LeftRightIntertwiners}} and \S{\upshape\ref{subsect:TwistIntertwiner}}, and where the numbers $u_k$, $v_k \in \C^*$ and $k_0$, $l_0$, $m_0\in \Z$ are defined as above.  
\end{prop}

\begin{proof}
 For each $k$, the edge weight system $(x_k, y_k, z_k)$ defines a standard representation $\rho_{x_ky_kz_k} \colon \T^q \to \End(\C^n)$ as in \S \ref{subsect:StandardReps}, and therefore a representation $\rho_k = \rho_{x_ky_kz_k} \circ \theta_{\tau_k}^{-1} \colon \T_{\tau_k}^q(S) \to \End(\C^n)$ by composition with the canonical isomorphism between $\T^q$ and the Chekhov-Fock algebra of the ideal triangulation $\tau_k$. 
 
Recall that $\tau_k$ is obtained from $\tau_{k-1}$ by a diagonal exchange and an edge relabelling. The formulas of \cite{Liu} for the Chekhov-Fock coordinate change isomorphisms in the special case of the one-puncture torus then show that 
$$
\theta_{\tau_{k-1}}^{-1} \circ \Phi_{\tau_{k-1}\tau_k}^q \circ \theta_{\tau_k} \colon \T^q \to \T^q
$$
coincides with the isomorphism $\mathcal L$ of \S \ref{subsect:LeftRightIntertwiners} if the diagonal exchange is performed along the edge $g_{k-1}$, and with $\mathcal R$ if  the diagonal exchange is performed along the edge $f_{k-1}$. Equivalently, this isomorphism $\theta_{\tau_{k-1}}^{-1} \circ \Phi_{\tau_{k-1}\tau_k}^q \circ \theta_{\tau_k} $ is equal to $\mathcal L$ if $\phi_k=L$, and to $\mathcal R$ if $\phi_k=R$. 
 
 Lemma~\ref{lem:LeftRightIntertwiners} then shows that 
 $$
 \rho_{k-1} \circ \Phi_{\tau_{k-1}\tau_k}^q (W) = \Lambda_k \circ  \rho_k(W) \circ \Lambda_k^{-1}
 $$
 for every $W \in \T_{\tau_k}^q(S_{1,1})$. 
 
 Similarly, let $\Psi_{\phi, \tau_0}^q \colon \T_{\tau_0}^q(S_{1,1}) \to \T_{\tau_{k_0}}^q(S_{1,1})$ be the isomorphism induced by $\phi$, sending the generator $ \T_{\tau_0}^q(S_{1,1})$ corresponding to an edge of $\tau_0$ to the generator of $\T_{\tau_{k_0}}^q(S_{1,1})$ corresponding to the image of that edge under $\phi$. With our conventions, $\Psi_{\phi, \tau_0}^q $ is just $\theta_{\tau_{k_0}} \circ \theta_{\tau_0}^{-1}$. Lemma~\ref{lem:TwistIntertwiner} then shows that 
 $$
 \rho_{k_0} \circ \Psi_{\phi, \tau_0}^q  (W) = T_{l_0m_0n_0} \circ \rho_0(W) \circ T_{l_0m_0n_0} ^{-1}. 
 $$
 
 Combining these properties, we see that 
 $$
 \rho_0 \circ \Phi_{\tau_{k_0} \tau_0}^q \circ \Psi_{\phi, \tau_0}^q (W) = \Lambda \circ \rho_0(W) \circ \Lambda^{-1}
 $$
 for the isomorphism $\Lambda \colon \C^n \to \C^n$ defined in the statement. 
 
 This is the property satisfied by the intertwiner $\bar \Lambda_{\phi, \bar r}^q$ of Proposition~\ref{prop:QuantumTeichmullerIntertwinerDefined}. The determinant of $\Lambda$ has modulus 1 by Lemmas~\ref{lem:LeftRightIntertwinerDeterminant} and \ref{lem:TwistIntertwiner}. 
 Proposition~\ref{prop:QuantumTeichmullerIntertwinerDefined} then shows that 
 $\bar \Lambda_{\phi, \bar r}^q$ and $\Lambda$ coincide up to conjugation and rescaling by a modulus 1 scalar. 
 
 Theorem~\ref{thm:CheFockAndKauffmanIntertwiners} then shows that the three intertwiners $\Lambda$, $\bar \Lambda_{\phi, \bar r}^q$ and $\Lambda_{\phi, r}^q$ coincide up to conjugation and modulus 1 rescaling. 
\end{proof}

Note that there was some freedom in this computation, as we could arbitrarily choose $n$--roots $v_k$ such that $v_k^n = 1+ u_k^n$. These choices are balanced by the correction factors $l_0$, $m_0$, $n_0  \in \Z$, in the sense that different choices for the $v_k$ will in general lead to different values for these correction factors.

We are of course interested in applying Proposition~\ref{prop:IntertwinerCompElementaryIntertwiner} in the context of Conjecture~\ref{con:MainConjecture}. We now connect the two points of view. 

 In the hypotheses of that conjecture, we were given a character $[r] \in \XX(S_{1,1})$ and a puncture weight $\theta_v \in \C$ such that $  \Tr r(\alpha_v) = -\E^{\theta_v} - \E^{- \theta_v}$ when  $\alpha_v$ is a small loop going once around the puncture. In the above construction we were given,  in addition to the periodic edge weight system $(a_0, b_0, c_0)$, $(a_1, b_1, c_1)$, \dots, $(a_{k_0}, b_{k_0}, c_{k_0}) = (a_0, b_0, c_0)$  defining the character $[\bar r] \in \XP(S_{1,1})$,  a global weight $h \in \C^*$ that is an $n$--root of the products $a_kb_kc_k$ (which are independent of $k$). 

\begin{lem}
\label{lem:ConnectViewpoints}
With the above data, the hypotheses of Theorem~{\upshape\ref{thm:CheFockAndKauffmanIntertwiners} }hold (and we can apply Proposition~{\upshape\ref{prop:IntertwinerCompElementaryIntertwiner}}) if and only if   $[ r] \in \XX(S_{1,1})$ lifts  $[\bar r] \in \XP(S_{1,1})$ and $h= \E^{\pm\frac1n \theta_v}$. 
\end{lem}
\begin{proof} The puncture weight $h_v$ of Proposition~\ref{prop:RepsCheFock} is related to $h$ by the property that $h_v=h^2$. Then the relation $p_v^2 = h_v + h_v^{-1} +2$ of Theorem~\ref{thm:CheFockAndKauffmanIntertwiners} holds if and only if $p_v = \pm (h + h^{-1})$. 

The sign $\pm$ is actually restricted by the geometry. Indeed, a simple computation (see for instance \cite[\S 8.4]{BonBook}) shows that $\Tr r(\alpha_v) = - a_0b_0c_0 -(a_0b_0c_0)^{-1} $.
Since $  \Tr r(\alpha_v) = -\E^{\theta_v} - \E^{- \theta_v}$ in the setup of Conjecture~\ref{con:MainConjecture}, it follows that $\E^{\theta_v}= (a_0 b_0 c_0)^{\pm1}$. Because of the hypothesis $p_v = \E^{\frac1n \theta_v} + \E^{-\frac1n\theta_v}$ in Conjecture~\ref{con:MainConjecture} and because $h^n = a_0b_0c_0$,  the  condition $p_v = \pm ( h+h^{-1} )$  is therefore equivalent to $h= \E^{\pm \frac1n \theta_v}$ (and $p_v = + ( h+h^{-1} )$).
\end{proof}

\subsection{Computing the correction factors $(l_0, m_0, n_0)$}
\label{subsect:CorrectionFactors}

We  need a practical way to compute the correction terms $(l_0, m_0, n_0) \in \Z^3$ that occurred in the previous section, and this in a way that is systematic in $n$ so that we can better estimate asymptotics in the framework of Conjecture~\ref{con:MainConjecture}. The general idea is that, whenever we need to choose an $n$--root for a quantity $w\in \Z^*$, we will take it of the form $\exp \frac1n W$ for a fixed ``logarithm'' $W\in \C$, namely a number independent of $n$ such that $\exp W=w$. We will try to systematically denote these logarithms by the capital letter corresponding to the lower-case original data. 

Conjecture~\ref{con:MainConjecture} is unchanged if we replace $\theta_v$ by $-\theta_v$ in its hypothesis. By Lemma~\ref{lem:ConnectViewpoints}, we can therefore arrange that $h = \E^{\frac1n \theta_v}$. 

Since $\E^{\theta_v}=h^n=a_0b_0c_0$, we can now  choose ``logarithms'' $A_0$, $B_0$, $C_0 \in \C$ for $a_0$, $b_0$, $c_0$, such that
\begin{align*}
 a_0 &= \exp A_0
 &
 b_0 &= \exp B_0
 &
 c_0 &= \exp C_0
\end{align*}
and $ A_0 +  B_0 +  C_0 = \theta_v$.

After these initial choices, we inductively define $A_k$, $B_k$, $C_k \in \C$ as follows. At each step of the induction, we arbitrarily choose a  ``logarithm'' $V_k \in \C$ such that 
$$\exp V_k = 1+a_k^{-1}. $$
We could systematically take $V_k = \log \left( 1+a_k^{-1} \right)$ for whatever (discontinuous) complex logarithm function $\log$ we prefer, but in \cite{BWY3} it will be convenient to make these $V_k$ depend continuously on the original data. We consequently allow here some flexibility in the choice of the $V_k$.  

We then set
\begin{align*}
A_k &=  -B_{k-1}
\\
B_k&=  2 V_k + A_{k-1}
\\
C_k &= -2 V_k  +2B_{k-1} + C_{k-1}
\end{align*}
if $\phi_k=L$, and 
\begin{align*}
A_{k} &=  -C_{k-1}
\\
B_{k} &=  2 V_k + B_{k-1}
\\
C_{k} &= -2 V_k +2 C_{k-1} +A_{k-1}.
\end{align*}
if $\phi_k=R$.

The construction is designed so that, if for each $n$ we define 
\begin{align}
 u_k &=
\textstyle
 q \exp \left(  - \frac1n A_k  \right)
&
 v_k &= 
 \textstyle \exp  \frac1n V_k 
\end{align}
\begin{align}
\label{eqn:DefineX_kY_kZ_kLambda}
 x_k &= \textstyle \exp \frac1n A_k
 &
 y_k &= \textstyle \exp \frac1n B_k
 &
 z_k &=\textstyle \exp \frac1n C_k
\end{align}
these $u_k$, $v_k$, $x_k$, $z_k$, $z_k$ satisfy the required relations (\ref{eqn:DefineX_kY_kZ_kLeft}--\ref{eqn:DefineX_kY_kZ_kRight}). 

In particular, $x_k^n = a_k$, $y_k^n = b_k$, $z_k^n = c_k$ and $x_k y_k z_k = h = \exp\frac1n \theta_v$ for every $k$. 

Since  $(a_{k_0}, b_{k_0}, c_{k_0})=(a_0, b_0, c_0)$ it follows that, for the corresponding logarithms,  there exists $\widehat l_0$, $\widehat m_0$, $\widehat n_0 \in \Z$ such that 
\begin{align*}
A_0 &=  A_{k_0} + 2\pi \I\, \widehat l_0
 \\
B_0 &= B_{k_0}  + 2\pi \I\, \widehat m_0
 \\
C_0 &=  C_{k_0} + 2\pi \I\, \widehat n_0.
\end{align*}

By construction, the quantity $A_k + B_k + C_k$ is independent of $k$. It follows that
$$
\widehat  l_0  + \widehat  m_0 +\widehat  n_0 =0. 
$$

\begin{lem}
\label{lem:CorrectionFactors}
If $q=\E^{\frac{2\pi\I}n}$ and for these choices of $u_k$, $v_k$, $x_k$, $y_k$, $z_k$, we can take in the formula of Proposition~{\upshape\ref{prop:IntertwinerCompElementaryIntertwiner}} the correction factors $l_0$, $m_0$, $n_0 \in \Z$ to be equal to
\begin{align*}
l_0 &=  \widehat l_0 \textstyle \frac{(n-1)^2}4 
&
m_0 &=  \widehat m_0 \textstyle \frac{(n-1)^2}4 
&
n_0 &= \widehat n_0 \textstyle \frac{(n-1)^2}4 
\end{align*}
where these quantities are defined as above. 
\end{lem}

\begin{proof} By construction,  
\begin{align*}
 x_0 &=  \textstyle \exp\frac1nA_0
 =  \textstyle \exp \left( \frac1n A_{k_0} +  \frac{2\pi\I}n\, \widehat l_0 \right)
 = q^{\widehat l_0} x_{k_0} = q^{4 l_0} x_{k_0}
\end{align*}
since $q^{(n-1)^2} = q$. Similarly, $y_0 =  q^{4 m_0} y_{k_0}$ and  $z_0 =  q^{4 n_0} z_{k_0}$. 

This is exactly what we needed to apply Proposition~{\upshape\ref{prop:IntertwinerCompElementaryIntertwiner}}. 
\end{proof}

\begin{rem}
 If  we take the different value $q=\E^{\frac{2k\pi\I}n}$ for some $k \in \Z$ coprime with $n$, then in the formulas of Lemma~\ref{lem:CorrectionFactors} the terms $\frac{(n-1)^2}4$ will get replaced by the inverse of $4k$ modulo~$n$. 
\end{rem}

\subsection{Explicit formulas}
\label{subsect:ExplicitComputation}
By Proposition~\ref{prop:IntertwinerCompElementaryIntertwiner}, the intertwiner of Proposition~\ref{prop:KauffmanIntertwinerDefined} can be taken as 
$$
\Lambda_{\phi, r}^q =  \Lambda_1   \Lambda_2 \dots   \Lambda_{k_0}  T_{l_0m_0n_0}
$$
for the isomorphisms $\Lambda_k$ and $T_{l_0m_0n_0}$ defined there. 
As a consequence, its trace is equal to 
$$
\Tr \Lambda_{\phi, r}^q  = 
\sum_{i_0, \, i_1, \dots, \, i_{k_0}=1}^n
(\Lambda_1)_{i_0i_1}   (\Lambda_2)_{i_1i_2} \dots   (\Lambda_{k_0})_{i_{k_0-1}i_{k_0}}  (T_{l_0m_0n_0})_{i_{k_0}i_0}
$$
where $(\Lambda_k)_{ij}$ and $(T_{l_0m_0n_0})_{ij}$ are the entries of the matrices of the elementary intertwiners $\Lambda_k$, $T_{l_0m_0n_0} \colon \C^n \to \C^n$. 

In particular, for the discrete quantum dilogarithm function of \S \ref{subsect:DiscreteQuantumDilog}, 
$$
(\Lambda_k)_{ij} = (L_{u_k v_k})_{ij} =  \frac {\QDL^q(u_k, v_k \vbar  2j) } {\left| D^q(u_k) \right|^{\frac1n}  \sqrt n } \,  q^{-i^2 +j^2+4ij +i-j}  
$$
if the matrix $\phi_k$ is equal to $L$, and
$$
(\Lambda_k)_{ij} = (R_{u_kv_k})_{ij} = \frac {\QDL^q(u_k, v_k \vbar  2j) } { \left| D^q(u_k) \right|^{\frac1n}  \sqrt n } \, q^{i^2 +3j^2-4ij+i-j} 
$$
if $\phi_k = R$. We can combine both cases by saying that
$$
(\Lambda_k)_{ij} = \frac {\QDL^q(u_k, v_k \vbar  2j) } {\left| D^q(u_k) \right|^{\frac1n}  \sqrt n } \, q^{ \epsilon_k(i^2 + j^2 -4ij) + 2j^2 +i-j} 
$$
with
$$
 \epsilon_k = 
\begin{cases}
-1 &\text{if } \phi_k=L\\
+1&\text{if } \phi_k=R.
\end{cases}
$$

Also,
$$
(T_{l_0m_0n_0})_{ij} = q^{2j (n_0 -m_0 )} \delta_{i, j+l_0}
$$

Therefore,
\begin{align*}
\Tr \Lambda_{\phi, r}^q
&=  A_n'\sum_{i_0,\, i_1, \dots,\, i_{k_0}=1}^n
\prod_{k=1}^{k_0}  \QDL^q(u_k, v_k \vbar  2i_k)   \\
& \qquad\qquad\qquad\qquad\qquad q^{ \sum_{k=1}^{k_0} \epsilon_k(i_{k-1}^2 + i_k^2 -4i_{k-1}i_k) +  2 \sum_{k=1}^{k_0} i_k^2 + i_0-i_{k_0} } \\
&  \qquad\qquad\qquad\qquad\qquad\qquad\qquad\qquad  q^{2i_0 (n_0 -m_0 )} \delta_{i_{k_0}, i_0+l_0}
\\
&=A_n\sum_{i_1,\, i_2, \dots,\, i_{k_0}=1}^n
\prod_{k=1}^{k_0}  \QDL^q(u_k, v_k \vbar  2i_k) \\
& \qquad\qquad\qquad\qquad q^{\epsilon_1(i_{k_0}^2 + i_1^2 -4i_{k_0} i_1) + \sum_{k=2}^{k_0} \epsilon_k(i_{k-1}^2 + i_k^2 -4i_{k-1}i_k) +  2 \sum_{k=1}^{k_0} i_k^2} \\
&  \qquad\qquad\qquad\qquad\qquad\qquad\qquad\qquad \qquad q^{ 4\epsilon_1 n_0 i_1 +2i_{k_0} (-\epsilon_1 kl_0 -m_0 +n_0)}
\end{align*}
with
$$
A_n = \frac
{ q^{l_0(\epsilon_1 l_0 + 2 m_0 -2 n_0-1) }}
{n^{\frac {k_0}2}\prod_{k=1}^{k_0} \left| D^q(u_k) \right|^{\frac1n} }
$$
(and $A_n' = A_n  q^{-l_0(\epsilon_1 l_0 + 2 m_0 -2 n_0-1) }$). If we set $\epsilon_{k_0+1}=\epsilon_1$, this can also be rewritten as
\begin{align*}
\Tr \Lambda_{\phi, r}^q
&=A_n\sum_{i_1,\, i_2, \dots,\, i_{k_0}=1}^n
\prod_{k=1}^{k_0}  \QDL^q(u_k, v_k \vbar  2i_k) \\
& \qquad\qquad\qquad\qquad\qquad\quad q^{ \sum_{k=1}^{k_0} i_k^2(\epsilon_k+ \epsilon_{k+1}+2)  -4  \sum_{k=1}^{k_0-1} \epsilon_{k+1} i_ki_{k+1}  -4 \epsilon_1 i_{k_0} i_1 }\\
&  \qquad\qquad\qquad\qquad\qquad\qquad\qquad\qquad\qquad  q^{ 4\epsilon_1 l_0 i_1 +2i_{k_0} (-\epsilon_1 l_0 -m_0 +n_0)}. 
\end{align*}

If we count indices modulo $k_0$, so that $i_{k_0+1} = i_1$, the third line can be written in a more compact way by using
$$
 -4  \sum_{k=1}^{k_0-1} \epsilon_{k+1} i_ki_{k+1}  -4 \epsilon_1 i_{k_0} i_1 
 =
  -4  \sum_{k=1}^{k_0} \epsilon_{k+1} i_ki_{k+1} 
$$

Finally, remember from Lemma~\ref{lem:CorrectionFactors} that we can take $l_0 = \widehat l_0 \frac{(n-1)^2}4$,  $m_0 = \widehat m_0 \frac{(n-1)^2}4$ and  $n_0 = \widehat n_0 \frac{(n-1)^2}4$ for integers $\widehat l_0$, $\widehat m_0$, $\widehat n_0 \in \Z$ that are independent of $n$, and such that $\widehat l_0 + \widehat m_0 + \widehat n_0=0$. Then the last term of the above expression is equal to
\begin{align*}
q^{ 4\epsilon_1 l_0 i_1 +2i_{k_0} (-\epsilon_1 l_0 -m_0 +n_0)}
&= q^{ \epsilon_1 \widehat l_0 i_1  + \frac{-\epsilon_1 \widehat l_0 -\widehat m_0 +\widehat n_0}2 i_{k_0} }.
\end{align*}
Note that $-\epsilon_1 \widehat l_0 -\widehat m_0 +\widehat n_0$ is even since $\widehat l_0 + \widehat m_0 + \widehat n_0=0$.

We summarize this computation in the following statement. 

\begin{prop}
Let $\phi \colon S_{1,1} \to S_{1,1}$ be an orientation preserving diffeomorphism of the one-puncture torus. Let $[r]\in \XX(S_{1,1})$ be a $\phi$--invariant character associated to a sequence of edge weight systems $(a_0, b_0, c_0)$, $(a_1, b_1, c_1)$, \dots, $(a_{k_0-1}, b_{k_0-1}, c_{k_0-1})$,$(a_{k_0}, b_{k_0}, c_{k_0})=(a_0, b_0, c_0)$ as in Lemma~{\upshape\ref{lem:OrganizePhi}}. Finally, let $\theta_v\in \C$ be such that $\E^{\theta_v}= a_0 b_0 c_0$ and, for an odd integer $n$,  let $\Lambda_{\phi, r}^q$ be the intertwiner associated by Proposition~{\upshape\ref{prop:KauffmanIntertwinerDefined}} to the data of $\phi$, $[r]$, $q=\E^{\frac{2\pi\I}n}$ and $p_v =  \E^{\frac1n \theta_v}  +\E^{-\frac1n \theta_v} $.

Then, up to multiplication by a scalar with modulus $1$, the trace of $\Lambda_{\phi, r}^q$ is equal to
\begin{equation*}
\begin{aligned}
\Tr \Lambda_{\phi, r}^q
&=\frac1
{n^{\frac {k_0}2}\prod_{k=1}^{k_0} \left| D^q(u_k) \right|^{\frac1n} }
\\
&\qquad\qquad
\sum_{i_1,\, i_2, \dots,\, i_{k_0}=1}^n
\prod_{k=1}^{k_0}  \QDL^q(u_k, v_k \vbar  2i_k) \\
&\qquad\qquad \qquad\qquad q^{ \sum_{k=1}^{k_0} i_k^2(\epsilon_k+ \epsilon_{k+1}+2)  -4  \sum_{k=1}^{k_0} \epsilon_{k+1} i_ki_{k+1}  }\\
&  \qquad\qquad\qquad\qquad\qquad\qquad   q^{ \epsilon_1 \widehat l_0 i_1  + \frac{-\epsilon_1 \widehat l_0 -\widehat m_0 +\widehat n_0}2 i_{k_0} }
\end{aligned}
\end{equation*}
where the quantities $u_k$, $v_k$, $\epsilon_k$, $\widehat l_0$, $\widehat m_0$, $\widehat n_0$ are defined in {\upshape\S \ref{subsect:IntertwinerCompElementaryIntertwiner}} and {\upshape\S \ref{subsect:CorrectionFactors}} and where, for $u$, $v\in \C$ with $v^n = 1+u^n$, 
$$
 \QDL^q(u,v \vbar  i) = v^{-i} \prod_{j=1}^i (1+ u q^{-2j})
$$
and \pushQED{\qed}
\begin{equation*}
D^q(u)  =  \prod_{i=1}^n \QDL^q(u, v \vbar  i)  = (1+u^n)^{- \frac{n+1}2} \prod_{j=1}^{n} (1+ u q^{-2j})^{n-j+1}.
\qedhere
\end{equation*} 
\end{prop}

\subsection{A few examples} 
\label{subsect:ExamplesComputation}

The algorithm developed in \S\S \ref{subsect:IntertwinerCompElementaryIntertwiner}--\ref{subsect:ExplicitComputation} has many moving pieces. It is probably useful to illustrate its implementation, and its subtleties, by applying it to a few examples. The following ones all correspond to the diffeomorphism  $\phi = LLR= \left(\begin{smallmatrix} 3&2\\1&1\end{smallmatrix}\right)$, and actually correspond to the three cases that we encountered in the numerical experiments of \S \ref{sect:Experimental}. 

Since the diffeomorphism $\phi$ is equal to $LLR$ in all three cases, we have that $k_0=3$, $\epsilon_1= -1$, $\epsilon_2=-1$ and $\epsilon_3=+1$. Therefore,
\begin{equation}
\label{eqn:ExampleLLR}
\begin{aligned}
\Tr \Lambda_{\phi, r}^q
&=\frac1
{n^{\frac {3}2} \left| D^q(u_1) \right|^{\frac1n}  \left| D^q(u_2) \right|^{\frac1n}  \left| D^q(u_3) \right|^{\frac1n} }
\\
&\qquad
\sum_{i_1,\, i_2, \, i_3 =1}^n
\QDL^q(u_1, v_1 \vbar  2i_1)\QDL^q(u_2, v_2 \vbar  2i_2)\QDL^q(u_3, v_3 \vbar  2i_3) \\
&\qquad\qquad \qquad\qquad\qquad q^{ 2 i_2^2 + 2 i_3^2 + 4 i_1 i_2 -4 i_2 i_3 + 4 i_3 i_1  }
  q^{- \widehat l_0 i_1  + \frac{ \widehat l_0 -\widehat m_0 +\widehat n_0}2 i_{3} }
\end{aligned}
\end{equation}
where the numbers $u_k$, $v_k\in \C$ and the integers $\widehat l_0$, $\widehat m_0$, $\widehat n_0\in \Z$ are determined by the algebraic data and by choices of logarithms. We now explain how these parameters are chosen, according to the case considered. 

First of all, since $\phi = LLR$, the edge weights associated to the corresponding ideal triangulation sweep are such that
\begin{align*}
 a_1 &= b_0^{-1}
 &
 b_1 &= (1+b_0)^2 a_0
 &
 c_1 &= (1+b_0)^{-2}b_0^2 c_0
 \\
  a_2 &= b_1^{-1}
 &
 b_2 &= (1+b_1)^2 a_1
 &
 c_2 &= (1+b_1)^{-2}b_1^2 c_1
 \\
  a_3 &= c_2^{-1}
 &
 b_3 &= (1+c_2)^2 b_2
 &
 c_3 &= (1+c_2)^{-2}c_2^2 a_2
 \\
 a_0&=a_3
 &
 b_0 &= b_3
 &
 c_0 &= c_3.
\end{align*}

In general, we can only rely on numerical solutions to this system of equations, but it turns out that this example is still sufficiently small that we can explicitly find all solutions by an elementary algebraic manipulation. Namely, the initial edge weights $(a_0, b_0, c_0)$ of these solutions are of the form
\begin{align*}
 b_0 &= -\frac{2 a_0+1}{2( a_0 + 1)} \pm \I \frac{\sqrt{8 a_0^2  + 11 a_0 + 4}}{2(a_0 + 1)\sqrt{a_0}} 
 \text{ or }
-\frac{2 a_0+3}{2( a_0 + 1)} \pm \I \frac{\sqrt{3 a_0 + 4}}{2(a_0 + 1)\sqrt{a_0}} 
 \\
 c_0 &=  \frac{(1 + a_0 (1 + b_0)^2)^2}{a_0^3 b_0^2 (1 + b_0)^2 },
\end{align*}
excluding in each case finitely many values for $a_0$ to guarantee that all edge weights $a_k$, $b_k$, $c_k$ are different from 0. In particular, the space of solutions consists of two connected complex curves in $\big( \C^* \big)^3$, which are disjoint by another computation. 

The hyperbolic character  $[\bar r_\hyp] \in \XP(S_{1,1})$ corresponding to the monodromy of the complete hyperbolic metric of the mapping torus $M_\phi$ is associated to the initial edge weight system $(a_0, b_0, c_0) = \left( \frac{-1-\I \sqrt 7}4, \frac{-3 + \I \sqrt 7}2, \frac {5-\I \sqrt 7}8 \right)$, which is in the curve defined by the first solution for $b_0$

\medskip\noindent\textsc{Example 1.} We consider the case where \nolinebreak
\begin{align*}
a_0 &=  -0.75 - 0.1 \,\I
\\
 b_0 &= -\frac{2 a_0+1}{2( a_0 + 1)} - \I \frac{\sqrt{8 a_0^2  + 11 a_0 + 4}}{2(a_0 + 1)\sqrt{a_0}} 
\approx 1.51712 + 1.20930 \,\I
 \\
 c_0 &=  \frac{(1 + a_0 (1 + b_0)^2)^2}{a_0^3 b_0^2 (1 + b_0)^2 }
 \approx  -2.25462 + 0.617377\, \I
\end{align*}
(where $\sqrt z$ is the usual complex square root, with $-\frac\pi2 < \Im \sqrt \leq \frac\pi2$). This is the example considered for the numerical data of Figure~\ref{fig:LLR}, largely selected for aesthetic reasons. It defines a $\phi$--invariant hyperbolic character $[\bar r] \in \XP(S_{1,1})$ which is in the same component of the fixed point set of the action of $\phi$ as the hyperbolic character $[\bar r_\hyp]$. In particular, it lifts to a character $[r] \in \XX(S_{1,1})$, and in facts admits exactly two such lifts since the fixed point set of the action of $\phi$ on $H^1(S_{1,1}; \Z/2)$ is isomorphic to $\Z/2$. 

The first diagram of Figure~\ref{fig:LLR} corresponds to the case where, for the standard complex logarithm function $\log$,  the puncture invariant $\theta_v$ is equal to
$$
\theta_v = \log a_0 + \log b_0 + \log c_0. 
$$
Then, if we follow the algorithm of \S \ref{subsect:CorrectionFactors}, we can choose
\begin{align*}
 A_0 &= \log a_0 
 &
 B_0 &= \log b_0 
 &
 C_0 &= \log c_0 
  \\
 & \approx -0.278871 - 3.00904 \,\I
 &
 & \approx 0.662750 + 0.672970 \,\I
 &
 & \approx 0.849133 + 2.87432 \,\I
 \\
 A_1 &= - B_0
 &
 B_1 &= 2 \log(1+a_1^{-1}) + A_0
 &
 C_1 &= - 2 \log(1+a_1^{-1}) + 2 B_0 + C_0
 \\
 A_2 &= - B_1
 &
 B_2 &= 2 \log(1+a_2^{-1}) + A_1
 &
 C_2 &= - 2 \log(1+a_2^{-1}) + 2 B_1 + C_1
 \\
 A_3 &= - C_2
 &
 B_3 &= 2 \log(1+a_3^{-1}) + B_2
 &
 C_3 &= - 2 \log(1+a_3^{-1}) + 2 C_2 + A_2
 \\
  & \approx -0.278871 - 3.00904 \,\I
 &
 & \approx 0.662750 + 0.672970 \,\I
 &
 & \approx 0.849133 + 2.87432 \,\I
\end{align*}
by systematically selecting $V_k = \log(1+a_k^{-1})$. 

Since we abstractly know that 
\begin{align*}
A_0 &=  A_3 + 2\pi \I\, \widehat l_0
 &
B_0 &= B_3  + 2\pi \I\, \widehat m_0
 &
C_0 &=  C_3 + 2\pi \I\, \widehat n_0.
\end{align*}
for some integers $\widehat l_0$, $\widehat m_0$, $\widehat n_0 \in \Z$, we numerically see that $\widehat l_0 = \widehat m_0 = \widehat n_0=0$. 

Therefore, in this example, the trace of the intertwiner $\Lambda_{\phi, r}^q$ is obtained by setting
\begin{align*}
 u_1 &= \exp   \frac{2\pi\I- A_1}n  
 &
 u_2 &=\exp   \frac{2\pi\I- A_2}n 
 &
 u_3 &=\exp \frac{2\pi\I- A_3}n 
 \\
 v_1 &= \exp \frac{\log(1+ a_1^{-1})}n
 &
 v_2 &= \exp \frac{\log(1+ a_2^{-1})}n
 &
 v_3 &= \exp \frac{\log(1+ a_3^{-1})}n
 \\
 \widehat l_0 &=0
 &
 \widehat m_0 &=0
 &
 \widehat n_0 &=0
\end{align*}
in Equation~(\ref{eqn:ExampleLLR}).

\medskip\noindent\textsc{Example 2.} In the previous example, we can  replace the puncture weight $\theta_v$ by 
$$
\theta'_v = \theta_v + 2 \pi \I \eta
$$
for some integer $\eta \in \Z$, while keeping the same periodic edge weight system $(a_0, b_0, c_0)$, $(a_1, b_1, c_1)$, $(a_2, b_2, c_2)$, $(a_3, b_3, c_3)=(a_0, b_0, c_0)$ for the ideal triangulation sweep. For instance, the second diagram of Figure~\ref{fig:LLR} corresponds to the case where $\eta=1$.  
  
Then, we can replace the logarithms $A_0$, $B_0$, $C_0$ of the previous case with 
\begin{align*}
 A_0' &= A_0 + 2 \pi\I \eta
 &
 B_0' &= B_0
 &
 C_0' &= C_0 
\end{align*}
in order to satisfy the relation $A_0'+B_0'+C_0' = \theta_v'$. The algorithm of \S \ref{subsect:CorrectionFactors} then gives
\begin{align*}
  A_1' &= - B_0' 
 &
 B_1' &= 2 \log(1+a_1^{-1}) + A_0' 
 &
 C_1' &= - 2 \log(1+a_1^{-1}) + 2 B_0' + C_0' 
 \\
   &=  A_1
 &
&=B_1 + 2 \pi\I \eta
 &
 &= C_1
 \\
 A_2' &= - B_1' 
 &
 B_2' &= 2 \log(1+a_2^{-1}) + A_1' 
 &
 C_2' &= - 2 \log(1+a_2^{-1}) + 2 B_1' + C_1' 
  \\
&= A_2 -  2 \pi\I \eta
 &
&=  B_2
 &
 &=  C_2 + 4 \pi\I \eta
 \\
 A_3' &= - C_2' 
 &
 B_3' &= 2 \log(1+a_3^{-1}) + B_2' 
 &
 C_3' &= - 2 \log(1+a_3^{-1}) + 2 C_2' + A_2' 
 \\
  &= A_3- 4\pi\I \eta
 &
&= B_3
 &
 &= C_3 + 6\pi\I\eta. 
\end{align*}

It follows that in this more general example the parameters $u_k$, $v_k$, $ \widehat l_0$, $ \widehat m_0$ and  $ \widehat n_0$ 
get replaced by
\begin{align*}
 u_1' &= \exp   \frac{2\pi\I- A_1'}n  = u_1
 &
 u_2' &=\exp   \frac{2\pi\I- A_2'}n = q^{-\eta} u_2
 &
 u_3' &=\exp \frac{2\pi\I- A_3'}n = q^{2\eta} u_3 
 \\
 v_1' &= \exp \frac{\log(1+ a_1^{-1})}n = v_1
 &
 v_2' &= \exp \frac{\log(1+ a_2^{-1})}n = v_2
 &
 v_3' &= \exp \frac{\log(1+ a_3^{-1})}n = v_3
 \\
 \widehat l_0' &=3 \eta
 &
 \widehat m_0' &=0
 &
 \widehat n_0' &=-3\eta. 
\end{align*}

\medskip\noindent\textsc{Example 3.} This time, we keep $\phi=LLR$ and the corresponding ideal triangulation sweep, but we choose a periodic edge weight system for the sweep whose initial edge weight system $(a_0'', b_0'', c_0'')$ is given by
\begin{align*}
a_0'' &=  1+  \,\I
\\
 b_0'' &= -\frac{2 a_0''+3}{2( a_0'' + 1)} - \I \frac{\sqrt{3 a_0'' + 4}}{2(a_0'' + 1)\sqrt{a_0''}} 
\approx -1.51564 - 0.311861 \,\I
 \\
 c_0'' &=  \frac{(1 + a_0'' (1 + b_0'')^2)^2}{a_0^{\prime\prime3} b_0^{\prime\prime2} (1 + b_0'')^2 }
 \approx  -0.367012 - 0.130244 \,\I
\end{align*}
A subtlety here is that it can be shown that the character $[\bar r ] \in \XP(S_{1,1})$ defined by this edge weight system does not lift to a $\phi$--invariant $\SL$--character $[r] \in \XX(S_{1,1})$. What is given by the formula (\ref{eqn:ExampleLLR}) is therefore, in this case, the trace of the intertwiner $\bar \Lambda_{\phi, \bar r}^q$ associated by Proposition~\ref{prop:QuantumTeichmullerIntertwinerDefined} to the puncture invariant $h_v = \E^{\frac2n \theta_v}$, where $\theta_v$ is chosen so that $\E^{\theta_v}=a_0'' b_0'' c_0''$. 

	This is the example illustrated in Figure~\ref{fig:LLRother}, for $\theta_v = \log a_0'' + \log b_0'' + \log c_0''$. Then, if we start with 
\begin{align*}
 A_0'' &= \log a_0 ''
 &
 B_0'' &= \log b_0'' 
 &
 C_0'' &= \log c_0'' 
  \\
 & \approx 0.346574 + 0.785398 \,\I
 &
 & \approx 0.436571 - 2.93866 \,\I
 &
 & \approx -0.943051 - 2.80058 \,\I,
\end{align*}
the same computations as above now lead us to
\begin{align*}
  A_3'' &\approx 0.346574 + 13.3518 \,\I
 &
 B_3'' &\approx 0.436571 + 3.34452 \,\I
 &
 C_3'' &\approx -0.943051 - 21.6501 \,\I,
\end{align*}
and the correction factors are therefore
\begin{align*}
 \widehat l_0'' &= -2
 &
 \widehat m_0'' &= -1
 &
 \widehat n_0'' &= + 3
\end{align*}

At this point, it may be hard to understand why the algebraic data of this example is so different from that of Examples 1 and 2 that, as observed in \S \ref{sect:Experimental}, it leads to fundamentally distinct asymptotic  behavior for the trace  of the associated intertwiner $\Lambda_{\phi, \bar r}^q$. The cancellations that occur in this case, and not in the two other examples, will be explained in \cite{BWY3}, where a specific criterion is developed. In the particular case of the diffeomorphism $\phi = LLR$,  this criterion is based on the parity of $\frac{\widehat l_0 - \widehat m_0 + \widehat n_0}2$.

\section{The case of a general surface}
\label{sect:ComputeIntertwinerGeneralSurf}

In addition to the fact that they are used in \cite{BWY2, BWY3}, our computations in the case of the one-puncture torus were meant as a demonstration project to show the effectiveness of the methods of \S \ref{sect:ComputeIntertwinerWithCheFock} to compute the intertwiner $\Lambda_{\phi, r}^q$ of Proposition~\ref{prop:KauffmanIntertwinerDefined}. The same methods can be implemented for a general surface, but at the expense of greatly increased computational complexity. We briefly describe how. 

Our data is an orientation-preserving diffeomorphism $\phi \colon S \to S$ of a surface $S$ with genus $g$ and $p$ punctures. We are also given a $\phi$--invariant character $[r]\in \XX(S)$  associated, as in \S \ref{subsect:SweepInvariantChar}, to an edge weight system for an ideal triangulation sweep, and $\phi$--invariant puncture weights as in Theorem~\ref{thm:ClassicalShadow}. We want to compute the  intertwiner $\Lambda_{\phi, r}^q$ associated by Proposition~\ref{prop:KauffmanIntertwinerDefined} to this data, by relying on Theorem~\ref{thm:CheFockAndKauffmanIntertwiners} and computing instead the intertwiner $\bar \Lambda_{\phi, r}^q$ associated  to suitably compatible data by Proposition~\ref{prop:QuantumTeichmullerIntertwinerDefined}. 

Our strategy requires some preparation for each surface $S$, after which it is easily implemented for each diffeomorphism $\phi \colon S \to S$ and each ideal triangulation sweep $\tau = \tau^{(0)}$, $ \tau^{(1)}$, \dots, $ \tau^{(k_0-1)}$, $ \tau^{(k_0)}=\phi(\tau)$. We divide this implementation into several steps.

For this, we fix a surface $S$ with genus $g$ and $p$ punctures. In particular, every ideal triangulation of $S$ has $e= 6g+3p-6$ edges. 

\medskip
\noindent\textsc{Step 1.} The first observation is that, up to orientation-preserving diffeomorphism, the surface $S$ admits only finitely many ideal triangulations $\tau$, which consequently gives us finitely many Chekhov-Fock algebras $\T_\tau^q(S)$. For each such ideal triangulation $\tau$, one needs a practical implementation of Proposition~\ref{prop:RepsCheFock}, associating a representation $\rho \colon \T_\tau^q(S) \to \End(\C^{n(3g+p-3)})$ to the data of  edge weights $a_i \in \C^* $  and compatible puncture weights $h_v \in \C^*$ for $\tau$. 

In practice, this requires the effective block diagonalization over the integers of the antisymmetric bilinear form $\Z^e \times \Z^e \to \Z$ defined by the matrix whose entries are the coefficients $\sigma_{ij} \in \{ -2, -1, 0, 1, 2\}$ occurring in the $q$--commutativity relations $X_iX_j = q^{2ij} X_jX_i$ defining the Chekhov-Fock algebra $\T_\tau^q$. Proposition~5 of \cite{BonLiu} predicts that there will be 
 $g $ blocks 
$\left(
\begin{smallmatrix}
0&-2\\2&0
\end{smallmatrix}
 \right)$,
 $ 2g+p-3$ blocks 
$\left(
\begin{smallmatrix}
0&-1\\1&0
\end{smallmatrix}
 \right)$
 and
 $p $ blocks 
$\left(
\begin{smallmatrix}
0
\end{smallmatrix}
 \right)$.
 In practice, this block diagonalization is easily obtained by the usual Gram-Schmidt methods. 
 
 Once the $q$--commutativity relations are block diagonalized, the practical implementation of Proposition~\ref{prop:RepsCheFock} involves a choice of $n$--roots $a_i^{\frac1n}$ for the edge weights $a_i \in \C^*$ constrained by the puncture weights $h_v$ in the sense that $h_v = a_{i_1}^{\frac1n} a_{i_2}^{\frac1n} \dots a_{i_k}^{\frac1n}$ whenever the puncture $v$ is adjacent to the edges $\gamma_{i_1}$, $\gamma_{i_2}$, \dots, $\gamma_{i_k}$. The construction associates to the $n$--roots $a_i^{\frac1n}$ a ``standard'' representation $\rho \colon \T_\tau^q(S) \to \End(\C^{n(3g+p-3)})$. 
 
 In our discussion of the case of the one-puncture torus, this step essentially corresponds to Proposition~\ref{prop:StandardRep}. In the general case as in this special case, the image $\rho(X_i) \in  \End(\C^{n(3g+p-3)})$ of each generator $X_i$ of $\T_\tau^q(S)$ is obtained by multiplying a fixed matrix (depending only on $q$ and on the combinatorics of $\tau$) by a certain monomial in the $n$--roots $a_i^{\frac1n}$. 
 
  \medskip
\noindent\textsc{Step 2.} The previous step associates a representation $\rho \colon \T_\tau^q(S) \to \End(\C^{n(3g+p-3)})$ to each suitable choice of $n$--roots $a_i^{\frac1n}$. Another choice of $n$--roots will be of the form $q^{m_i}a_i^{\frac1n}$ for  integers $m_i \in \Z$, and will define a representation $\rho' \colon \T_\tau^q(S) \to \End(\C^{n(3g+p-3)})$ that is isomorphic to $\rho$ by an isomorphism $L_{m_1m_2\dots m_e} \colon \C^{n(3g+p-3)} \to \C^{n(3g+p-3)}$. Namely, 
$$\rho(W) = L_{m_1m_2\dots m_e} \circ \rho'(W) \circ L_{m_1m_2\dots m_e}^{-1} \in\End(\C^{n(3g+p-3)})$$ 
for every $W \in \T_\tau^q(S)$. 

Once we are given the block diagonalization of the previous step, the isomorphism $ L_{m_1m_2\dots m_e} $ can be easily computed (up to scalar multiplication) and looks very much like the twist isomorphism $T_{l_0m_0n_0}$ of \S \ref{subsect:TwistIntertwiner}. In particular, when $S$ is the one-puncture torus, the isomorphism $ L_{m_1 m_2m_3} $ can be chosen equal to $ T_{m_1' m_2'm_3'} $ with $m_i = 4 m_i' \negthickspace \mod n$. 

Choose a normalization  such that $\left| \det  L_{m_1m_2\dots m_e} \right|=1$. 
 
 \medskip
\noindent\textsc{Step 3.} For a given ideal triangulation $\tau$, there are finitely many ideal triangulations $\tau'$ that are obtained from $\tau$ by a diagonal exchange or an edge relabelling. We saw in \S \ref{subsect:SweepInvariantChar} that, at least under a genericity condition guaranteeing that we do not try to divide by 0 in the case of a diagonal exchange, an edge weight system $a \in (\C^*\big)^e$ for $\tau$ determines an edge weight  system $a' \in (\C^*\big)^e$ for $\tau'$ that is associated to the same character in $\XP(S)$. The gist of Proposition~\ref{prop:CheFockRepsCompatibleIdealTriangChange}, namely Lemma~27 of \cite{BonLiu}, asserts the following: If $\rho \colon \T_\tau^q(S) \to \End(\C^{n(3g+p-3)})$ and $\rho' \colon \T_{\tau'}^q(S) \to \End(\C^{n(3g+p-3)})$ are two standard representations respectively associated to systems of $n$--roots $a_i^{\frac1n}$ and $a_i^{\prime\frac1n}$ compatible with the same puncture invariants $h_v = h_v'$, then, for the Chekhov-Fock coordinate change $ \Phi_{\tau\tau'}^q  \colon \T_{\tau'}^q(S) \to \T_\tau^q(S)$, the representation $\rho \circ \Phi_{\tau\tau'}^q \colon \T_{\tau'}^q(S) \to \End(\C^{n(3g+p-3)})$ makes sense and is isomorphic to $\rho'$ by an isomorphism $\Lambda \colon \C^{n(3g+p-3)} \to \C^{n(3g+p-3)}$. More precisely, 
$$
\rho \circ \Phi_{\tau\tau'}^q (W) = \Lambda \circ \rho'(W') \circ \Lambda^{-1}  \in\End(\C^{n(3g+p-3)}) 
$$
for every $W' \in \T_{\tau'}^q(S)$. Note that, once we know it exists, the isomorphism $\Lambda$ can be explicitly determined by solving the finitely many linear equations $\Lambda \circ \big(\rho \circ \Phi_{\tau\tau'}^q (X_i') \big)= \Lambda \circ \rho' (X_i')$ given by the generators $X_i'$ of $ \T_{\tau'}^q(S)$. 

In this Step 3, we first select a rule that, given a choice of $n$--roots $a_i^{\frac1n}$ for an edge weight system $a \in \big( \C^* \big)^e$ for $\tau$, specifies $n$--roots $a_i^{\prime\frac1n}$ for the associated edge weight system $a' \in \big( \C^* \big)^e$ for $\tau'$, in such a way that these $n$--roots define the same puncture weights $h_v = h_v'$ in the sense that
$$
a_{i_1}^{\frac1n} a_{i_2}^{\frac1n} \dots a_{i_k}^{\frac1n}
=
a_{i_1'}^{\prime\frac1n} a_{i_2'}^{\prime\frac1n} \dots a_{i_{k'}'}^{\prime\frac1n}
\ (=h_v = h_v')
$$
whenever the puncture $v$ is adjacent to the edges $\gamma_{i_1}$, $ \gamma_{i_2}$,  \dots, $\gamma_{i_k} $ of $\tau $ and to the edges $\gamma_{i_1'}'$, $ \gamma_{i_2'}'$, \dots, $\gamma_{i_{k'}'}' $ of $\tau'$. This rule may not necessarily be continuous. 

Then, once the rule is fixed, $\rho \colon \T_\tau^q(S) \to \End(\C^{n(3g+p-3)})$ is the standard representation associated to the $n$--roots  $a_i^{\frac1n}$, and $\rho' \colon \T_{\tau'}^q(S) \to \End(\C^{n(3g+p-3)})$ is the standard representation associated to the $n$--roots $a_i^{\prime\frac1n}$, we select an isomorphism $\Lambda \colon \C^{n(3g+p-3)} \to \C^{n(3g+p-3)}$ such that $\rho \circ \Phi_{\tau\tau'}^q (W) = \Lambda \circ \rho'(W') \circ \Lambda^{-1}  $ for every $W' \in \T_{\tau'}^q(S)$ and normalized so that $\left| \det \Lambda \right|$. This $\Lambda$ will depend on the $n$--roots  $a_i^{\prime\frac1n}$.

Up to diffeomorphism, there are only finitely many pairs $(\tau, \tau')$ where the ideal triangulation $\tau'$ is obtained from $\tau$ by a diagonal exchange or by an edge relabelling. This Step~3 therefore provides finitely many linear isomorphisms $\Lambda \colon \C^{n(3g+p-3)} \to \C^{n(3g+p-3)}$, considered as functions of the $n$--roots $a_i^{\frac1n}$. 

In the case of the one-puncture torus, this Step~3 corresponds to Lemma~\ref{lem:LeftRightIntertwiners}. 

  \medskip
\noindent\textsc{Step 4.} The first three steps were preparatory, and depended only on the surface $S$. We are now ready to tackle our general goal, with the data of an orientation-preserving diffeomorphism $\phi \colon S \to S$ and a $\phi$--invariant character $[r]\in \XX(S)$  associated to a periodic edge weight system $a^{(0)}$, $a^{(1)}$, \dots,  $a^{(k_0)} \in \left( \C^* \right)^{e}$  for the ideal triangulation sweep $\tau = \tau^{(0)}$, $ \tau^{(1)}$, \dots, $ \tau^{(k_0-1)}$, $ \tau^{(k_0)}=\phi(\tau)$, and $\phi$--invariant puncture weights $h_v$ such that 
$$
h_v^n = a_{i_1}^{(0)} a_{i_2}^{(0)} \dots a_{i_l}^{(0)} 
$$
if the puncture $v$ is adjacent to the edges $\gamma_{i_1}$, $ \gamma_{i_2}$,  \dots, $\gamma_{i_k} $ of $\tau $. 

With this data, start with an arbitrary choice of $n$--roots $\left( a_i^{(0)} \right)^{\frac1n}$ for the edge weights of $\tau = \tau^{(0)}$. Then, for every $k=1$, $2$, \dots, $k_0$, the rule selected in Step~3 determines preferred $n$--roots $\left( a_i^{(k)} \right)^{\frac1n}$ for the edge weights of $\tau^{(k)}$. These $n$--roots specify a standard representation $\rho_k \colon \T_{\tau^{(k)}}^q(S) \to \End( \C^{n(3g+p-3)} )$ as in Step~1, and Step~3 provides a linear isomorphism $\Lambda \colon \C^{n(3g+p-3)} \to \C^{n(3g+p-3)}$ between the representations $\rho_{k-1} \circ \Phi^q_{\tau^{(k-1)} \tau^{(k)}}$ and $\rho_k \colon \T_{\tau^{(k)}}^q(S) \to \End( \C^{n(3g+p-3)} )$. 

Since the edge weight system is periodic, each edge weight $ a_i^{(k_0)} $ is equal to $ a_i^{(0)} $. However, because the $n$--roots were selected by repeated application of Step~3, we only know that there exists integers $m_i \in \Z$ such that $\left( a_i^{(0)} \right)^{\frac1n} = q^{m_i} \left( a_i^{(k_0)} \right)^{\frac1n}$. Step~2 associates an isomorphism $L_{m_1m_2\dots m_e} \colon \C^{n(3g+p-3)} \to \C^{n(3g+p-3)}$ to these integers. 

Finally, the argument that we used in the proof of Proposition~\ref{prop:IntertwinerCompElementaryIntertwiner} shows that the isomorphism $\Lambda_{\phi, r}^q$ associated to the data by Proposition~\ref{prop:KauffmanIntertwinerDefined} is equal to
$$
\Lambda_{\phi, r}^q = \Lambda_1 \circ \Lambda_2 \circ \dots \circ \Lambda_{k_0} \circ  L_{m_1m_2\dots m_e} 
$$
which concludes our computation.

\medskip
In its full generality, the method above involves the computation of a lot of data. In practice, it is often possible to reduce  this data by using various combinatorial arguments. 

For instance, in the case of the one-puncture torus, there exist two (labelled) ideal triangulations up to orientation-preserving diffeomorphism, but we were able to restrict attention to one type. Also, by combining each diagonal exchange with an edge relabelling, we only had to consider two moves, the ``left'' move and the ``right'' move. 

A very similar simplification can be used for the four-puncture sphere, if we restrict attention to ideal triangulations of tetrahedral type, where the edges can be labelled as $e$, $f$, $g$, $e'$, $f'$, $g'$ in such a way that $e$, $f$, $g$ occur clockwise around a face and $e'$, $f'$, $g'$ are respectively disjoint from $e$, $f$, $g$. (In this case, there are many more possible ideal triangulations.) In this setup, we can arrange that the diagonal exchanges occur along pairs of disjoint edges which, combined with suitable edge relabelling, again leads to the consideration of only a ``left'' and ``right'' move.


\bibliographystyle{amsalpha}
\bibliography{BWY1}
 
\end{document}